\documentclass[a4paper,10pt,reqno, english]{amsart}  
\usepackage[utf8]{inputenc}
\usepackage[T1]{fontenc}
\usepackage{amsmath,amsthm}
\usepackage{amsfonts,amssymb,enumerate}
\usepackage{url,paralist}
\usepackage{mathtools}  
\usepackage{multicol}
\usepackage[colorlinks=true,urlcolor=blue,linkcolor=red,citecolor=magenta]{hyperref}
\usepackage{enumerate}
\usepackage{anysize}
\usepackage[arrow,curve,matrix,tips,2cell]{xy}  
  \SelectTips{eu}{10} \UseTips
  \UseAllTwocells
\usepackage{lscape}
\graphicspath{{fig/}}

\theoremstyle{plain}
\newtheorem{theorem}{Theorem}[section]
\newtheorem{lemma}[theorem]{Lemma}

\newtheorem{corollary}[theorem]{Corollary}

\newtheorem{conjecture}[theorem]{Conjecture}
\newtheorem*{theorem*}{Theorem}

\newtheorem*{claim*}{Claim}
\newtheorem*{lemma*}{Lemma}

\theoremstyle{definition}

\newtheorem{examples}[theorem]{Examples}
\newtheorem{remark}[theorem]{Remark}

\newtheorem{problem}[theorem]{Problem}

\newcommand{\dictionary}[1]{{#1}$\,^{\textcolor{blue}{\mathrm{dict}}}$} 
\newcommand{\dist}{\mathrm{dist}}

\newcommand{\R}{\mathbb{R}}

\newcommand{\N}{\mathbb{N}}
\newcommand{\Z}{\mathbb{Z}}

\newcommand{\F}{\mathbb{F}}
\newcommand{\B}{\mathrm{B}}
\newcommand{\E}{\mathrm{E}}
\DeclareMathOperator{\im}{im}

\newcommand\Sym{\mathfrak S} 
 \newcommand\oo{\mathfrak o} 
    \newcommand\sgn{{\rm sgn\,}}
    \newcommand\card{{\rm card\,}}
    \newcommand\xx{\mathbf{x}} 
 \newcommand\ee{\mathbf{e}}
   
\newcommand\conv{{\rm conv}}

\newcommand{\conn}{\operatorname{conn}}

\newcommand{\id}{\operatorname{id}}
\newcommand{\ind}{\operatorname{index}}

\newcommand{\res}{\operatorname{res}}
\newcommand{\colim}{\operatorname{colim}}
\newcommand{\trf}{\operatorname{tr}} 
\newcommand{\relint}{\operatorname{relint}} 
\newcommand\sk{\operatorname{sk}}

\begin{document}

\title[Beyond the Borsuk--Ulam theorem]{Beyond the Borsuk--Ulam theorem:\\ The topological Tverberg story}



\author[Blagojevi\'c]{Pavle V. M. Blagojevi\'{c}} 
\thanks{The research by Pavle V. M. Blagojevi\'{c} leading to these results has
        received funding from DFG via Collaborative Research Center TRR~109 ``Discretization in Geometry and Dynamics.''
        Also supported by the grant ON 174008 of the Serbian Ministry of Education and Science.}
\address{Inst. Math., FU Berlin, Arnimallee 2, 14195 Berlin, Germany\hfill\break%
\mbox{\hspace{4mm}}Mat. Institut SANU, Knez Mihailova 36, 11001 Beograd, Serbia}
\email{blagojevic@math.fu-berlin.de} 
\author[Ziegler]{G\"unter M. Ziegler} 
\thanks{The research by G\"unter M. Ziegler received funding from DFG via the Research Training Group “Methods for Discrete Structures” and the Collaborative Research Center TRR~109 ``Discretization in Geometry and Dynamics.''}  
\address{Inst. Math., FU Berlin, Arnimallee 2, 14195 Berlin, Germany} 
\email{ziegler@math.fu-berlin.de}

\date{May 21, 2017}

\dedicatory{Dedicated to the memory of Ji\v{r}\'{\i} Matou\v{s}ek}

\maketitle 

\begin{abstract}
	Bárány's ``topological Tverberg conjecture'' from 1976 states that
	any continuous map of an $N$-simplex $\Delta_N$ to~$\R^d$, for $N\ge(d+1)(r-1)$, 
	maps points from $r$ disjoint faces in~$\Delta_N$ to the same point in $\R^d$.
	The proof of this result for the case when $r$ is a prime, as well as some colored version of 
	the same result, using the results of Borsuk--Ulam and Dold on the non-existence of equivariant maps
	between spaces with a free group action, were main topics of Matou\v{s}ek's 2003 book
	``Using the Borsuk--Ulam theorem.''
	
	In this paper we show how advanced equivariant topology methods allow one to go beyond the prime case of the
	topological Tverberg conjecture.
	
	First we explain in detail how equivariant cohomology tools 
	(employing the Borel construction, comparison of Serre spectral sequences, Fadell--Husseini index, etc.) 
	can be used to prove the
	topological Tverberg conjecture whenever $r$ is a prime power.
	Our presentation includes a number of improved proofs as well as
	new results, such as a complete determination of the 
	Fadell--Husseini index of chessboard complexes in the prime case.
	
	Then we introduce the ``constraint method,'' which applied to suitable ``unavoidable complexes'' 
	yields a great variety of variations and corollaries to the topological
	Tverberg theorem, such as the
	``colored'' and the ``dimension-restricted'' (Van Kampen--Flores type) versions.
	
	Both parts have provided crucial components to the
	recent spectacular counter-examples in high dimensions for the case when $r$ is not a prime power.	
\end{abstract}


 
\section{Introduction}
\label{sec : Introduction}

Ji\v{r}\'{\i}  Matou\v{s}ek's 2003 book ``\emph{Using the Borsuk--Ulam Theorem.\ Lectures on Topological Methods in Combinatorics and Geometry}'' \cite{Matousek2008} is an inspiring introduction to the use of equivariant methods in Discrete Geometry.
Its main tool is the Borsuk--Ulam theorem, and its generalization by Albrecht Dold, which says that there is no equivariant map from an $n$-connected space to an $n$-dimensional finite complex that is equivariant with respect to a non-trivial finite group acting freely.
One of the main applications of this technology in Matou\v{s}ek's book was a proof for Bárány's ``topological Tverberg conjecture'' on $r$-fold intersections in the case when $r$ is a prime, originally due to Imre B\'ar\'any, Senya Shlosman and Andr\'as Sz\H{u}cs \cite{Barany1981}.
This conjecture claimed that for any continuous map $f:\Delta_N\rightarrow\R^d$, when $N\ge(d+1)(r-1)$, there are $r$ points in disjoint faces of the simplex~$\Delta_N$ that $f$ maps to the same point in $\R^d$.

The topological Tverberg conjecture was extended to the case when $r$ is a prime power by Murad \"{O}zaydin in an unpublished paper from 1987 \cite{Oezaydin1987}.
This cannot, however, be achieved via the Dold theorem, since in the prime power case the group actions one could use on the codomain are not free. 
So more advanced methods are needed, such as the Serre spectral sequence for the Borel construction and the Fadell--Husseini index. 
In this paper we present the area about and around the topological Tverberg conjecture, with complete proofs for all of the results, which include the prime power case of the topological Tverberg conjecture. 

\"{O}zaydin in 1987 not only proved the topological Tverberg theorem for prime power $r$, but he also showed, using equivariant obstruction theory, that the approach fails when $r$ is not a prime power: In this case the equivariant map one looks for does exist.

In a spectacular recent development, Isaac Mabillard and Uli Wagner \cite{Mabillard2014} \cite{Mabillard2015} have developed an $r$-fold version of the classical ``Whitney trick'' (cf. \cite{Whitney1944}), which yields the failure of the generalized Van Kampen--Flores theorem when $r\ge6$ is not a prime power.
Then Florian Frick observed that this indeed implies the existence of counterexamples to the topological Tverberg conjecture \cite{Frick2015} \cite{Blagojevic2015} by
a lemma of Gromov \cite[p.\,445]{Gromov2010} that is an instance of the constraint method of Blagojevi\'c, Frick and Ziegler \cite[Lemma\,4.1(iii) and Lemma\,4.2]{Blagojevic2014}.
(See \cite{Barany2016} for a popular rendition of the story.)

The Tverberg theorem from 1966 \cite{Tverberg1966} and its conjectured extension to continuous maps (the topological Tverberg conjecture) have seen a great number of variations and extensions, among them ``colored'' variants as well as versions with restricted dimensions (known as generalized Van Kampen--Flores theorems).
Although many of these were first obtained as independent results, sometimes with very similar proof patterns, our presentation shows that there are many easy implications between these results, using in particular the ``constraint method'' applied to ``unavoidable complexes,'' as developed by the present authors with Florian Frick \cite{Blagojevic2014}.
(Mikhail Gromov \cite[p.\,445]{Gromov2010} had sketched one particular instance: The topological Tverberg theorem for maps to $\R^{n+1}$ implies a generalized Van Kampen--Flores theorem for maps to $\R^n$.)
Thus we can summarize the implications in the following scheme, which shows that all further main results follow from two sources, the topological Tverberg theorem for prime powers,  and the optimal colored Tverberg theorem of the present authors with Benjamin Matschke \cite{Blagojevic2009}, which up to now even for affine maps is available only for the prime case:

{\footnotesize
\[
\xymatrix@!C= 4em{
  &\text{Topological Tverberg for } p^n \ar[dl]\ar[ddl]|!{[dl];[dl]}\hole \ar[drr]\ar[dddl]|!{[dl];[dl]}\hole|!{[ddl];[ddl]}\hole    &  &  &\text{Optimal colored Tverberg for }p\ar[dr]\ar[dl]& &\\
\text{Generalized Van Kampen--Flores for }p^n &  & &\text{Topological Tverberg for } p & &\text{B\'ar\'any--Larman for }p-1\\
\text{Colored Van Kampen--Flores for }p^n\ar[drrrr]\ar[d]   & & & & \\
\text{Weak colored Tverberg for }p^n\ar[drrrr] & & & &\text{Vre\'cica--\v{Z}ivaljevi\'c colored Tverberg for }p^n\text{ type B}\ar[d]\\
& & & &\text{\v{Z}ivaljevi\'c--Vre\'cica colored Tverberg for }p^n\text{ type A}
}
\]
}

\smallskip

Our journey in this paper starts with Radon's 1921 theorem and its topological version, in Section~\ref{sec : The Begining}.
Here the Borsuk--Ulam theorem is all that's needed.
In Section~\ref{sec : The Tverberg theorem} we state the topological Tverberg conjecture and first prove it in the prime case (with a proof that is close to the original argument by Bárány, Shlosman and Sz{\H{u}}cs), and then for prime powers---this is where we go ``beyond the Borsuk--Ulam theorem.''
Implications and corollaries of the topological Tverberg theorem are developed in Section~\ref{sec : Corollaries of the topological Tverberg theorem}---so that's where we put constraints, and ``add color.''
In Section~\ref{sec : Counterexamples to the topological Tverberg conjecture} we get to the counterexamples.
And finally in Section~\ref{sec : The optimal colored Tverberg theorem} we discuss the ``optimal colored Tverberg conjecture,''
which is a considerable strengthening of Tverberg's theorem, but up to now has been proven only in the prime case.

A summary of the main topological concepts and tools used in this paper is given at the end in the form of a dictionary, where a reference to the dictionary in the text is indicated by \dictionary{concept}.

\bigskip
\noindent
{\em Acknowledgements}. We are grateful to Alexander Engstr\"om and Florian Frick for excellent observations on drafts of this paper and many useful comments.
We want to express our gratitude to Peter Landweber for his continuous help and support in improving this manuscript.

\section{The Beginning}
\label{sec : The Begining}

\subsection{Radon's theorem}

One of the first cornerstone results of convex geometry is a 1921 theorem of Johann Radon about overlapping convex hulls of points in a Euclidean space.  

\medskip
Let $\R^d$ be a $d$-dimensional Euclidean space. 
Let $\xx_1,\ldots,\xx_m$ be points in $\R^d$ and let $\alpha_1,\ldots,\alpha_m$ be non-negative real numbers that sum up to $1$, that is, $\alpha_1\geq 0,\ldots,\alpha_m\geq 0$ and $\alpha_1+\cdots+\alpha_m=1$.
The \emph{convex combination} of the points $\xx_1,\ldots,\xx_m$ determined by the scalars $\alpha_1,\ldots,\alpha_m$ is the following point in $\R^d$:
\[
\xx=\alpha_1\xx_1+\cdots+\alpha_m\xx_m .
\]
For a subset $C$ of $\R^d$ we define the \emph{convex hull} of $C$, denoted by $\conv(C)$, to be the set of all convex combinations of finitely many points in $C$:
\[
\conv(C):=\{\alpha_1\xx_1+\cdots+\alpha_m\xx_m : m\in\N,\,\xx_i\in C,\, \alpha_i\in\R_{\geq 0},\, \alpha_1+\cdots+\alpha_m=1\}.
\]
Now Radon's theorem can be stated as follows and proved using elementary linear algebra.

\begin{theorem}[Radon's theorem, point configuration version \cite{Radon1921}]
Let $\R^d$ be a $d$-dimensional Euclidean space, and let $X\subseteq \R^d$ be a subset with (at least) $d+2$ elements. 
Then there are disjoint subsets $P$ and $N$ of $X$ with the property that
\[
\conv(P)\cap\conv(N)\neq\emptyset.
\]
\end{theorem}

\begin{proof}
Let $X=\{\xx_1,\ldots,\xx_{d+2}\}\subset\R^d$. 
The homogeneous system of $d+1$ linear equations in $d+2$ variables
\[
\alpha_1\xx_1+\cdots+\alpha_{d+2}\xx_{d+2}=0,\qquad \alpha_1+\cdots+\alpha_{d+2}=0
\]
has a non-trivial solution, say $\alpha_1=a_1,\dots,\alpha_{d+2}=a_{d+2}$.
Denote
\[
P:=\{i : a_i>0\}\qquad\text{and}\qquad N:=\{i : a_i\leq 0\}.
\]
Then $P\cap N=\emptyset$ while $P\neq\emptyset$ and $N\neq\emptyset$, and
\begin{eqnarray*}
a_1\xx_1+\cdots+a_{d+2}\xx_{d+2} = 0   & \Rightarrow &   \sum_{i\in P} a_i\xx_i= \sum_{i\in N} -a_i\xx_i,\\
a_1+\cdots+a_{d+2} = 0                 & \Rightarrow &   \sum_{i\in P} a_i= \sum_{i\in N} -a_i=:A,
\end{eqnarray*} 
where $A>0$.
Consequently, the following point is in the intersection of convex hulls of $P$ and $N$:
\[
\xx:=\sum_{i\in P} \frac{a_i}{A}\xx_i= \sum_{i\in N} \frac{-a_i}{A}\xx_i\in\conv(\{\xx_i : i\in P\})\cap\conv(\{\xx_i : i\in N\}).\vspace{-5pt}
\]
\end{proof}

\begin{figure}
\centering
\includegraphics[scale=0.8]{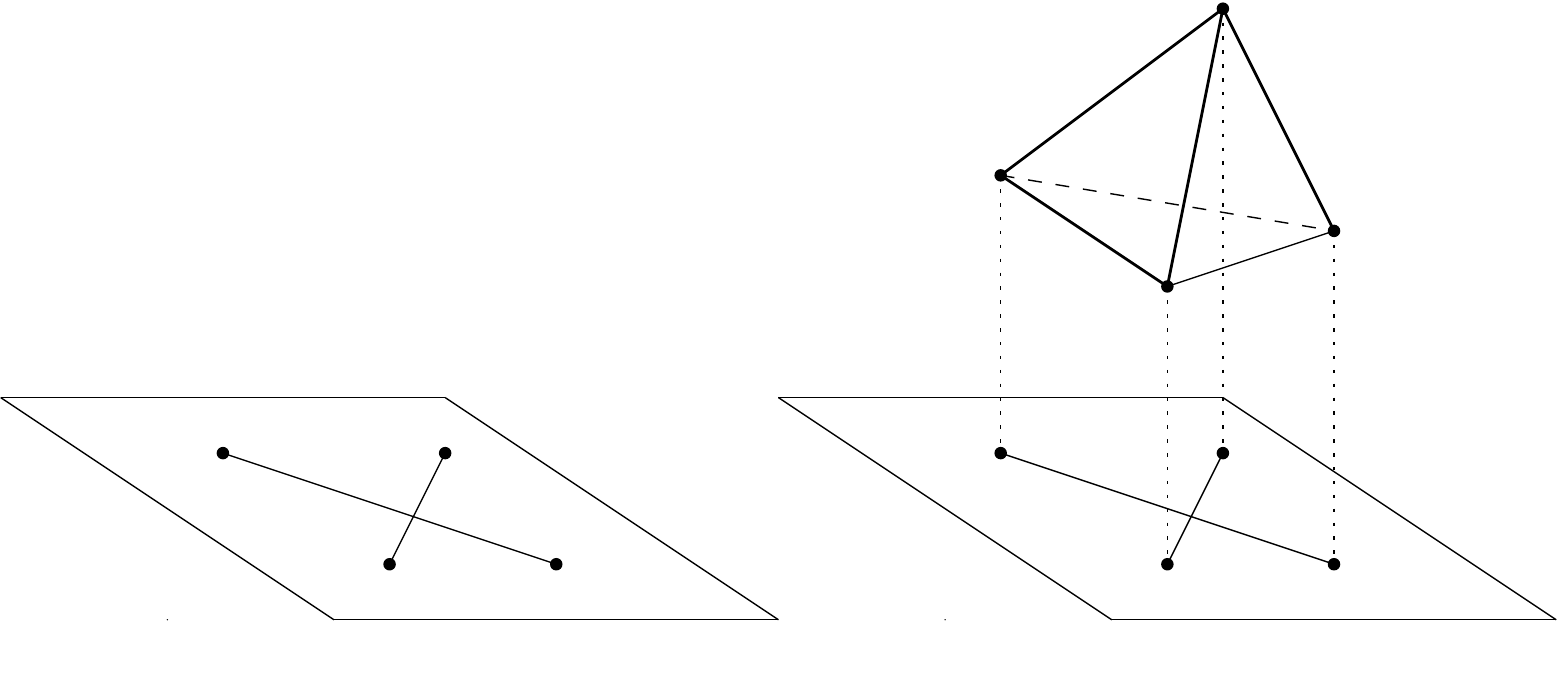}
\caption{\small Illustration of Radon's theorem in the plane for both versions of the theorem.}
\end{figure}
 
In order to reformulate Radon's theorem we recall the notion of an affine map.
A map $f :  D\rightarrow\R^d$ defined on a subset $D\subseteq\R^k$ 
is \emph{affine} if for every $m\in\N$, $\xx_1,\ldots,\xx_m\in D$, 
and $\alpha_1,\ldots,\alpha_m\in\R$ with 
$\alpha_1+\cdots+\alpha_m=1$ and $\alpha_1\xx_1+\cdots+\alpha_m\xx_m\in D$, we have
\[
f(\alpha_1\xx_1+\cdots+\alpha_m\xx_m)=\alpha_1 f(\xx_1)+\cdots+\alpha_m f(\xx_m).
\]
Here and in the following let $\Delta_k:=\conv\{\ee_1,\ldots,\ee_{k+1}\}$ be 
the \emph{standard $k$-dimensional
simplex}: This simplex given as the convex hull of the standard basis of $\R^{k+1}$
has the disadvantage of not being full-dimensional in $\R^k$, but it has the extra advantage
of being obviously symmetric (with symmetry given by permutation of coordinates).
With this, Radon's theorem can be restated as follows.

\begin{theorem}[Radon's theorem, affine map version]
Let $f : \Delta_{d+1}\rightarrow \R^d$ be an affine map. 
Then there are disjoint faces $\sigma_1$ and $\sigma_2$ of the $(d+1)$-simplex
$\Delta_{d+1}$ with the property that
\[
f(\sigma_1)\cap f(\sigma_2)\neq\emptyset.
\]
\end{theorem}

With this version of Radon's theorem at hand, it is natural to ask: 
\emph{Would Radon's theorem still hold if instead of an affine map we consider an arbitrary continuous map $f : \Delta_{d+1} \rightarrow \R^d$?}

\subsection{The topological Radon theorem}

The question we have just asked was answered in 1979 by Ervin Bajm\'oczy and Imre B\'ar\'any \cite{Bajmoczy1979}, using the Borsuk--Ulam theorem.

\begin{theorem}[Topological Radon theorem]
\label{th : Topological Radon theorem} 
Let $f : \Delta_{d+1}\rightarrow \R^d$ be any continuous map.
Then there are two disjoint faces $\sigma_1$ and $\sigma_2$ of $\Delta_{d+1}$ 
whose images under $f$ intersect,
\[
f(\sigma_1)\cap f(\sigma_2)\neq\emptyset.
\]
\end{theorem}

\begin{proof}
Let $\Delta_{d+1}=\conv\{\ee_1,\ldots,\ee_{d+2}\}$ be the standard simplex.
Consider the subcomplex $X$ of the polyhedral complex $\Delta_{d+1}\times\Delta_{d+1}$ given by
\[
X:=
\{
(x_1,x_2)\in\Delta_{d+1}\times\Delta_{d+1} :
\text{there are faces }
\sigma_1 ,\sigma_2 \subset\Delta_{d+1}\text{ such that }\sigma_1 \cap \sigma_2 =\emptyset,\ x_1\in\sigma_1,\ x_2\in \sigma_2\}.
\]
The group $\Z/2=\langle\varepsilon\rangle$ acts freely on $X$ by $\varepsilon\cdot(x_1,x_2)=(x_2,x_1)$.

Let us assume that the theorem does not hold.
Then there exists a continuous map $f : \Delta_{d+1}\rightarrow \R^d$ such that 
$f(x_1)\neq f(x_2)$ for all $(x_1,x_2)\in X$. 
Consequently the map $g :  X \rightarrow S^{d-1}$ given by 
\[
g(x_1,x_2):=\frac{f(x_1)-f(x_2)}{\|f(x_1)-f(x_2)\|},
\]
is continuous and $\Z/2$-equivariant, where the action on $S^{d-1}=S(\R^d)$, the unit sphere in $\R^d$, is the standard antipodal action.

Next we define a continuous $\Z/2$-equivariant map from a $d$-sphere to $X$.
For this we do not use the standard $d$-sphere, but 
the unit sphere $S(W_{d+2})$ in the
hyperplane $W_{d+2}:=\{(a_1,\ldots,a_{d+2})\in\R^{d+2} : a_1+\cdots+a_{d+2}=0\}\subset\R^{d+2}$,
that is,
\[
S(W_{d+2})\ =\ 
\{
(a_1,\ldots,a_{d+2})\in\R^{d+2} :
 a_1+\cdots+a_{d+2}=0,\, a_1^2+\cdots+a_{d+2}^2=1
\}.
\]
This representation of the $d$-sphere
also has the standard antipodal $\Z/2$-action. 
The map $h :  S(W_{d+2})\rightarrow X$ is defined by
\[
h (a_1,\ldots,a_{d+2}):=
\Big(\sum_{a_i\geq 0}\frac{\,a_i}{A}\,\ee_i,
\sum_{a_i<0}\frac{-a_i}{A}\,\ee_i\Big),
\]
where $A:=\sum_{a_i > 0}a_i=-\sum_{a_i< 0}a_i>0$. 
This is easily checked to be well-defined and continuous;
the image point lies in the cell $\conv\{\ee_i:a_i>0\} \times \conv\{\ee_j:a_j<0\}$
of the complex $\Delta_{d+1}\times\Delta_{d+1}$.

The composition map $g\circ h:S(W_{d+2})\rightarrow S^{d-1}$ yields a continuous $\Z/2$-equivariant map 
from a free $d$-sphere to a free $(d-1)$-sphere 
that contradicts the \dictionary{Borsuk--Ulam theorem}.
Thus the theorem holds.
\end{proof}

\subsection{The Van Kampen--Flores theorem}
The topological Radon theorem guarantees that for every continuous map $\Delta_{d+1}\rightarrow \R^d$ there exist two pairwise disjoint faces whose $f$-images overlap.  
It is natural to ask: \emph{Is it possible to say something about the dimension of the disjoint faces whose $f$-images intersect?}
In the spirit of Poincar\'e's classification of mathematical problems \cite[Lec.\,1]{Arnold2016} this \emph{binary problem} has a quick answer \emph{no}, but if understood as an \emph{interesting problem} it has an answer:
If we are willing to spend an extra vertex/dimension, 
meaning, put the simplex $\Delta_{d+2}$ in place of $\Delta_{d+1}$, 
we get the following theorem from the 1930s of Egbert R. Van Kampen and Antonio Flores \cite{Flores1932} \cite{VanKampen1933}.

\begin{theorem}[Van Kampen--Flores theorem] 
Let $d\geq 2$ be an even integer, and let $f : \Delta_{d+2} \rightarrow \R^d$ be a 
continuous map.  
Then there are disjoint faces $\sigma_1$ and $\sigma_2$ of $\Delta_{d+2}$ of dimension at most $d/2$ whose images under $f$ intersect, 
\[
f(\sigma_1)\cap f(\sigma_2)\neq\emptyset.
\]
\end{theorem}

\begin{proof}
Let $g : \Delta_{d+2} \rightarrow \R^{d+1}$ be a continuous map defined by
\[
g(x):=\big(f(x),\dist(x,\sk_{d/2}(\Delta_{d+2}))\big)
\]
where $\sk_{d/2}(\Delta_{d+2})$ denotes the ${d}/2$-skeleton of the simplex $\Delta_{d+2}$, and $\dist(x,\sk_{d/2}(\Delta_{d+2}))$ is the distance of the point $x$ from the subcomplex $\sk_{d/2}(\Delta_{d+2})$.
Observe that if $x\in\relint\sigma$ and $\dist(x,\sk_{d/2}(\Delta_{d+2}))=0$, then the simplex $\sigma$ belongs to the subcomplex $\sk_{d/2}(\Delta_{d+2})$.

Now the topological Radon theorem can be applied to the continuous map $g : \Delta_{d+2} \rightarrow \R^{d+1}$.
It yields the existence of points $x_1\in\relint\sigma_1$ and $x_2\in\relint\sigma_2$, with $\sigma_1\cap\sigma_2=\emptyset$, such that
\[
g(x_1)=g(x_2)\qquad \Longleftrightarrow\qquad f(x_1)=f(x_2),\quad \dist(x_1,\sk_{d/2}(\Delta_{d+2}))=\dist(x_2,\sk_{d/2}(\Delta_{d+2})).
\]
If one of the simplices $\sigma_1$, or $\sigma_2$, would belong to $\sk_{d/2}(\Delta_{d+2})$, then 
\[
\dist(x_1,\sk_{d/2}(\Delta_{d+2}))=\dist(x_2,\sk_{d/2}(\Delta_{d+2}))=0,
\]
implying that both $\sigma_1$ and $\sigma_2$ belong to $\sk_{d/2}(\Delta_{d+2})$, which would concludes the proof of the theorem.

In order to prove that at least one of the faces $\sigma_1$ and $\sigma_2$ belongs to $\sk_{d/2}(\Delta_{d+2})$,
note that these are two disjoint faces of the simplex $\Delta_{d+2}$, which
has $d+3$ vertices, so by the pigeonhole principle one of them
has at most $\lfloor(d+3)/2\rfloor=d/2+1$ vertices.  
\end{proof}

The proof we have presented is an example of the \emph{constraint method} developed in \cite{Blagojevic2014}.
An important message of this proof is that the Van Kampen--Flores theorem is a corollary of the topological Radon theorem.
It is clear that we could have considered a continuous map $f$ defined only on the $d/2$-skeleton. 

\medskip
All the results we presented so far have always claimed something about intersections of the images of two disjoint faces $\sigma_1$ and $\sigma_2$, which we refer to as $2$-fold overlap, or intersection. 
\emph{What about $r$-fold overlaps, for $r>2$?}

\section{The topological Tverberg theorem}
\label{sec : The Tverberg theorem}

\subsection{The topological Tverberg conjecture}
\label{subsec : The topological Tverberg conjectrue}

In 1964, freezing in a hotel room in Manchester,
the Norwegian mathematician Helge Tverberg proved the following $r$-fold generalization of Radon's theorem \cite{Tverberg1966}.
It had been conjectured by Bryan Birch in 1954, who had established the result
in the special case of dimension $d=2$ \cite{Birch1959}.
The case $d=1$ is easy, see below. 
(See \cite{Ziegler2011} for some of the stories surrounding these discoveries.)

\begin{theorem}[Tverberg's theorem] 
Let $d\geq 1$ and $r\geq 2$ be integers, $N=(d+1)(r-1)$, and let $f : \Delta_N\rightarrow \R^d$ be an affine map.
Then there exist $r$ pairwise disjoint faces $\sigma_1,\ldots,\sigma_r$ of the simplex $\Delta_N$ whose $f$-images overlap,  
\begin{equation}
\label{eq:Tverberg property}
f(\sigma_1)\cap\cdots\cap f(\sigma_r)\neq\emptyset.
\end{equation}
\end{theorem}

Any collection of $r$ pairwise disjoint faces $\sigma_1,\ldots,\sigma_r$ of the simplex $\Delta_N$ having property \eqref{eq:Tverberg property} is called a \emph{Tverberg partition} of the map $f$. 

The dimension of the simplex in the theorem is optimal, it cannot be decreased.
To see this consider the affine map $h : \Delta_{N-1}\rightarrow \R^d$ given on the vertices of $\Delta_{N-1}=\conv\{\ee_1,\ldots,\ee_{N}\}$ by
\begin{equation}
	\label{map h}
	\ee_{i} \overset{h}{\longmapsto }  u_{\lfloor (i-1)/(r-1)\rfloor } 
\end{equation}
where $\{u_0,\ldots, u_d\}$ is an affinely independent set in $\R^d$,
e.g., $ (u_0,\ldots, u_d)=(0,\ee_1,\dots,\ee_d)$.
For each vertex of the simplex $\conv \{u_0,\ldots, u_d\}$ 
the cardinality of its preimage is $r-1$ 
\[
|h^{-1}(\{u_0\})|=\cdots= | h^{-1}(\{u_d\}) | = r-1,
\]
and so the map $h$ has no Tverberg partition.
Even more is true: Any affine map $h : \Delta_{N-1}\rightarrow \R^d$ that is in general position cannot have a Tverberg partition.

As in the case of Radon's theorem it is natural to ask:
\emph{Would the Tverberg theorem still hold if instead of an affine map 
$f : \Delta_{N}\rightarrow \R^d$ we would consider an arbitrary continuous map?} 
This was first asked by B\'ar\'any in a 1976 letter to Tverberg.
In May of 1978 Tverberg posed the question in Oberwolfach, stating it for a general 
$N$-polytope in place of the $N$-simplex, see \cite{Gruber1979}.
(The problem for a general $N$-polytope can be reduced to the case of
the $N$-simplex by a theorem of Grünbaum: 
Every $N$-polytope as a cell complex is a refinement of the $N$-simplex \cite[p.\,200]{Grunbaum2003}.)
Thus, the topological Tverberg conjecture started its life in the late 1970s.

\begin{conjecture}[Topological Tverberg conjecture] 
Let $d\geq 1$ and $r\geq 2$ be integers, $N=(d+1)(r-1)$, and let $f : \Delta_N\rightarrow \R^d$ be a continuous map.
Then there exist $r$ pairwise disjoint faces $\sigma_1,\ldots,\sigma_r$ of the simplex $\Delta_N$ whose $f$-images overlap,  
\begin{equation}
\label{eq:Topological Tverberg property}
 f(\sigma_1)\cap\cdots\cap f(\sigma_r)\neq\emptyset.
\end{equation}
\end{conjecture}

The case $r=2$ of the topological Tverberg conjecture amounts to the 
topological Radon theorem, so it holds.
The topological Tverberg conjecture is also easy to verify for $d=1$, as follows.

\begin{theorem}[Topological Tverberg theorem for $d=1$] 
Let $r\geq 2$ be an integer, and let $f : \Delta_{2r-2}\rightarrow \R$ be a continuous map.
Then there exist $r$ pairwise disjoint faces $\sigma_1,\ldots,\sigma_r$ of the simplex $\Delta_{2r-2}$ whose $f$-images overlap,  
\begin{equation*}
 f(\sigma_1)\cap\cdots\cap f(\sigma_r)\neq\emptyset.
\end{equation*}	
\end{theorem}

\begin{proof}
	Let $f : \Delta_{2r-2}\rightarrow \R$ be continuous.
	Sort the vertices of the simplex $\Delta_{2r-2}=\conv\{\ee_1,\ldots,\ee_{2r-1}\}$ such that
	$
	f(\ee_{\pi(1)})\leq f(\ee_{\pi(2)})\leq \cdots \leq 
	f(\ee_{\pi(2r-2)})\leq f(\ee_{\pi(2r-1)})
	$.
	Then the collection of $r-1$ edges and one vertex of $\Delta_{2r-2}$
	\[
	\sigma_1=[\ee_{\pi(1)},\ee_{\pi(2r-1)}], \
	\sigma_2=[\ee_{\pi(2)},\ee_{\pi(2r-2)}],  \ldots, \
	\sigma_{r-1}=[\ee_{\pi(r-1)},\ee_{\pi(r+1)}], \
	\sigma_{r}=\{\ee_{r}\}
	\]
	is a Tverberg partition for the map $f$.
\end{proof}

At first glance the topological Tverberg conjecture is a \emph{binary problem} in the Poincar\'e classification of mathematical problems.
To our surprise it is safe to say, at this point in time, that the topological Tverberg conjecture {was} one of the most \emph{interesting problems} that shaped interaction between Geometric Combinatorics on one hand and Algebraic and Geometric Topology on the other hand for almost four decades.

After settling the topological Tverberg conjecture for $d=1$ and $r=2$ we want to advance. How?

\subsection{Equivariant topology steps in}
\label{subsec : Equivariant topology steps in}

Let $d\geq 1$ and $r\geq 2$ be integers, and let $N=(d+1)(r-1)$.
Our effort to handle the topological Tverberg conjecture starts with an assumption that there is a  counterexample to the conjecture with  parameters $d$ and $r$.
Thus there is a continuous map $f : \Delta_N\rightarrow \R^d$ such that for every $r$-tuple $\sigma_1,\ldots,\sigma_r$ of pairwise disjoint faces of the simplex $\Delta_N$ their $f$-images do not intersect, that is,
\begin{equation}
	\label{eq : map f counterexample}
	f(\sigma_1)\cap\cdots\cap f(\sigma_r)=\emptyset.
\end{equation}

In order to capture this feature of our counterexample $f$ we parametrize all $r$-tuples of pairwise disjoint faces of the simplex $\Delta_N$.
This can be done in two similar, but different ways.

\subsubsection{The $r$-fold $2$-wise deleted product}

The  \emph{$r$-fold $2$-wise \dictionary{deleted product}} of a simplicial complex $K$ is the cell complex
\[
K^{\times r}_{\Delta(2)}:=
\{
(x_1,\ldots,x_r)\in \sigma_1\times\cdots\times \sigma_r\subset K^{\times r} :
\sigma_i\cap\sigma_j=\emptyset\text{  for  }i\neq j
\},	
\]
where $\sigma_1,\ldots,\sigma_r$ are non-empty faces of $K$.
The symmetric group $\Sym_r$ acts (from the left) on $K^{\times r}_{\Delta(2)}$ by
\[
\pi\cdot (x_1,\ldots,x_r):=	(x_{\pi^{-1}(1)},\ldots,x_{\pi^{-1}(r)}),
\]
for $\pi\in\Sym_r$ and $(x_1,\ldots,x_r)\in K^{\times r}_{\Delta(2)}$.
This action is free due to the fact that if $(x_1,\ldots,x_r)\in K^{\times r}_{\Delta(2)}$ then $x_i\neq x_j$ for all $i\neq j$.
We have seen a particular instance before:
The complex $X$ that we used in the proof 
of the topological Radon theorem \ref{th : Topological Radon theorem}
was $(\Delta_{d+1})^{\times2}_{\Delta(2)}$.
For more details on the deleted product construction 
see for example \cite{Barany1981} or \cite[Sec.\,6.3]{Matousek2008}.

In the case when $K$ is a simplex the topology of the deleted product 
$K^{\times r}_{\Delta(2)}$ is known from the following result of 
B\'ar\'any, Shlosman and Sz\H{u}cs \cite[Lem.\,1]{Barany1981}.

\begin{theorem}
\label{th : Connectivity of deleted product}
Let $N$ and $r$ be positive integers with $N\geq r-1$. 
Then $(\Delta_N)^{\times r}_{\Delta(2)}$ is an $(N-r+1)$-dimensional and $(N-r)$-connected CW complex.
\end{theorem}

\begin{proof}
A typical face of the CW complex $(\Delta_N)^{\times r}_{\Delta(2)}$ is of the form $\sigma_1\times\cdots\times\sigma_r$, where $\sigma_1,\ldots,\sigma_r$ are pairwise disjoint simplices.
Consequently, the number of vertices of these simplices together cannot exceed $N+1$, or in the language of dimension
\[
\dim (\sigma_1)+1+\cdots+\dim (\sigma_r)+1\leq N+1.
\]
The equality is attained when all the vertices are used, this is when $\sigma_1,\ldots,\sigma_r$ is a maximal face of dimension 
\[
\dim (\sigma_1\times\cdots\times\sigma_r) =  \dim(\sigma_1)+\cdots+\dim (\sigma_r)=N-r+1.
\]
Thus $(\Delta_N)^{\times r}_{\Delta(2)}$ is an $(N-r+1)$-dimensional CW complex.

\smallskip
For $N=r-1$ the deleted product $(\Delta_N)^{\times r}_{\Delta(2)}$ is the $0$-dimensional simplicial complex $[r]$ and the statement of the theorem holds.
Thus, we can assume that $N\geq r$.

\smallskip
For $N\geq r$ we establish the connectivity of the deleted product of a simplex by induction on $r$ making repeated use of the following classical 1957 theorem of Stephen Smale~\cite[Main Thm.]{Smale1957}:
\begin{quote}
	{\small
	\textbf{Smale's Theorem.} \emph{Let $X$ and $Y$ be connected, locally compact, separable metric spaces, and in addition let $X$ be locally contractible.
	Let $f:X\rightarrow Y$ be a continuous surjective proper map, that is, any preimage of a compact set is compact.
	If for every $y\in Y$ the preimage $f^{-1}(\{y\})$ is locally contractible and $n$-connected, then the induced homomorphism
	\[
	f_{\#} :  \pi_i(X)\rightarrow \pi_i(Y)
	\]
	is an isomorphism for all $0\leq i\leq n$, and is an epimorphism for $i=n+1$.}}
\end{quote}
Recall that $\Delta_{N}$ denotes the standard simplex, whose vertices $\ee_1,\ldots,\ee_{N+1}$ form the standard basis of~$\R^{N+1}$.
The induction starts with $r=1$ and the theorem claims that the simplex $\Delta_N$, a contractible space, is $(N-1)$-connected, which is obviously true.

In the case $r=2$ consider the surjection 
$p_1 : (\Delta_N)^{\times 2}_{\Delta(2)}\rightarrow \sk_{N-1}(\Delta_N)$ 
given by the projection on the first factor.
Any point $x_1$ of the $(N-1)$-skeleton $\sk_{N-1}(\Delta_N)$ of the simplex $\Delta_N$ lies in the relative interior of a face,
\[
x_1\in \relint\big(\conv\{\ee_{i} : i\in T\subseteq [N+1] \}\big)
\]
where $1\leq |T|\leq N$.
Let us denote the complementary set of vertices by $S:=\{\ee_i : i\notin T\}\neq\emptyset$ and its convex hull by $\Delta_S:=\conv (S)\cong\Delta_{|S|-1}$.
The fiber of the projection map $p_1$ over $x_1$ is given by
\[
p_1^{-1}(\{x_1\})=\{(x_1,x_2)\in (\Delta_N)^{\times 2}_{\Delta(2)} : x_2\in \Delta_S\}\cong \Delta_S,
\]
and consequently it is contractible.
By Smale's theorem the projection $p_1$ induces an isomorphism between homotopy groups.
Since we are working in the category of CW complexes the Whitehead theorem \cite[Thm.\,11.2]{Bredon2010} implies a homotopy equivalence of $(\Delta_N)^{\times 2}_{\Delta(2)}$ and $\sk_{N-1}(\Delta_N)$.
The $(N-1)$-skeleton of a simplex is $(N-2)$-connected and thus the theorem holds in the case $r=2$.

\smallskip
For the induction hypothesis assume that $(\Delta_N)^{\times i}_{\Delta(2)}$ is $(N-i)$-connected for all $i\leq k< r$.
In the induction step we want to prove that $(\Delta_N)^{\times (k+1)}_{\Delta(2)}$ is $(N-k-1)$-connected. 

\smallskip
Now consider the projection onto the first $k$ factors,
\[
p_k : (\Delta_N)^{\times (k+1)}_{\Delta(2)}\rightarrow \sk_{N-k}\big((\Delta_N)^{\times k}_{\Delta(2)}\big).
\]
Since $(\Delta_N)^{\times k}_{\Delta(2)}$ is $(N-k)$-connected 
by induction hypothesis, its $(N-k)$-skeleton 
$\sk_{N-k}\big((\Delta_N)^{\times k}_{\Delta(2)}\big)$ is $(N-k-1)$-connected.
For a typical point of the codomain we have that
\[
(x_1,\ldots,x_k)\in
\relint\big(\conv\{\ee_{i} : i\in T_1\subseteq [N+1] \}\big)
\times\cdots\times
\relint\big(\conv\{\ee_{i} : i\in T_k\subseteq [N+1] \}\big),
\]
where $T_i\cap T_j=\emptyset$ for all $1\leq i<j\leq k$, and $|T_1|-1+\cdots+|T_k|-1\leq N-k$.
As before, consider the complementary set of vertices $S:=\{\ee_i : i\notin T_1\cup\cdots\cup T_k \}\neq\emptyset$ and its convex hull $\Delta_S=\conv (S) \cong\Delta_{|S|-1}$.
The fiber of the projection map $p_k$ over $(x_1,\ldots,x_k)$ is given by
\[
p_k^{-1}(\{(x_1,\ldots,x_k)\})=\{(x_1,\ldots,x_k,x_{k+1})\in (\Delta_N)^{\times (k+1)}_{\Delta(2)} : x_{k+1}\in\Delta_S \}\cong \Delta_S,
\]
so it is contractible.
Again Smale's theorem applied to the projection $p_k$ induces an isomorphism between homotopy groups of $(\Delta_N)^{\times (k+1)}_{\Delta(2)}$ and $\sk_{N-k}\big((\Delta_N)^{\times k}_{\Delta(2)}\big)$.
Moreover, the Whitehead theorem implies that these spaces are homotopy equivalent. 
Since, $\sk_{N-k}\big((\Delta_N)^{\times k}_{\Delta(2)}\big)$ is $(N-k-1)$-connected we have concluded the induction step and the theorem is proved.
\end{proof}

\begin{remark}
Our proof of Theorem~\ref{th : Connectivity of deleted product} 
may be traced back to a proof in the lost preprint version of the paper \cite{Barany1981}.
Indeed, in the published version the first sentence of \cite[Proof of Lem.\,1]{Barany1981} says:
\begin{quote}
{\small
For this elementary proof we are indebted to the referee. Our original proof used the Leray spectral sequence.}
\end{quote}
Here we used Smale's theorem in place of the Leray spectral sequence argument.
\end{remark}

\subsubsection{The $r$-fold $k$-wise deleted join} 
Let $K$ be a simplicial complex. 
The  \emph{$r$-fold $k$-wise \dictionary{deleted join}} of the simplicial complex $K$ is the simplicial complex
\[
K^{* r}_{\Delta(k)}:=
\big\{
\lambda_1 x_1+\cdots+\lambda_r x_r\in \sigma_1*\cdots *\sigma_r\subset K^{*r} :
(\forall I\subset [n])\, \card I\geq k \Rightarrow \bigcap_{i\in I} \sigma_i=\emptyset
\big\},	
\]
where $\sigma_1,\ldots,\sigma_n$ are faces of $K$, including the empty face.
Thus in the case $k=2$ we have
\[
K^{* r}_{\Delta(2)}:=
\{
\lambda_1 x_1+\cdots+\lambda_r x_r\in \sigma_1*\cdots *\sigma_r\subset K^{*r} :
\sigma_i\cap\sigma_j=\emptyset\text{  for  }i\neq j
\}.
\]
The symmetric group $\Sym_r$ acts (from the left) on $K^{* r}_{\Delta(2)}$ as follows
\[
\pi\cdot (\lambda_1 x_1+\cdots+\lambda_r x_r):=	\lambda_{\pi^{-1}(1)} x_{\pi^{-1}(1)}+\cdots+\lambda_{\pi^{-1}(r)} x_{\pi^{-1}(r)},
\]
where $\pi\in\Sym_n$ and $\lambda_1 x_1+\cdots+\lambda_r x_r\in K^{* r}_{\Delta(2)}$.
This action is free only in the case when $r=2$.

\begin{examples}
\label{example:deleted product and joins}
(Compare Figure~\ref{fig3: examples}.)
\begin{compactenum}[(1)]
\item Let $K=\Delta_1$ be the $1$-simplex. 
Then $K^{\times 2}_{\Delta(2)}=S^0$ while $K^{*2}_{\Delta(2)}\cong S^1$.

\item For $K=S^0$ we have that $K^{\times 2}_{\Delta(2)}=S^0$, and $K^{*2}_{\Delta(2)}$ is a disjoint union of two intervals.

\item If $K=[3]$ then $K^{*2}_{\Delta(2)}\cong S^1$.
\item When $K=[k]$, the deleted join $K^{*r}_{\Delta(2)}$ is the $k\times r$ \dictionary{chessboard complex}, which is denoted by~$\Delta_{k,r}$.
\end{compactenum}
\end{examples}

\begin{figure}
\centering
\includegraphics[scale=0.75]{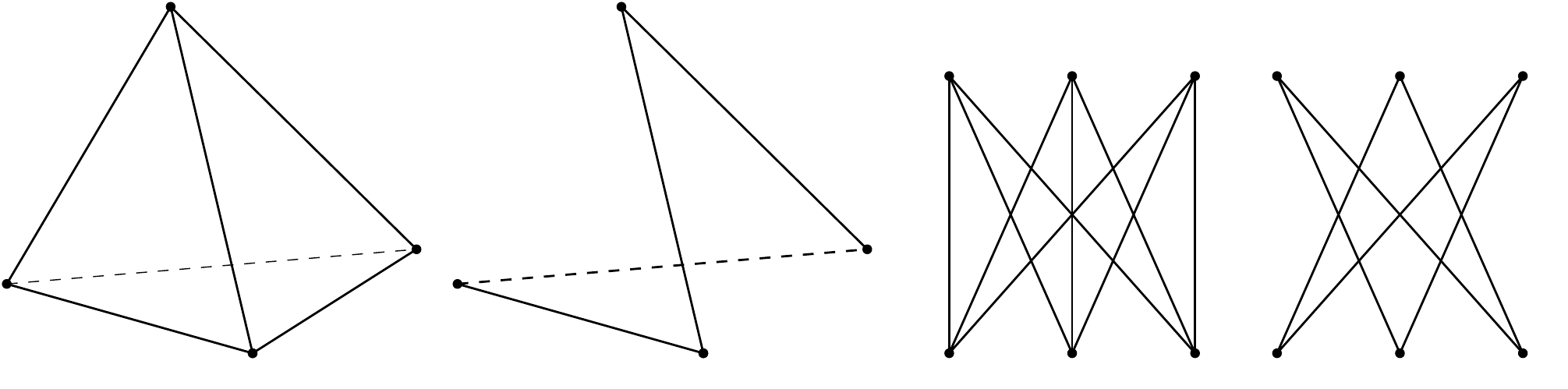}
\caption{\small The complexes $K^{*2}$ and $K^{*2}_{\Delta(2)}$ for $K=\Delta_1$ and $K=[3]$.}
\label{fig3: examples}
\end{figure}

The following lemma establishes the commutativity of the join and the deleted join operations on simplicial complexes. 
We state it for $k$-wise deleted joins and prove it here only for $2$-wise deleted joins.
For more details and insight consult the sections 
\emph{``Deleted Products Good''} and \emph{``\ldots Deleted Joins Better''} 
in Matou\v{s}ek's book, \cite[Sections.\,5.4 and 5.5]{Matousek2008}.

\begin{lemma}
\label{lemma:commutativity of join and deleted join}
Let $K$ and $L$ be simplicial complexes, and let $n\geq 2$ and $k\geq 2$ be integers.
There exists an isomorphism of simplicial complexes:
\[
(K*L)^{*n}_{\Delta(k)}\cong K^{*n}_{\Delta(k)} *  L^{*n}_{\Delta(k)}.
\]
\end{lemma}
\begin{proof}
We give a proof only for the case $k=2$.
Let $\sigma_1,\ldots,\sigma_n$ and $\tau_1,\ldots,\tau_n$ be simplices in $K$ and $L$, respectively, such that $\sigma_i\cap \sigma_j=\emptyset$ and $\tau_i\cap \tau_j=\emptyset$ for all $i\neq j$.
In addition, since the simplicial complexes $K$ and $L$ have disjoint vertex sets, we get that $\sigma_i\cap \tau_j=\emptyset$ as well for all $i$ and $j$.
Thus, for all $i\neq j$ we obtain an equivalence:
\[
(\sigma_i\cup \tau_i)\cap (\sigma_j\cup \tau_j)=\emptyset
\qquad\text{ if and only if }\qquad
\sigma_i\cap \sigma_j=\emptyset\text{ and }\tau_i\cap \tau_j=\emptyset.
\]
It induces a bijection between the following simplices of $(K*L)^{*n}_{\Delta(2)}$ and $ K^{*n}_{\Delta(2)} *  L^{*n}_{\Delta(2)}$ by:
\[
(\sigma_1*\tau_1) *_{\Delta(2)}\cdots*_{\Delta(2)}(\sigma_n*\tau_n)
\quad\longleftrightarrow\quad
(\sigma_1*_{\Delta(2)}\cdots*_{\Delta(2)}\sigma_n)*(\tau_1*_{\Delta(2)}\cdots*_{\Delta(2)}\tau_n).\vspace{-5pt}
\]
\end{proof}

A direct consequence of the previous lemma is the following useful fact.
\begin{lemma}
\label{lemma:deleted join of simplex}
Let $r\geq 2$ and $2\leq k\leq r$ be integers.
Then
\begin{compactenum}[\rm ~~(1)]
\item $(\Delta_N)^{*r}_{\Delta(2)}\cong [r]^{*(N+1)}$,
\item $(\Delta_N)^{*r}_{\Delta(k)}\cong (\sk_{k-2}(\Delta_{r-1}))^{*(N+1)}$.
\end{compactenum}
\end{lemma}
\begin{proof}
$(\Delta_N)^{*r}_{\Delta(k)}\cong
([1]^{*(N+1)})^{*r}_{\Delta(k)}\cong
([1]^{*r}_{\Delta(k)})^{*(N+1)}\cong
(\sk_{k-2}(\Delta_{r-1}))^{*(N+1)}$.
\end{proof}

\noindent
Thus the $r$-fold $2$-wise deleted join of an $N$-simplex $(\Delta_N)^{*r}_{\Delta(2)}$ is an $N$-dimensional and $(N-1)$-connected simplicial complex. 

\subsubsection{Equivariant maps induced by $f$}
\label{subsec : Equivariant maps induced by f}
Recall that, at the beginning of Section~\ref{subsec : Equivariant topology steps in}, we have fixed integers  $d\geq 1$ and $r\geq 2$, and in addition we assumed the existence of the continuous map $f : \Delta_N\rightarrow \R^d$ that is a counterexample to the topological Tverberg theorem.

Define continuous maps induced by $f$ in the following way:
\begin{compactitem}
	\item the \emph{product map} is
\[
P_f :  (\Delta_N)^{\times r}_{\Delta(2)}\rightarrow (\R^d)^{\times r}\cong (\R^d)^{\oplus r},
\qquad
(x_1,\ldots,x_r)\longmapsto (f(x_1),\ldots, f(x_r));
\]
	\item the \emph{join map} is
\[
J_f :  (\Delta_N)^{* r}_{\Delta(2)}\rightarrow (\R^{d+1})^{\oplus r},
\qquad
\lambda_1 x_1+\cdots +\lambda_r x_r\longmapsto (\lambda_1, \lambda_1 f(x_1))\oplus\cdots\oplus(\lambda_r, \lambda_r f(x_r)).
\]
\end{compactitem}
The codomains $(\R^d)^{\oplus r}$ and $(\R^{d+1})^{\oplus r}$ of the maps $P_f$ and $J_f$ are equipped with the action of the symmetric group $\Sym_r$ given by permutation of the corresponding $r$ factors, that is
\[
\pi\cdot (y_1,\ldots,y_r) = (y_{\pi^{-1}(1)},\ldots,y_{\pi^{-1}(r)})
\quad\text{and}\quad
\pi\cdot (z_1,\ldots,z_r) = (z_{\pi^{-1}(1)},\ldots,z_{\pi^{-1}(r)})
\]
for $(y_1,\ldots,y_r)\in (\R^d)^{\oplus r}$ and $(z_1,\ldots,z_r)\in (\R^{d+1})^{\oplus r}$.
Then both maps $P_f$ and $J_f$ are $\Sym_r$-equivariant.
Indeed, the following diagrams commute:
{\small%
\[
\xymatrix{
(x_1,\ldots,x_r)\ar[r]^-{P_f}\ar[d]_{\pi\cdot}  & (f(x_1),\ldots,f(x_r))\ar[d]_{\pi\cdot}\\
(x_{\pi^{-1}(1)},\ldots,x_{\pi^{-1}(r)})\ar[r]^-{P_f} & (f(x_{\pi^{-1}(1)}),\ldots,f(x_{\pi^{-1}(r)})),
}
\]}%
and
{\small%
\[
\xymatrix{
\lambda_1 x_1+\cdots+\lambda_r x_r\ar[r]^-{J_f}\ar[d]_{\pi\cdot}  & (\lambda_1, \lambda_1 f(x_1))\oplus\cdots\oplus(\lambda_r, \lambda_r f(x_r))\ar[d]_{\pi\cdot}\\
\lambda_{\pi^{-1}(1)} x_{\pi^{-1}(1)}+\cdots+ \lambda_{\pi^{-1}(r)}x_{\pi^{-1}(r)}\ar[r]^-{J_f} & (\lambda_{\pi^{-1}(1)}, \lambda_{\pi^{-1}(1)} f(x_{\pi^{-1}(1)}))\oplus\cdots\oplus(\lambda_{\pi^{-1}(r)}, \lambda_{\pi^{-1}(r)} f(x_{\pi^{-1}(r)})).
}
\]}%
The $\Sym_r$-invariant subspaces 
\[
D_P:=\{(y_1,\ldots,y_r)\in (\R^d)^{\oplus r} : y_1=\cdots=y_r\}
\quad\text{and}\quad
D_J:=\{(z_1,\ldots,z_r)\in (\R^{d+1})^{\oplus r} : z_1=\cdots=z_r\}
\]
of the codomains $(\R^d)^{\oplus r}$ and $(\R^{d+1})^{\oplus r}$, respectively, are called the \emph{thin diagonals}.
The crucial property of the maps $P_f$ and $J_f$, for a counterexample continuous map $f : \Delta_N\rightarrow \R^d$, is that
\begin{equation}
	\label{eq : empty intersection with diagonals}
	\im (P_f) \cap D_P=\emptyset
	\qquad\text{and}\qquad
	\im (J_f) \cap D_J=\emptyset.
\end{equation}
Indeed, the property \eqref{eq : map f counterexample} of the map $f$ immediately implies that $\im (P_f)$ and $D_P$ are disjoint.
For the second relation of \eqref{eq : empty intersection with diagonals} assume that 
\[
(\lambda_1, \lambda_1 f(x_1))\oplus\cdots\oplus(\lambda_r, \lambda_r f(x_r))\in \im (J_f) \cap D_J\neq\emptyset
\] 
for some $\lambda_1 x_1+\cdots+\lambda_r x_r\in (\Delta_N)^{* r}_{\Delta(2)}$.
Then $\lambda_1=\cdots=\lambda_r=\tfrac1{r}$ and consequently $f(x_1)=\cdots=f(x_r)$.

\smallskip
Therefore, the maps $P_f$ and $J_f$ induce $\Sym_r$-equivariant maps
\begin{equation}
	\label{eq : eq-maps-01}
	(\Delta_N)^{\times r}_{\Delta(2)}\rightarrow (\R^d)^{\oplus r}{\setminus}D_P
\qquad\text{and}\qquad
(\Delta_N)^{* r}_{\Delta(2)}\rightarrow (\R^{d+1})^{\oplus r}{\setminus}D_J
\end{equation}
that, with an obvious abuse of notation, are again denoted by $P_f$ and $J_f$, respectively.
Let us denote by 
\begin{equation}
	\label{eq : eq-maps-02}
	R_P :  (\R^d)^{\oplus r}{\setminus}D_P\rightarrow D_P^{\perp}{\setminus}\{0\} \rightarrow S(D_P^{\perp})	
	\quad\text{and}\quad
	R_J :  (\R^{d+1})^{\oplus r}{\setminus}D_J\rightarrow D_J^{\perp}{\setminus}\{0\} \rightarrow S(D_J^{\perp})	
\end{equation}
the compositions of projections and deformation retractions.
Here $U^{\perp}$ denotes the orthogonal complement of the subspace $U$ in the relevant ambient real vector space, while $S(V)$ denotes the unit sphere in the real vector space $V$.
Both maps $R_P$ and $R_J$ are $\Sym_r$-equivariant maps with respect to the introduced actions.

Furthermore, let $\R^r$ be a vectors space with the (left) action of the symmetric group $\Sym_r$ given by the permutation of coordinates. 
Then the subspace $W_r=\big\{(t_1,\ldots,t_r)\in\R^r : \sum_{i=1}^r t_i=0\big\}$ is an $\Sym_r$-invariant subspace of dimension $r-1$.
There is an isomorphism of real $\Sym_r$-representations
\[
D_P^{\perp}\cong W_r^{\oplus d}
\qquad\text{and}\qquad
D_J^{\perp}\cong W_r^{\oplus (d+1)}.
\] 
Using this identification of $\Sym_r$-representations the $\Sym_r$-equivariant maps $R_P$ and $R_J$, defined in \eqref{eq : eq-maps-02}, can be presented by
\begin{equation}
	\label{eq : eq-maps-03}
	R_P :  (\R^d)^{\oplus r}{\setminus}D_P\rightarrow S(W_r^{\oplus d})
	\qquad\text{and}\qquad
	R_J :  (\R^{d+1})^{\oplus r}{\setminus}D_J\rightarrow  S(W_r^{\oplus (d+1)}).
\end{equation}

Finally we have the theorem we were looking for.
It will give us a chance to employ methods of algebraic topology to attack the topological Tverberg conjecture.

\begin{theorem}
	\label{th : criterion for topological Tverberg}
	Let $d\geq 1$ and $r\geq 2$ be integers, and let $N=(d+1)(r-1)$.
    If there exists a counterexample to the topological Tverberg conjecture, then there exist $\Sym_r$-equivariant maps
\[
	(\Delta_N)^{\times r}_{\Delta(2)}\rightarrow S(W_r^{\oplus d})
	\qquad\text{and}\qquad
	(\Delta_N)^{* r}_{\Delta(2)}\rightarrow  S(W_r^{\oplus (d+1)}).
\]
\end{theorem}
\begin{proof}
If $f : \Delta_N\rightarrow \R^d$ is a counterexample to the topological Tverberg conjecture, then by composing maps from \eqref{eq : eq-maps-01} and \eqref{eq : eq-maps-03} we get $\Sym_r$-equivariant maps
\[
R_P\circ P_f :  (\Delta_N)^{\times r}_{\Delta(2)}\rightarrow S(W_r^{\oplus d})
\qquad\text{and}\qquad
R_J\circ J_f :  (\Delta_N)^{* r}_{\Delta(2)}\rightarrow  S(W_r^{\oplus (d+1)}).\vspace{-5pt}
\]
\end{proof}

Now we have constructed our equivariant maps.
The aim is to \emph{find as many $r$'s as possible such that an $\Sym_r$-equivariant map}
\begin{equation}
	\label{eq : eq-maps-04}
	(\Delta_N)^{\times r}_{\Delta(2)}\rightarrow S(W_r^{\oplus d}),
	\qquad\emph{or}\qquad
	(\Delta_N)^{* r}_{\Delta(2)}\rightarrow  S(W_r^{\oplus (d+1)}),
\end{equation} 
{\em cannot exist}.
For this we keep in mind that the $\Sym_r$-action on $(\Delta_N)^{\times r}_{\Delta(2)}$ is free, while for $r\geq 3$ the $\Sym_r$-action on $(\Delta_N)^{* r}_{\Delta(2)}$ is not free.

\subsection{The topological Tverberg theorem}
\label{subsec : The topological Tverberg theorem for r a prime}

The story of the topological Tverberg conjecture continues with a 1981 breakthrough of B\'ar\'any, Shlosman and Sz\H{u}cs \cite{Barany1981}. 
They proved that in the case when $r$ is a prime, there is no $\Z/r$-equivariant map $(\Delta_N)^{\times r}_{\Delta(2)}\rightarrow S(W_r^{\oplus d})$, and consequently no $\Sym_r$-equivariant map can exist.
Hence, Theorem~\ref{th : criterion for topological Tverberg} settles the topological Tverberg conjecture in the case when $r$ is a prime.
We give a proof of this result relying on the following theorem of Dold \cite{Dold1983} \cite[Thm.\,6.2.6]{Matousek2008}:
\begin{quote}
	{\small
	\textbf{Dold's theorem.} \emph{Let $G$ be a non-trivial finite group.
For an $n$-connected $G$-space $X$ and at most $n$-dimensional free $G$-CW complex $Y$ there cannot be any continuous $G$-equivariant map $X\rightarrow Y$.}}
\end{quote}
 
\begin{theorem}[Topological Tverberg theorem for primes $r$]
 Let $d\geq 1$ be an integer, let $r\geq 2$ be a prime, $N=(d+1)(r-1)$, and let $f : \Delta_N\rightarrow \R^d$ be a continuous map.
Then there exist $r$ pairwise disjoint faces $\sigma_1,\ldots,\sigma_r$ of the simplex $\Delta_N$ whose $f$-images overlap, that is 
\begin{equation*}
 f(\sigma_1)\cap\cdots\cap f(\sigma_r)\neq\emptyset.
\end{equation*}	
\end{theorem} 

\begin{proof}
According to Theorem \ref{th : criterion for topological Tverberg} it suffices to prove that there cannot be any $\Sym_r$-equivariant map $(\Delta_N)^{\times r}_{\Delta(2)}\rightarrow S(W_r^{\oplus d})$.
Let $\Z/r$ be the subgroup of $\Sym_r$ generated by the cyclic permutation $(123\ldots r)$.
Then it is enough to prove that there is no $\Z/r$-equivariant map $(\Delta_N)^{\times r}_{\Delta(2)}\rightarrow S(W_r^{\oplus d})$.
For that we are going to use Dold's theorem.

\noindent
The assumption that $r$ is a prime implies that the action of $\Z/r$ on the the sphere $S(W_r^{\oplus d})$ is free.
Now, since
\begin{compactitem}
	\item $(\Delta_N)^{\times r}_{\Delta(2)}$ is an $(N-r)$-connected $\Z/r$-space, and
	\item $S(W_r^{\oplus d})$ is a free $(N-r)$-dimensional $\Z/r$-CW complex,	
\end{compactitem}
the theorem of Dold implies that a $\Z/r$-equivariant map $(\Delta_N)^{\times r}_{\Delta(2)}\rightarrow S(W_r^{\oplus d})$ cannot exist.
\end{proof}
\noindent
The same argument yields that there cannot be any $\Sym_r$-equivariant map $(\Delta_N)^{* r}_{\Delta(2)}\rightarrow  S(W_r^{\oplus (d+1)})$, when $r$ is a prime.
Observe that for an application of the theorem of Dold the nature of the group action on the domain is of no importance.

\bigskip 
The next remarkable step followed a few years later.
In 1987 in his landmark unpublished manuscript \"Ozaydin \cite{Oezaydin1987} extended the result of B\'ar\'any, Shlosman and Sz\H{u}cs and proved that the topological Tverberg conjecture holds for $r$ a prime power.
He proved even more and left the topological Tverberg conjecture as a teaser for generations of mathematicians to come.
But this story will come a bit later. 

The first published proof of the topological Tverberg theorem for $r$ a prime power appeared in a paper of Aleksei Yu.\ Volovikov \cite{Volovikov1996-1}; see Remark \ref{rem : volovikov}.
Here we give a proof of the topological Tverberg theorem for prime powers
based on a comparison of Serre spectral sequences which uses a \dictionary{consequence of the localization theorem} for \dictionary{equivariant cohomology} \cite[Cor.\,1, p.\,45]{Hsiang1975}.
For background on spectral sequences we refer to the textbooks by John McCleary \cite{McCleary2001} and by Anatoly Fomenko and Dmitry Fuchs \cite{Fomenko2016}.

\begin{theorem}[Topological Tverberg theorem for prime powers $r$]
\label{th : Topological Tverberg theorem for prime power}
 Let $d\geq 1$ be an integer, let $r\geq 2$ be a prime power, $N=(d+1)(r-1)$, and let $f : \Delta_N\rightarrow \R^d$ be a continuous map.
Then there exist $r$ pairwise disjoint faces $\sigma_1,\ldots,\sigma_r$ of the simplex $\Delta_N$ whose $f$-images overlap, that is 
\begin{equation*}
 f(\sigma_1)\cap\cdots\cap f(\sigma_r)\neq\emptyset.
\end{equation*}	
\end{theorem} 
\begin{proof}
Let $d\geq 1$ be an integer, and let $r=p^n$ for $p$ a prime.
By Theorem \ref{th : criterion for topological Tverberg} it suffices to prove that there cannot be any $\Sym_r$-equivariant map $(\Delta_N)^{\times r}_{\Delta(2)}\rightarrow S(W_r^{\oplus d})$.

Consider the elementary abelian group $(\Z/p)^n$ and the regular embedding $\mathrm{reg} :  (\Z/p)^n\rightarrow \Sym_{r}$, as explained in  \cite[Ex.\,2.7, p.\,100]{Adem2004}.
It is given by the left translation action of $(\Z/p)^n$ on itself: 
To each element $g\in (\Z/p)^n$ we associate the permutation $L_g :   (\Z/p)^n\rightarrow  (\Z/p)^n$ from $\mathrm{Sym}((\Z/p)^n)\cong\Sym_{r}$ given by $L_g(x)=g+x$.
We identify the elementary abelian group $(\Z/p)^n$ with the subgroup $\im(\mathrm{reg})$ of the symmetric group $\Sym_r$.
Thus, in order to prove the non-existence of an $\Sym_r$-equivariant map it suffices 
to prove the non-existence of a $(\Z/p)^n$-equivariant map $(\Delta_N)^{\times r}_{\Delta(2)}\rightarrow S(W_r^{\oplus d})$.

Our proof takes several steps; the crucial ingredient is a comparison of Serre spectral sequences. As it will be by contradiction,
let us now assume that a $(\Z/p)^n$-equivariant map $\varphi  :  (\Delta_N)^{\times r}_{\Delta(2)}\rightarrow S(W_r^{\oplus d})$ exists.

\smallskip
\noindent{\bf (1)}
Let $\lambda$ denote the \dictionary{Borel construction} fiber bundle
\[
\lambda\quad :\quad
(\Delta_N)^{\times r}_{\Delta(2)}\rightarrow \E (\Z/p)^n\times_{(\Z/p)^n}(\Delta_N)^{\times r}_{\Delta(2)}\rightarrow \B (\Z/p)^n,
\]
while $\rho$ denotes the Borel construction fiber bundle
\[
\rho\quad :\quad
S(W_r^{\oplus d})\rightarrow \E (\Z/p)^n\times_{(\Z/p)^n}S(W_r^{\oplus d})\rightarrow \B (\Z/p)^n.
\]
Then the map $\varphi$ would induce the following morphism between fiber bundles $\lambda$ and $\rho$:
\[
\xymatrix{
\E (\Z/p)^n\times_{(\Z/p)^n} (\Delta_N)^{\times r}_{\Delta(2)}\ar[rr]^-{\id\times_{(\Z/p)^n} \varphi}\ar[d] &  &\E (\Z/p)^n\times_{(\Z/p)^n} S(W_r^{\oplus d})\ar[d]\\
\B (\Z/p)^n \ar[rr]^-{=}           &  &\B (\Z/p)^n.
}
\]
This bundle morphism induces a morphism of associated cohomology Serre spectral sequences:
\[
E^{i,j}_s (\lambda):=E^{i,j}_s (\E (\Z/p)^n\times_{(\Z/p)^n} (\Delta_N)^{\times r}_{\Delta(2)})\ \overset{\Phi^{i,j}_s}{\longleftarrow} \  E^{i,j}_s(\E (\Z/p)^n\times_{(\Z/p)^n} S(W_r^{\oplus d}))=:E^{i,j}_s (\rho)
\] 
such that on the zero row of the second term
\[
E^{i,0}_2 (\lambda):=E^{i,0}_2 (\E (\Z/p)^n\times_{(\Z/p)^n} (\Delta_N)^{\times r}_{\Delta(2)})\ \overset{\Phi^{i,0}_2}{\longleftarrow} \ E^{i,0}_2(\E (\Z/p)^n\times_{(\Z/p)^n} S(W_r^{\oplus d}))=:E^{i,0}_2 (\rho)
\]
is the identity.
Here for the morphisms we use the simplified notation $\Phi^{i,j}_s:=E^{i,j}_s(\id\times_{(\Z/p)^n}\varphi)$.

Before calculating both spectral sequences we recall the cohomology of $\B (\Z/p)^n$ with coefficients in the field $\F_p$ and denote it as follows:
\begin{equation*}
\begin{array}{lllll}
p=2: &H^*(\B\left((\Z/2) ^n\right);\F_2)=H^*((\Z/2)^n;\F_2)       \cong    \F_2[t_1,\ldots,t_n],                                & \deg t_{i}=1                     \\
p>2: & H^*(\B\left((\Z/p)^n\right);\F_p)=H^*((\Z/p)^n;\F_p)        \cong     \F_p[t_1,\ldots,t_n]\otimes \Lambda [e_1,\ldots,e_n], & \deg t_{i}=2,\,\deg e_{i}=1,
\end{array}
\end{equation*}
where $\Lambda[\,\cdot\,]$ denotes the exterior algebra. 

\smallskip
\noindent{\bf (2)}
First, we consider the Serre spectral sequence, with coefficients in the field $\F_p$, associated to the fiber bundle $\lambda$.
The $E_2$-term of this spectral sequence can be computed as follows:
\begin{eqnarray*}
	E^{i,j}_2(\lambda)  & = & H^i(\B \left((\Z/p)^n\right); \mathcal{H}^j((\Delta_N)^{\times r}_{\Delta(2)};\F_p))= H^i((\Z/p)^n; H^j((\Delta_N)^{\times r}_{\Delta(2)};\F_p))\\
	&=& \begin{cases}
		H^i((\Z/p)^n;\F_p), &\text{for } j=0,\\
		H^i( (\Z/p)^n; H^{N-r+1}((\Delta_N)^{\times r}_{\Delta(2)};\F_p)), &\text{for } j=N-r+1,\\
		0,  &\text{otherwise},
	\end{cases}
\end{eqnarray*}
since by Theorem \ref{th : Connectivity of deleted product} the deleted product $(\Delta_N)^{\times r}_{\Delta(2)}$ is an $(N-r+1)$-dimensional, $(N-r)$-connected simplicial complex and consequently $H^j((\Delta_N)^{\times r}_{\Delta(2)};\F_p)\neq 0$ only for $j=0$ or $j=N-r+1$.
Thus, the only possibly non-zero differential of the spectral sequence is $\partial_{N-r+2}$ and therefore $E^{i,0}_2(\lambda)\cong E^{i,0}_{\infty}(\lambda)$ for $i\leq N-r+1$. 

\smallskip
\noindent{\bf (3)}
The second Serre spectral sequence, with coefficients in the field $\F_p$, we consider is associated to the fiber bundle $\rho$.
In this case the fundamental group of the base space $\pi_1 (\B (\Z/p)^n)\cong (\Z/p)^n$ acts trivially on the cohomology $H^*(S(W_r^{\oplus d});\F_p)$.
Indeed, when $p=2$ the group $(\Z/2)^n$ can only act trivially on $H^0(S(W_r^{\oplus d});\F_2)\cong\F_2$ and on $H^{N-r}(S(W_r^{\oplus d});\F_2)\cong\F_2$.
For $p$ an odd prime all elements of the group $(\Z/p)^n$ have odd order and therefore the action is trivial on the $\F_p$ vector spaces $H^0(S(W_r^{\oplus d});\F_p)\cong\F_p$ and $H^{N-r}(S(W_r^{\oplus d});\F_p)\cong\F_p$. 
Thus the $E_2$-term of this spectral sequence is of the form
\begin{eqnarray*}
E^{i,j}_2(\rho)  &=& 	H^i(\B\left( (\Z/p)^n\right); \mathcal{H}^j(S(W_r^{\oplus d});\F_p))= H^i((\Z/p)^n; H^j(S(W_r^{\oplus d});\F_p))\\
&\cong & H^i((\Z/p)^n;\F_p)\otimes_{\F_p}H^j(S(W_r^{\oplus d});\F_p)
\cong 
\begin{cases}
		H^i( (\Z/p)^n;\F_p), & \text{for }j=0\text{ or } N-r,\\
		0,  &\text{otherwise}.
	\end{cases}
\end{eqnarray*}
Moreover, if $\ell\in H^{N-r}(S(W_r^{\oplus d});\F_p)\cong\F_p$ denotes a generator then the $(N-r)$-row of the $E_2$-term is a free $H^*((\Z/p)^n;\F_p)$-module generated by the element $1\otimes_{\F_p}\ell\in E^{0,N-r}_2(\rho)\cong H^{N-r}(S(W_r^{\oplus d});\F_p)$.
The only possible non-zero differential is 
\[
 \partial_{N-r+1} :  E^{0,N-r}_{N-r+1}(\rho) \rightarrow E^{N-r+1,0}_{N-r+1}(\rho).
\]
Consequently we have that
\begin{compactitem}
	\item $E_2$ and $E_{N-r+1}$ terms coincide, that is $E^{*,*}_2(\rho) \cong E^{*,*}_{N-r+1}(\rho)$,
	\item $(N-r)$-row of the $E_{N-r+1}$-term is a free $H^*((\Z/p)^n;\F_p)$-module generated by $1\otimes_{\F_p}\ell\in E^{0,N-r}_{N-r+1}(\rho)$,
	\item $\partial_{N-r+1}$, as all differentials, is an $H^*((\Z/p)^n;\F_p)$-module morphism.
\end{compactitem}
Therefore, the differential $\partial_{N-r+1}$ is zero if and only if 
\[
\partial_{N-r+1}(1\otimes_{\F_p}\ell)=0\in E^{N-r+1,0}_{N-r+1}(\rho)\cong E^{N-r+1,0}_2(\rho).
\]
Furthermore, if $\partial_{N-r+1}(1\otimes_{\F_p}\ell)=0$, then $E^{*,*}_2(\rho) \cong E^{*,*}_{\infty}(\rho)$.
Hence, the projection map 
\[
\E (\Z/p)^n\times_{(\Z/p)^n} S(W_r^{\oplus d})\rightarrow \B (\Z/p)^n
\]
induces a monomorphism in cohomology
\[
H^*(\B (\Z/p)^n;\F_p)\rightarrow H^*(\E (\Z/p)^n\times_{(\Z/p)^n} S(W_r^{\oplus d});\F_p).
\]
Now the following consequence of the localization theorem \cite[Cor.\,1, p.\,45]{Hsiang1975}, which 
in the case of finite groups holds only for elementary abelian groups, comes into play:
\begin{quote}
	{\small
	\textbf{Theorem.} \emph{Let $p$ be a prime, $G=(\Z/p)^n$ with $n\geq 1$, and let $X$ be a finite $G$-CW complex.
	The fixed point set $X^G$ of the space $X$ is non-empty if and only if the map in cohomology 
	$H^*(\B G;\F_p)\rightarrow H^*(\E G\times_{G} X;\F_p)$,
	induced by the projection $\E G\times_{G} X\rightarrow\B G$, is a monomorphism.
	}}
\end{quote}
Since the fixed point set $S(W_r^{\oplus d})^{(\Z/p)^n}=\emptyset$ of the sphere is empty, the theorem we just quoted implies that the map in cohomology
\[
H^*(\B (\Z/p)^n;\F_p)\rightarrow H^*(\E (\Z/p)^n\times_{(\Z/p)^n} S(W_r^{\oplus d});\F_p).
\]
is \emph{not} a monomorphism. 
Consequently, the element
\[
a:=\partial_{N-r+1}(1\otimes_{\F_p}\ell)\neq 0\in E^{N-r+1,0}_{N-r+1}(\rho)\cong E^{N-r+1,0}(\rho)
\] 
is \emph{not} zero.
\begin{figure}
\centering
\includegraphics[scale=0.81]{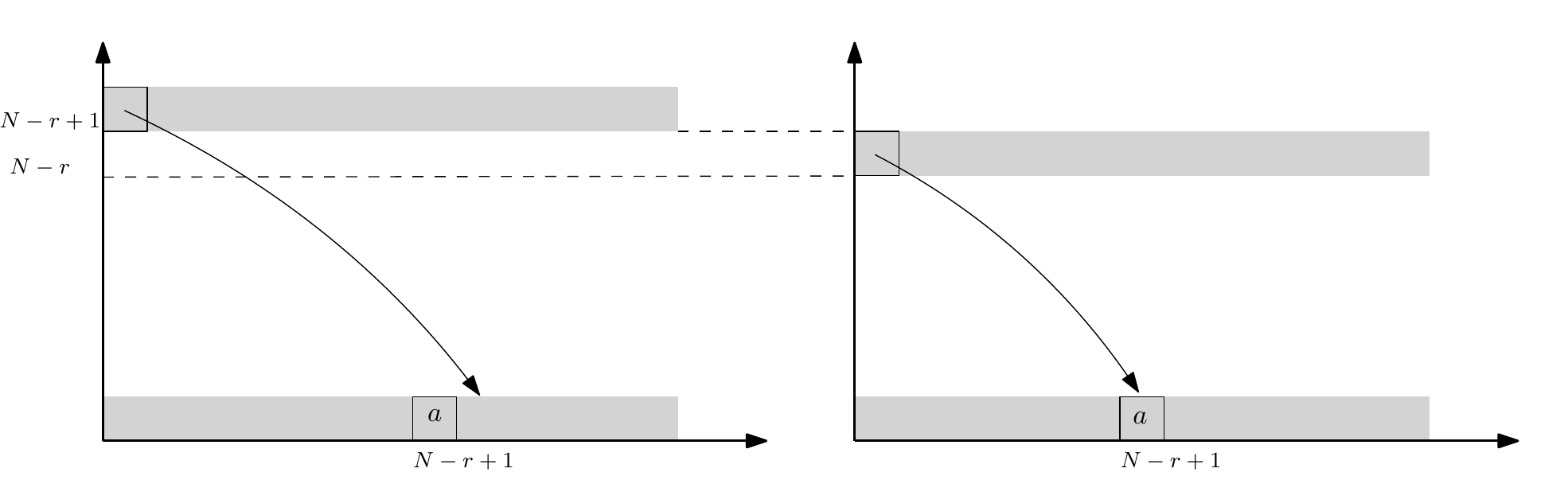}
\caption{\small Illustration of the spectral sequences $E^{*,*}_*(\lambda)$ and $E^{*,*}_*(\rho)$ and the morphism between them $\Phi^{*,*}_*: E^{*,*}_*(\lambda) \leftarrow E^{*,*}_*(\rho)$ that is the identity between the $0$-rows up to the $E_{N-r+1}$-term.}
\end{figure}

\smallskip
\noindent{\bf (4)}
Finally, to reach a contradiction with the assumption that the $(\Z/p)^n$-equivariant map $\varphi$ exists we track the element 
$a:=\partial_{N-r+1}(1\otimes_{\F_p}\ell)\neq 0\in E^{N-r+1,0}_{N-r+1}(\rho)\cong E^{N-r+1,0}_2(\rho)$ along the morphism of spectral sequences
\[\Phi^{N-r+1,0}_s :  E^{N-r+1,0}_s(\rho) \rightarrow E^{N-r+1,0}_s(\lambda) .\]
Since the differentials in both spectral sequences are zero in all terms $E_s(\rho)$ and $E_s(\lambda)$ for $2\leq s\leq N-r$ we have that
$\Phi^{*,0}_{s'}$ is the identity for $2\leq s'\leq N-r+1$.
In particular, the morphism  
\[
\Phi^{N-r+1,0}_{N-r+1} :  E^{N-r+1,0}_{N-r+1}(\rho) \rightarrow E^{N-r+1,0}_{N-r+1}(\lambda)
\]
is still identity as it was in the second term, and so $\Phi^{N-r+1,0}_{N-r+1}(a)=a$.
Passing to the $(N-r+1)$-term, with a slight abuse of notation, we have that 
\[
\Phi^{N-r+1,0}_{N-r+2}([a])=[a],
\]
where $[a]$ denotes the class induced by $a$ in the appropriate $(N-r+2)$-term of the spectral sequences.
Since $a:=\partial_{N-r+1}(1\otimes_{\F_p}\ell)\in E^{N-r+1,0}_{N-r+1}(\rho)$ and $0\neq a\in E^{N-r+1,0}_2(\lambda)\cong E^{N-r+1,0}_{\infty}(\lambda)$ passing to the next $E_{N-r+2}$-term we reach a contradiction:
\[
\Phi^{N-r+1,0}_{N-r+2}(0)=[a]=a\neq 0,
\]
because the class of the element $a$ in $E^{N-r+1,0}_{N-r+2}(\rho)$ vanishes (domain of $\Phi^{N-r+1,0}_{N-r+2}$) while in $E^{N-r+1,0}_{N-r+2}(\lambda)$ it does not vanish (codomain of $\Phi^{N-r+1,0}_{N-r+2}$).
Hence, there cannot be any $(\Z/p)^n$-equivariant map $(\Delta_N)^{\times r}_{\Delta(2)}\rightarrow S(W_r^{\oplus d})$, and the proof of the theorem is complete.
\end{proof}

\smallskip
\noindent
In the language of the \dictionary{Fadell--Husseini index}, as introduced in \cite{Fadell1988}, we have computed that
\[
	\ind_{(\Z/p)^n} ((\Delta_N)^{\times r}_{\Delta(2)};\F_p)\subseteq H^{\geq N-r+2} (\B (\Z/p)^n;\F_p).
\]
Furthermore, we showed the existence of an element $a\in H^{\geq N-r+1} (\B (\Z/p)^n;\F_p)$ that has the property
\begin{equation}
	\label{eq : element a}
	0\neq a \in \ind_{(\Z/p)^n} (S(W_r^{\oplus d});\F_p)\cap H^{N-r+1} (\B (\Z/p)^n;\F_p).
\end{equation}
Consequently $\ind_{(\Z/p)^n} (S(W_r^{\oplus d});\F_p)\not\subseteq\ind_{(\Z/p)^n} ((\Delta_N)^{\times r}_{\Delta(2)};\F_p)$ and so the monotonicity property of the Fadell--Husseini index implies the non-existence of a $(\Z/p)^n$-equivariant map $(\Delta_N)^{\times r}_{\Delta(2)}\rightarrow S(W_r^{\oplus d})$.

The element $a$ with the property \eqref{eq : element a} can be specified explicitly. 
It is the Euler class of the vector bundle 
\[
W_r^{\oplus d}\rightarrow \E (\Z/p)^n\times_{(\Z/p)^n} W_r^{\oplus d}\rightarrow \B (\Z/p)^n.
\]
From the work of Mann and Milgram \cite{Mann1982} we get that for an odd prime $p$ 
\[
a=\omega\cdot \Big( \prod_{(\alpha_1,\ldots,\alpha_n)\in\F_p^n{\setminus}\{0\} }(\alpha_1 t_1+\cdots +\alpha_n t_n)\Big)^{d/2},
\]
where $\omega\in\F_p{\setminus}\{0\}$, while for $p=2$ we have that
\[
a=\Big( \prod_{(\alpha_1,\ldots,\alpha_n)\in\F_2^n {\setminus}\{0\}} (\alpha_1 t_1+\cdots +\alpha_n t_n)\Big)^{d}.
\]
The square root in $\F_p[t_1,\ldots ,t_n]$ is not uniquely determined for an odd prime $p$ and $d$ odd:
The factor $\omega$ accounts for an arbitrary square root being taken.

\begin{remark}
\label{rem : volovikov}
Volovikov, in his 1996 paper \cite{Volovikov1996-1}, proved the following extension of the topological Tverberg theorem for continuous maps to manifolds:
\begin{quote}
{\bf Theorem.} 
{\em Let $d\geq 1$ be an integer, let $r\geq 2$ be a prime power, and $N=(d+1)(r-1)$.
For any topological $d$-manifold $M$ and any continuous map $f : \Delta_N\rightarrow M$, there  exist $r$ pairwise disjoint faces $\sigma_1,\ldots,\sigma_r$ of the simplex $\Delta_N$ whose $f$-images overlap, that is 
\[
 f(\sigma_1)\cap\cdots\cap f(\sigma_r)\neq\emptyset.
\]}
\end{quote}
 
\end{remark}

\section{Corollaries of the topological Tverberg theorem}
\label{sec : Corollaries of the topological Tverberg theorem}

Over time many results were discovered that were believed to be substantial extensions or analogs of the topological Tverberg theorem, such as the generalized Van Kampen--Flores theorem of Karanbir Sarkaria \cite{Sarkaria1991-1} and Aleksei Volovikov \cite{Volovikov1996-2}, the colored Tverberg theorems of Rade \v{Z}ivaljevi\'c and Sini\v{s}a Vre\'cica \cite{Zivaljevic1992} \cite{Vrecica1994} and Pablo Sober\'on's  result on Tverberg points with equal barycentric coordinates \cite{Soberon2013}.
It turned out only recently that the elementary idea of constraint functions together with the concept of “unavoidable complexes” introduced in \cite{Blagojevic2014} transforms all these results into simple corollaries of the topological Tverberg theorem.

Well, if all these results are corollaries, is there any genuine extension of the topological Tverberg theorem?
The answer to this question will bring us to the fundamental work of B\'ar\'any and Larman \cite{Barany1992}, and the optimal colored Tverberg theorem \cite{Blagojevic2009} from 2009.
But this will be the story of the final section of this paper.

\subsection{The generalized Van Kampen--Flores theorem}
\label{subsec : Generalized van Kampen--Flores theorem}

The first corollary we prove is the following generalized Van Kampen--Flores Theorem that was originally proved by Sarkaria \cite{Sarkaria1991-1} for primes and then by  Volovikov \cite{Volovikov1996-2} for prime powers.
The fact that this result can be derived easily from the topological Tverberg theorem by 
adding an extra component to the map  
was first sketched by Gromov in \cite[Sec.\,2.9c]{Gromov2010}; this can be seen as a first instance of the constraint method \cite[Thm.\,6.3]{Blagojevic2014} ``at work.''
  
\begin{theorem}[The generalized Van Kampen--Flores Theorem]
	\label{thm : genrelized_van_Kampen_Flores}
Let $d\ge1$ be an integer, let $r$ be a prime power, let $k\ge\lceil \tfrac{r-1}rd\rceil$ and $N=(d+2)(r-1)$, and let  $f :  \Delta_N \rightarrow \R^d$ be a continuous map.
Then there exist $r$ pairwise disjoint faces $\sigma_1, \dots, \sigma_r$ in the $k$-skeleton $\sk_k(\Delta_N)$ of the simplex $\Delta_N$ whose $f$-images overlap, 
\begin{equation*}
	f(\sigma_1)\cap\cdots\cap f(\sigma_r)\neq\emptyset.
\end{equation*}	
\end{theorem}
\begin{proof}
For the proof we use two ingredients, the topological Tverberg theorem and the pigeonhole principle.
First, consider the continuous map $g :  \Delta_N\rightarrow\R^{d+1}$ defined by 
\[
g(x)=(f(x),\dist(x, \sk_k(\Delta_N)).
\]
Since $N=(d+2)(r-1)=((d+1)+1)(r-1)$ and $r$ is a prime power we can apply the topological Tverberg theorem to the map $g$.
Consequently, there exist $r$ pairwise disjoint faces $\sigma_1,\ldots,\sigma_r$ with points $x_1\in\relint\sigma_1,\ldots,x_r\in\relint\sigma_r$ such that $g(x_1) = \cdots = g(x_r)$, that is,
\[
f(x_1) = \cdots = f(x_r)
\qquad\text{and}\qquad
\dist(x_1,\sk_k(\Delta_N))=\dots=\dist(x_r,\sk_k(\Delta_N)).
\]
One of the faces $\sigma_1,\ldots,\sigma_r$ has to belong to $\sk_k(\Delta_N)$.
Indeed, if all the faces $\sigma_1,\ldots,\sigma_r$, which are disjoint, would not belong to $\sk_k(\Delta_N)$, then the simplex $\Delta_N$ should have at least
    \[
     |\sigma_1|+\dots+|\sigma_r|\geq r(k+2)\geq r\big(\lceil \tfrac{r-1}rd\rceil+2\big)\geq (r-1)(d+2)+2=N+2
    \]
vertices.
Thus, since one of the faces is in the $k$-skeleton $\dist(x_1,\sk_k(\Delta_N))=\dots=\dist(x_r,\sk_k(\Delta_N))=0$, 
and consequently $\sigma_1\in \sk_k(\Delta_N),\ldots,\sigma_r\in \sk_k(\Delta_N)$, completing the proof of the theorem.
\end{proof}

\subsection{The colored Tverberg problem of B\'ar\'any and Larman}
\label{subsec : Colored Tverberg problem BL}

In their 1990 study on halving lines and halving planes, B\'ar\'any, Zoltan F\"uredi and L\'aszl\'o Lov\'asz \cite{Barany1990} realized a need for a colored version of the Tverberg theorem.
The sentence from this paper
\begin{quote}
	{\small For this we need a colored version of Tverberg’s theorem.}
\end{quote}
opened a new chapter in the study of extensions of the Tverberg theorem, both affine and topological.
Soon after, in 1992, B\'ar\'any and David Larman in \cite{Barany1992} formulated the colored Tverberg problem and brought to light a conjecture that motivated the progress in the area for decades to come.

Let $N\geq 1$ be an integer and let $\mathcal{C}$ be the set of vertices of the simplex $\Delta_N$. 
A \emph{coloring} of the set of vertices $\mathcal{C}$ by $\ell$ colors is a partition $(C_1,\ldots,C_{\ell})$ of $\mathcal{C}$, that is $\mathcal{C}=C_1\cup\cdots\cup C_{\ell}$ and $C_i\cap C_j=\emptyset$ for $1\leq i< j\leq\ell$.
The elements of the partition $(C_1,\ldots,C_{\ell})$ are called \emph{color classes}.
A face $\sigma$ of the simplex $\Delta_N$ is a \emph{rainbow} face if $|\sigma\cap C_i|\leq 1$ for all $1\leq i\leq \ell$.
The subcomplex of all rainbow faces of the simplex $\Delta_N$ induced by the coloring $(C_1,\ldots,C_{\ell})$ will be denoted by $R_{(C_1,\ldots,C_{\ell})}$ and will be called the \emph{rainbow subcomplex}.
There is an isomorphism of simplicial complexes $R_{(C_1,\ldots,C_{\ell})}\cong C_1*\cdots *C_{\ell}$.

\begin{problem}[B\'ar\'any--Larman colored Tverberg problem]
Let $d\geq 1$ and $r\geq 2$ be integers.
Determine the smallest number $n=n(d,r)$ such that for every affine map $f:\Delta _{n-1}\rightarrow \R^{d}$, and every coloring $(C_1,\ldots,C_{d+1})$ of the vertex set $\mathcal{C}$ of the simplex $\Delta_{n-1}$ by $d+1$ colors with each color of size at least $r$, there exist $r$ pairwise disjoint rainbow faces $\sigma_1, \dots, \sigma_r$ of $\Delta_{n-1}$  whose $f$-images overlap,
\begin{equation*}
	f(\sigma_1)\cap\cdots\cap f(\sigma_r)\neq\emptyset.
\end{equation*}	
\end{problem}
\noindent 
A trivial lower bound for the function $n(d,r)$ is $(d+1)r$.
B\'{a}r\'{a}ny and Larman proved that the trivial lower bound is tight in the cases $n(r,1)=2r $ and $n(r,2)=3r$, and presented a proof by Lov\'{a}sz for $n(2,d)=2(d+1)$.
Furthermore, they conjectured the following equality.

\begin{conjecture}[B\'{a}r\'{a}ny--Larman conjecture]
Let $r\ge2$ and $d\ge1$ be integers. 
Then $n(d,r)=(d+1)r$.	
\end{conjecture}

Now we present the proof of Lov\'{a}sz for the B\'{a}r\'{a}ny--Larman conjecture in the case $r=2$ from the paper of B\'ar\'any and Larman \cite[Thm.\,(iii)]{Barany1992}.

\begin{theorem}
	\label{th : Lovasz}
	Let $d\ge1$ be an integer.
	Then $n(2,d)=2(d+1)$.
\end{theorem}
\begin{proof}
	Let $n=2(d+1)$, and let $f : \Delta_{n-1}\rightarrow\R^d$ be an affine map.
	Furthermore, consider a coloring $(C_1,\ldots,C_{d+1})$ of the vertex set $\mathcal{C}$ of the simplex $\Delta_{n-1}$ by $d+1$ colors where $|C_1|=\cdots=|C_{d+1}|=2$.
	Denote $C_i=\{v_i,-v_i\}$ for $1\leq i\leq d+1$.
	The subcomplex of all rainbow faces of the simplex $\Delta_{n-1}$ is the join $R:=R_{(C_1,\ldots,C_{d+1})}=C_1*\cdots*C_{d+1}$.
	In this case, the rainbow subcomplex $R$ can be identified with the boundary of the cross-polytope $[2]^{*(d+1)} \cong S^{d}$. 
	Here $[2]$, as before, denotes the $0$-dimensional simplicial complex with two vertices. 
	
	\smallskip
	The restriction map $f|_R :  [2]^{*(d+1)}\rightarrow \R^d$ is a piecewise affine map, and therefore continuous.
	The Borsuk--Ulam theorem yields the existence of a point $x\in [2]^{*(d+1)}\cong S^{d}$ on the sphere with the property that $f|_R(x)=f|_R(-x)$.
	The point $x\in [2]^{*(d+1)}$ belongs to the relative interior of a unique simplex in the boundary of the cross-polytope  $[2]^{*(d+1)}$,
	\[
	x\in \relint \big(\conv \{\varepsilon_{i_1}v_{i_1},\ldots, \varepsilon_{i_k}v_{i_k}\}\big),
	\]
	where  $\varepsilon_{i_a}\in\{-1,+1\}$  and $1\leq k\leq d+1$.
	Thus, 
	$-x\in \relint\big( \conv \{-\varepsilon_{i_1}v_{i_1},\ldots, -\varepsilon_{i_k}v_{i_k}\}\big)$.
	Since the rainbow faces $\conv \{\varepsilon_{i_1}v_{i_1},\ldots, \varepsilon_{i_k}v_{i_k}\}$ and $\conv \{-\varepsilon_{i_1}v_{i_1},\ldots, -\varepsilon_{i_k}v_{i_k}\}$ are disjoint, and 
	\[
	f|_R(x)=f|_R(-x)\in 
	f|_R\big(\relint (\conv \{\varepsilon_{i_1}v_{i_1},\ldots, \varepsilon_{i_k}v_{i_k}\})\big)\cap
	f|_R\big(\relint(\conv \{-\varepsilon_{i_1}v_{i_1},\ldots, -\varepsilon_{i_k}v_{i_k}\})\big)\neq\emptyset,
	\]
	we have proved the theorem.
\end{proof}

\subsection{The colored Tverberg problem of \v{Z}ivaljevi\'c and Vre\'cica}
\label{subsec : Colored Tverberg problem ZV}
In response to the work of B\'ar\'any and Larman a modified colored Tverberg problem was presented by \v{Z}ivaljevi\'c and Vre\'cica in their influential paper \cite{Zivaljevic1992} from 1992. 

\begin{problem}[The \v{Z}ivaljevi\'c--Vre\'cica colored Tverberg problem]
Let $d\geq 1$ and $r\geq 2$ be integers.
Determine the smallest number $t=t(d,r)$ (or $t=tt(d,r)$) such that for every affine (or continuous) map $f:\Delta\rightarrow \R^{d}$, and every coloring $(C_1,\ldots,C_{d+1})$ of the vertex set $\mathcal{C}$ of the simplex $\Delta$ by $d+1$ colors with each color of size at least $t$, there exist $r$ pairwise disjoint rainbow faces $\sigma_1, \dots, \sigma_r$ of $\Delta$  whose $f$-images overlap, that is
\begin{equation*}
	f(\sigma_1)\cap\cdots\cap f(\sigma_r)\neq\emptyset.
\end{equation*}	
\end{problem}

\noindent
Observe that in the language of the function $t(d,r)$ the B\'{a}r\'{a}ny--Larman conjecture says that $t(d,r)=r$ for all $r\ge2$ and $d\ge1$.
Furthermore, proving that $t(d,r)<+\infty$ does \emph{not} imply $n(d,r)<+\infty$, while proving $t(d,r)=r$ would imply that $n(d,r)=r(d+1)$.

In order to address the modified problem \v{Z}ivaljevi\'c and Vre\'cica needed to know the connectivity of chessboard complexes.
For that they recalled the following result of Anders Bj\" orner, Lov\' asz, Vre\'{c}ica and \v{Z}ivaljevi\'{c} \cite[Thm.\,1.1]{Bjorner1994}.
Its connectivity lower bound is best possible according to \cite{Shareshian2007}.

\begin{theorem}
\label{th : conn. of chessboard}
Let $m\geq 1$ and $n\geq 1$ be integers.
The chessboard $\Delta_{m,n}$ is $\nu$-connected, where
\[
\nu=\min\left\{ m, n, \left\lfloor \tfrac{m+n+1}{3} \right\rfloor \right\}-2.
\]
\end{theorem}
\begin{proof}
Without loss of generality we can assume that $1\leq m\leq n$.
The proof proceeds by induction on $\min\{m,n\}=m$.
In the case $m=1$ the statement of the theorem is obviously true.
For $m=2$ we distinguish between two cases:
\begin{compactitem}
\item If $n=2$, then $\nu=\min\{2,2,  \left\lfloor \tfrac{5}{3} \right\rfloor\}-2=-1$ and $\Delta_{2,2}$ is just a disjoint union of two edges, and
\item If $n\geq 3$, then $\nu=\min\{2,2,  \left\lfloor \tfrac{n+3}{3} \right\rfloor\}-2=0$ and $\Delta_{2,n}$ is path connected.
\end{compactitem}
Let $m\geq3$, and let us assume that the statement of the theorem holds for every chessboard $\Delta_{m',n'}$ where $1\leq \min\{m',n'\}<m$.
Now we prove the statement of the theorem for the chessboard $\Delta_{m,n}$.

\smallskip
Let $K_{\ell}$ for $1\leq \ell\leq n$ be a  subcomplex of $\Delta_{m,n}$ defined by
\[
\{(i_0,j_0),\ldots, (i_k,j_k) \}\in K_{\ell} \ \Longleftrightarrow \ \{(i_0,j_0),\ldots, (i_k,j_k), (1,\ell) \}\in \Delta_{m,n}.
\]
The family of subcomplexes $\mathcal{K}:=\{K_{\ell} : 1\leq \ell\leq n\}$ covers the chessboard $\Delta_{m,n}$. 
Moreover, each subcomplex $K_{\ell}$ is a cone over the chessboard $\Delta_{m-1,n-1}$, and therefore contractible.
Since, for $\sigma\subseteq [n]$ we have that
\[
\bigcap \{K_{\ell} : \ell\in\sigma\}=\emptyset
 \ \Longleftrightarrow \ 
\sigma=[n], 
\] 
the \dictionary{nerve} $N_{\mathcal{K}}$ of the family $\mathcal{K}$ is homeomorphic to the boundary of an $(n-1)$-simplex $\partial\Delta_{n-1}$.
Thus, $N_{\mathcal{K}}\cong S^{n-2}$ is $(n-3)$-connected.
Furthermore, for $\sigma\subseteq [n]$ with the property that $2\leq |\sigma|\leq n-1$ the intersection $\bigcap \{K_{\ell} : \ell\in\sigma\}$ is homeomorphism with the chessboard $\Delta_{m-1,n-|\sigma|}$.
The induction hypothesis can be applied to each of these intersections.
Therefore,  
\[
\conn \big(\bigcap \{K_{\ell} : \ell\in\sigma\}\big) = \conn (\Delta_{m-1,n-|\sigma|}) \geq 
\min\left\{ m-1, n-|\sigma|, \left\lfloor \tfrac{m+n-|\sigma|}{3} \right\rfloor \right\}-2.
\]
Now we will apply the following connectivity version of the \dictionary{Nerve theorem} due to Bj\"orner, see \cite[Thm.\,10.6]{Bjorner1995}. 
\begin{quote}
	{\small
	\textbf{Theorem.} \emph{
	Let $K$ be a finite simplicial complex, or a regular CW-complex, and let $\mathcal{K}:=\{ K_i : i\in I\}$ be a cover of $K$ by a family of subcomplexes, $K=\bigcup\{ K_i : i\in I\}$.
\begin{compactenum}[\rm (1)]
\item If for every face $\sigma$ of the nerve $N_{\mathcal{K}}$ the intersection $\bigcap \{K_i : i\in \sigma\}$ is contractible, then $K$ and $N_{\mathcal{K}}$ are homotopy equivalent, $K\simeq N_{\mathcal{K}}$.
\item If for every face $\sigma$ of the nerve $N_{\mathcal{K}}$ the intersection $\bigcap \{K_i : i\in \sigma\}$ is $(k-|\sigma|+1)$-connected, then the complex $K$ is $k$-connected if and only if the nerve $N_{\mathcal{K}}$ is $k$-connected.
\end{compactenum}
	}}
\end{quote}
In the case of the covering $\mathcal{K}$ of the chessboard $\Delta_{m,n}$, where $2<m\leq n$, we have that
\begin{compactitem}
\item for every face $\sigma$ of the nerve $N_{\mathcal{K}}$ the intersection $\bigcap \{K_i : i\in \sigma\}$ is contractible when $|\sigma|=1$, and 
\begin{eqnarray*}
	\conn\big( \bigcap \{K_{\ell} : \ell\in\sigma\}\big) &\geq  &
\min\left\{ m-1, n-|\sigma|, \left\lfloor \tfrac{m+n-|\sigma|}{3} \right\rfloor \right\}-2\\
&\geq & \min\left\{ m, n, \left\lfloor \tfrac{m+n+1}{3} \right\rfloor \right\}-2 -|\sigma|+1\\
&\geq &\nu -|\sigma|+1,
\end{eqnarray*}
when $2\leq |\sigma|\leq n-1$, while
\item the nerve $N_{\mathcal{K}}$ of the family $\mathcal{K}$ is $(n-3)$-connected with
\[
n-3\geq  \min\left\{ m, n, \left\lfloor \tfrac{m+n+1}{3} \right\rfloor \right\}-2 = \nu.
\] 
\end{compactitem}
Therefore, according to the Nerve theorem applied for the cover $\mathcal{K}$ the chessboard $\Delta_{m,n}$ is $\nu$-connected.
This concludes the induction step.
\end{proof}

The knowledge on the connectivity of the chessboard complexes was the decisive information both for the original proof  of the \v{Z}ivaljevi\'c and Vre\'cica colored Tverberg theorem \cite[Thm.\,1]{Zivaljevic1992}, which worked only for primes, and for the following version of the proof for prime powers; see also the proof of \v{Z}ivaljevi\'c \cite[Thm.\,3.2\,(2)]{Zivaljevic1998}.

\begin{theorem}[Colored Tverberg theorem of \v{Z}ivaljevi\'c and Vre\'cica]
\label{th : colored Tverberg of ZV}
	Let $d\geq 1$ be an integer, and let $r\geq 2$ be a prime power.
	For every continuous map $f : \Delta\rightarrow \R^{d}$, and every coloring $(C_1,\ldots,C_{d+1})$ of the vertex set $\mathcal{C}$ of the simplex $\Delta$ by $d+1$ colors with each color of size at least $2r-1$, there exist $r$ pairwise disjoint rainbow faces $\sigma_1, \dots, \sigma_r$ of $\Delta$  whose $f$-images overlap, that is
\begin{equation*}
	f(\sigma_1)\cap\cdots\cap f(\sigma_r)\neq\emptyset.
\end{equation*}
\end{theorem}

\noindent
In the language of the function $tt(d,r)$ the previous theorem yields the upper bound $tt(d,r)\leq 2r-1$ when $r$ is a prime power.
This bound implies the bound $t(d,r)\leq tt(d,r) \leq 4r-3$ for arbitrary $r$ via Bertrand's postulate.

\begin{proof}
Let $r=p^n$ for $p$ a prime and $n\geq1$.
Let $f : \Delta\rightarrow \R^{d}$ be a continuous map from a simplex $\Delta$ whose set of vertices $\mathcal{C}$ is colored by $d+1$ colors $(C_1,\ldots,C_{d+1})$.
Without loss of generality assume that $|C_1|=\cdots=|C_{d+1}|=2r-1$.
In addition assume that the map $f$ is a counterexample for the statement of the theorem.
Set $M:=(d+1)(2r-1)-1$ and $N:=(d+1)(r-1)$, so $\Delta$ is an $M$-dimensional simplex.
Now, the proof of the theorem will be presented in several steps.

\smallskip
\noindent{\bf (1)}
The rainbow subcomplex of the simplex $\Delta$ induced by the coloring $(C_1,\ldots,C_{d+1})$ in this case is
\[
R_{(C_1,\ldots,C_{d+1})}\cong C_1*\cdots *C_{d+1}\cong [2r-1]^{*(d+1)}.
\]  
The $r$-fold $2$-wise deleted join of the rainbow subcomplex $R_{(C_1,\ldots,C_{d+1})}$ can be identified, with the help of Lemma \ref{lemma:commutativity of join and deleted join} and Example \ref{example:deleted product and joins},  as follows
\[
(R_{(C_1,\ldots,C_{d+1})})^{*r}_{\Delta(2)}\cong
\big([2r-1]^{*(d+1)}\big)^{*r}_{\Delta(2)}\cong
\big([2r-1]^{*r}_{\Delta(2)}\big)^{*(d+1)}\cong
(\Delta_{2r-1,r})^{*(d+1)}.
\]
The action of the symmetric group $\Sym_r$ on the chessboard $\Delta_{2r-1,r}$ is assumed 
to be given by permutation of columns of the chessboard, that is
\[
\pi\cdot \{(i_0,j_0),\ldots, (i_k,j_k)\} = \{(i_0,\pi(j_0)),\ldots, (i_k,\pi(j_k))\},
\]
for $\pi\in\Sym_r$ and $\{(i_0,j_0),\ldots, (i_k,j_k)\}$ a simplex in $\Delta_{2r-1,r}$.
Furthermore, the chessboard $\Delta_{2r-1,r}$ is an $(r-1)$-dimensional and according to Theorem~\ref{th : conn. of chessboard} an $(r-2)$-connected simplicial complex.
Therefore
\begin{eqnarray}\begin{split}
	\label{eq : connectivity in colored tverberg}
	\dim \big ((\Delta_{2r-1,r})^{*(d+1)}\big) &= (d+1)r-1=N+d, 	\\
	\conn\big((\Delta_{2r-1,r})^{*(d+1)}\big)  &= (d+1)r-2=N+d-1.
\end{split}\end{eqnarray}

\smallskip
\noindent{\bf (2)}
Now, along the lines of Section \ref{subsec : Equivariant maps induced by f}, the continuous map $f : \Delta\rightarrow \R^{d}$ induces the join map
\[
J_f :  (\Delta)^{* r}_{\Delta(2)}\rightarrow (\R^{d+1})^{\oplus r},
\qquad
\lambda_1 x_1+\cdots +\lambda_r x_r\longmapsto (\lambda_1, \lambda_1 f(x_1))\oplus\cdots\oplus(\lambda_r, \lambda_rf(x_r)).
\]
Both domain and codomain of the join map $J_f$ are equipped with the action of the symmetric group $\Sym_r$ in such a way that $J_f$ is an $\Sym_r$-equivariant map.
The deleted join of the rainbow complex $(R_{(C_1,\ldots,C_{d+1})})^{*r}_{\Delta(2)}$ is an $\Sym_r$-invariant subcomplex of $(\Delta)^{* r}_{\Delta(2)}$.
Thus, the restriction map
\[
J_f':=J_f|_{(R_{(C_1,\ldots,C_{d+1})})^{*r}_{\Delta(2)}}
 :  (R_{(C_1,\ldots,C_{d+1})})^{*r}_{\Delta(2)}\rightarrow (\R^{d+1})^{\oplus r}
\]
is also an $\Sym_r$-equivariant map.
Next consider the thin diagonal
\[
D_J=\{(z_1,\ldots,z_r)\in (\R^{d+1})^{\oplus r} : z_1=\cdots=z_r\}.
\]
This is an $\Sym_r$-invariant subspace of $(\R^{d+1})^{\oplus r}$.
The key property of the map $J_f'$ we have constructed for any counterexample continuous map $f : \Delta\rightarrow \R^d$, is that $\im (J_f') \cap D_J=\emptyset$.
Thus $J_f'$ induces an $\Sym_r$-equivariant map
\begin{equation}
	\label{eq : eq-maps-01-colored}
	(R_{(C_1,\ldots,C_{d+1})})^{*r}_{\Delta(2)}
	\rightarrow 
	(\R^{d+1})^{\oplus r}{\setminus}D_J
\end{equation}
which we, with an obvious abuse of notation, again denote by $J_f'$.
Furthermore, let 
\begin{equation}
	\label{eq : eq-maps-02-colored}
	R_J :  (\R^{d+1})^{\oplus r}{\setminus}D_J\rightarrow D_J^{\perp}{\setminus}\{0\} \rightarrow S(D_J^{\perp})	
\end{equation}
be the composition of the appropriate projection and deformation retraction.
The map $R_J$ is $\Sym_r$-equivariant.
Recall from Section \ref{subsec : Equivariant maps induced by f} that there is an isomorphism of real $\Sym_r$-representations
$D_J^{\perp}\cong W_r^{\oplus (d+1)}$.
Here $W_r=\big\{(t_1,\ldots,t_r)\in\R^r : \sum_{i=1}^r t_i=0\big\}$ and it is equipped with the (left) action of the symmetric group $\Sym_r$ given by permutation of coordinates.
After the identification of the $\Sym_r$-representations the $\Sym_r$-equivariant map $R_J$ defined in \eqref{eq : eq-maps-02-colored} has the form
\begin{equation}
	\label{eq : eq-maps-03-colored}
	R_J :  (\R^{d+1})^{\oplus r}{\setminus}D_J\rightarrow  S(W_r^{\oplus (d+1)}).
\end{equation}
Thus we have proved that if there exists a counterexample map $f$ for the theorem, then there exists an $\Sym_r$-equivariant map
\begin{equation}
	\label{eq : eq-maps-04-colored}
	(R_{(C_1,\ldots,C_{d+1})})^{*r}_{\Delta(2)}\rightarrow  S(W_r^{\oplus (d+1)}).
\end{equation}

\smallskip
In the final step we reach a contradiction by proving that an $\Sym_r$-equivariant map \eqref{eq : eq-maps-04-colored} cannot exist, concluding that a counterexample $f$ could not exist in the first place.
The proof of the non-existence of an equivariant map is following the footsteps of the proof of Theorem \ref{th : Topological Tverberg theorem for prime power}.

\smallskip
\noindent{\bf (3)}
Consider the elementary abelian group $(\Z/p)^n$ and its regular embedding $\mathrm{reg} :  (\Z/p)^n\rightarrow \Sym_{r}$.
Now it suffices to prove the non-existence of a $(\Z/p)^n$-equivariant map $(R_{(C_1,\ldots,C_{d+1})})^{*r}_{\Delta(2)}\rightarrow S(W_r^{\oplus (d+1)})$.
To prove the non-existence of such a map assume the opposite:
let $\varphi :  (R_{(C_1,\ldots,C_{d+1})})^{*r}_{\Delta(2)}\rightarrow S(W_r^{\oplus (d+1)})$ be a $(\Z/p)^n$-equivariant map.

\smallskip
Denote by $\lambda$ the Borel construction fiber bundle
\[
\lambda\quad : \quad
(R_{(C_1,\ldots,C_{d+1})})^{*r}_{\Delta(2)}\rightarrow \E (\Z/p)^n\times_{(\Z/p)^n} (R_{(C_1,\ldots,C_{d+1})})^{*r}_{\Delta(2)} \rightarrow \B (\Z/p)^n,
\]
and by $\rho$ the following Borel construction fiber bundle
\[
\rho\quad : \quad
S(W_r^{\oplus (d+1)})\rightarrow \E (\Z/p)^n\times_{(\Z/p)^n} S(W_r^{\oplus (d+1)}) \rightarrow \B (\Z/p)^n.
\]
Then the map $\varphi$ induces the following morphism of fiber bundles
\[
\xymatrix{
\E (\Z/p)^n\times_{(\Z/p)^n} (R_{(C_1,\ldots,C_{d+1})})^{*r}_{\Delta(2)}\ar[rr]^-{\id\times_{(\Z/p)^n} \varphi}\ar[d] &  &\E (\Z/p)^n\times_{(\Z/p)^n} S(W_r^{\oplus (d+1)})\ar[d]\\
\B (\Z/p)^n \ar[rr]^-{=}           &  &\B (\Z/p)^n.
}
\]
In turn, this morphism induces a morphism of corresponding Serre spectral sequences
\[
E^{i,j}_s (\lambda):=E^{i,j}_s (\E (\Z/p)^n\times_{(\Z/p)^n}(R_{(C_1,\ldots,C_{d+1})})^{*r}_{\Delta(2)})\overset{\Phi^{i,j}_s}{\longleftarrow}  E^{i,j}_s(\E (\Z/p)^n\times_{(\Z/p)^n} S(W_r^{\oplus (d+1)}))=:E^{i,j}_s (\rho)
\]
with the property that on the zero row of the second term the induced map 
\[
E^{i,0}_2 (\lambda):=E^{i,0}_2 (\E (\Z/p)^n\times_{(\Z/p)^n}(R_{(C_1,\ldots,C_{d+1})})^{*r}_{\Delta(2)})\overset{\Phi^{i,0}_2}{\longleftarrow}  E^{i,0}_2(\E (\Z/p)^n\times_{(\Z/p)^n} S(W_r^{\oplus (d+1)}))=:E^{i,0}_2 (\rho)
\]
is the identity.
Again we use simplified notation by setting $\Phi^{i,j}_s:=E^{i,j}_s(\id\times_{(\Z/p)^n}\varphi)$.

Depending on the parity of the prime $p$, we have that
\begin{equation*}
\begin{array}{lllll}
p=2: &H^*(\B\left((\Z/2) ^n\right);\F_2)=H^*((\Z/2)^n;\F_2)       \cong    \F_2[t_1,\ldots,t_n],                                & \deg t_{i}=1                     \\
p>2: & H^*(\B\left((\Z/p)^n\right);\F_p)=H^*((\Z/p)^n;\F_p)        \cong     \F_p[t_1,\ldots,t_n]\otimes \Lambda [e_1,\ldots,e_n], & \deg t_{i}=2,\,\deg e_{i}=1,
\end{array}
\end{equation*}
where $\Lambda[\,\cdot\,]$ denotes the exterior algebra.

\smallskip
\noindent{\bf (4)}
First we consider the Serre spectral sequence, with coefficients in the field $\F_p$, associated to the fiber bundle $\lambda$.
Using the connectivity of $(R_{(C_1,\ldots,C_{d+1})})^{*r}_{\Delta(2)}$ derived in \eqref{eq : connectivity in colored tverberg}, we get that the $E_2$-term of this spectral sequence is
\begin{eqnarray*}
	E^{i,j}_2(\lambda)  & = & H^i(\B\left( (\Z/p)^n\right); \mathcal{H}^j((R_{(C_1,\ldots,C_{d+1})})^{*r}_{\Delta(2)};\F_p))=H^i((\Z/p)^n; H^j((R_{(C_1,\ldots,C_{d+1})})^{*r}_{\Delta(2)};\F_p))\\
	&\cong & \begin{cases}
		H^i((\Z/p)^n;\F_p), &\text{for } j=0,\\
		H^i( (\Z/p)^n; H^{N+d}((R_{(C_1,\ldots,C_{d+1})})^{*r}_{\Delta(2)};\F_p)), & \text{for } j=N+d,\\
		0,  &\text{otherwise}.
	\end{cases}
\end{eqnarray*}
Consequently, the only possibly non-zero differential of this spectral sequence is $\partial_{N+d+1}$ and therefore $E^{i,0}_2(\lambda) \cong E^{i,0}_{\infty}(\lambda)$ for $i\leq N+d$. 

\begin{figure}
\centering
\includegraphics[scale=0.81]{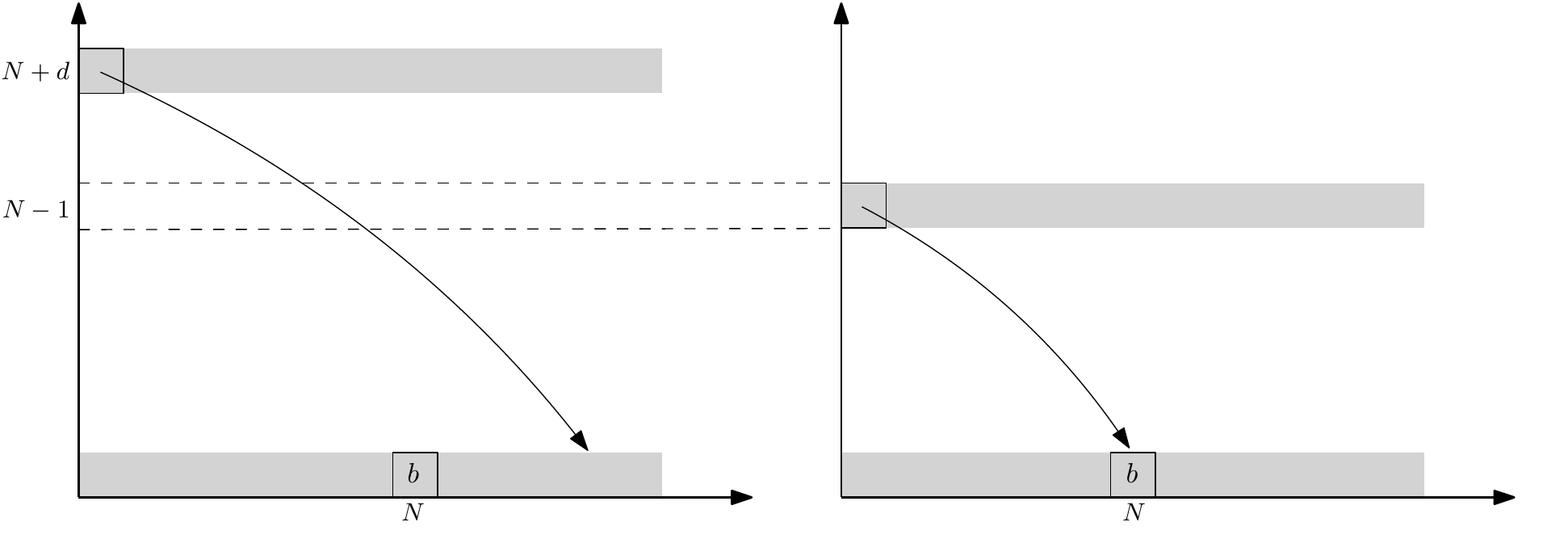}
\caption{Illustration of the spectral sequences $E^{*,*}_*(\lambda)$ and $E^{*,*}_*(\rho)$ and the morphism between them $\Phi^{*,*}_*: E^{*,*}_*(\lambda) \leftarrow E^{*,*}_*(\rho)$ that is the identity between the $0$-rows up to the $E_{N}$-term.}
\end{figure}

\smallskip
\noindent{\bf (5)}
The Serre spectral sequence, with coefficients in the field $\F_p$, associated to the fiber bundle $\rho$ is similar to the one appearing in the proof of Theorem \ref{th : Topological Tverberg theorem for prime power}.
Briefly, the $E_2$-term of this spectral sequence is
\[
E^{i,j}_2(\rho) =H^i((\Z/p)^n; H^j(S(W_r^{\oplus (d+1)});\F_p))\cong
\begin{cases}
		H^i( (\Z/p)^n;\F_p), & \text{for } j=0\text{ or } N-1,\\
		0,  &\text{otherwise}.
	\end{cases}
\]
Letting $\ell\in H^{N-1}(S(W_r^{\oplus (d+1)});\F_p)\cong\F_p$ denote a generator, then the $(N-1)$-row of the $E_2$-term can be seen as a free $H^*((\Z/p)^n;\F_p)$-module generated by $1\otimes_{\F_p}\ell\in E^{0,N-1}_2(\rho) \cong H^{N-1}(S(W_r^{\oplus (d+1)});\F_p)$.
Thus the only possible non-zero differential is $\partial_{N} :  E^{i,N-1}_{N}(\rho) \rightarrow E^{N+i,0}_{N}(\rho)$ and it is completely determined by the image $\partial_{N}(1\otimes_{\F_p}\ell)$.
As in the proof of Theorem \ref{th : Topological Tverberg theorem for prime power}, a consequence of the localization theorem implies that $b:=\partial_{N}(1\otimes_{\F_p}\ell)\neq 0\in E^{N,0}_{N}(\rho) \cong E^{N,0}_2(\rho)$ is \emph{not} zero.

\smallskip
\noindent{\bf (6)}
To reach the desired contradiction we track the element $b\in  E^{N,0}_{N}(\rho)\cong E^{N,0}_2(\rho)$ along the morphism of spectral sequences
\[
\Phi^{N,0}_s :  E^{N,0}_s(\rho) \rightarrow E^{N,0}_s(\lambda).
\]
The differentials in both spectral sequences are zero in all terms $E_s(\rho)$ and $E_s(\lambda)$ for $2\leq s\leq N-1$.
Thus, $\Phi^{*,0}_{s'}$ is an isomorphism for all $2\leq s'\leq N$.
In particular, the morphism  
$
\Phi^{N,0}_{N} :  E^{N,0}_{N}(\rho) \rightarrow E^{N,0}_{N}(\lambda)
$
is the identity, as it was in the second term, and so $\Phi^{N,0}_{N}(b)=b$.
When passing to the $(N+1)$-term, with a slight abuse of notation, we get
\[
\Phi^{N,0}_{N+1}([b])=[b],
\]
where $[b]$ denotes the class induces by $b$ in the appropriate $(N+1)$-term of the spectral sequences.
Since $b:=\partial_{N}(1\otimes_{\F_p}\ell)\in E^{N,0}_{N}(\rho)$ and $0\neq b\in E^{N,0}_2(\lambda) \cong E^{N,0}_{\infty}(\lambda)$ we have reached a contradiction:
\[
\Phi^{N,0}_{N+1}(0)=[b]=b\neq 0.
\]
Therefore, there cannot be any $(\Z/p)^n$-equivariant map $(R_{(C_1,\ldots,C_{d+1})})^{*r}_{\Delta(2)}\rightarrow S(W_r^{\oplus (d+1)})$, and the proof of the theorem is complete.
\end{proof}

As part of the proof of the previous theorem the following general criterion was derived.

\begin{corollary}
	\label{cor : CSTM for colored Tverberg}
	Let $(C_1,\ldots,C_{m})$ be a coloring of the simplex $\Delta$ by $m$ colors.
	If there is no $\Sym_r$-equivariant map
	\[
	\Delta_{|C_1|,r}*\cdots *	\Delta_{|C_m|,r}\cong (R_{(C_1,\ldots,C_{m})})^{*r}_{\Delta(2)}\rightarrow S(W_r^{\oplus (d+1)}),
	\]
	then for every continuous map $f : \Delta\rightarrow \R^{d}$ there exist $r$ pairwise disjoint rainbow faces $\sigma_1, \dots, \sigma_r$ of $\Delta$  whose $f$-images overlap, 
\begin{equation*}
	f(\sigma_1)\cap\cdots\cap f(\sigma_r)\neq\emptyset.
\end{equation*}
\end{corollary} 

The proof of Theorem \ref{th : colored Tverberg of ZV} could have been written in the language of the Fadell--Husseini index \cite{Fadell1988}.
The non-existence of an $(\Z/p)^n$-equivariant map $(R_{(C_1,\ldots,C_{d+1})})^{*r}_{\Delta(2)}\rightarrow S(W_r^{\oplus (d+1)})$ would then follow from the observation that 
\[
\ind_{(\Z/p)^n} ((R_{(C_1,\ldots,C_{d+1})})^{*r}_{\Delta(2)};\F_p) \not\supseteq \ind_{(\Z/p)^n} (S(W_r^{\oplus (d+1)});\F_p).
\]
More precisely, we have computed that
\begin{eqnarray*}
\ind_{(\Z/p)^n} ((R_{(C_1,\ldots,C_{d+1})})^{*r}_{\Delta(2)};\F_p) &=&\ind_{(\Z/p)^n} ((\Delta_{2r-1,r})^{*(d+1)};\F_p)\\
&\subseteq &  H^{\geq N+d+1} (\B (\Z/p)^n;\F_p),
\end{eqnarray*}
when $|C_1|=\cdots=|C_{d+1}|=2r-1$.
Actually we proved more:
\begin{equation}
	\label{eq : index chess}
	\ind_{(\Z/p)^n} ((\Delta_{2r-1,r})^{* k};\F_p)\subseteq H^{\geq k r} (\B (\Z/p)^n;\F_p).
\end{equation}
Furthermore we have found an element $b\in H^{N} (\B (\Z/p)^n;\F_p)$ with the property that
\[
	0\neq b \in \ind_{(\Z/p)^n} (S(W_r^{\oplus (d+1)});\F_p)\cap H^{N} (\B (\Z/p)^n;\F_p),
\]
and moreover
\begin{equation}
	\label{eq : index of the sphere S(W_r)}
	\ind_{(\Z/p)^n} (S(W_r^{\oplus (d+1)});\F_p)=\langle b \rangle .
\end{equation}
The element $b$ with this property is the Euler class of the vector bundle 
\[
W_r^{\oplus (d+1)}\rightarrow \E (\Z/p)^n\times_{(\Z/p)^n} W_r^{\oplus (d+1)}\rightarrow \B (\Z/p)^n.
\]
The work of Mann and Milgram \cite{Mann1982} allows us to specify the element $b$ completely: For $p$ an odd prime it is 
\[
b=\omega\cdot \Big( \prod_{(\alpha_1,\ldots,\alpha_n)\in\F_p^n{\setminus}\{0\} }(\alpha_1 t_1+\cdots +\alpha_n t_n)\Big)^{(d+1)/2},
\]
where $\omega\in\F_p{\setminus}\{0\}$, while for $p=2$ it is
\[
b=\Big( \prod_{(\alpha_1,\ldots,\alpha_n)\in\F_2^n {\setminus}\{0\}} (\alpha_1 t_1+\cdots +\alpha_k t_n)\Big)^{d+1}.
\]
The square root in $\F_p[t_1,\ldots ,t_n]$ is not uniquely determined
for an odd prime $p$ and $d$ odd. 
Thus we consider an arbitrary square root.

\smallskip
Combining these index computations we have that 
\begin{eqnarray}
\label{eq : difference in dimensions}	
0\neq b &\in & \ind_{(\Z/p)^n} (S(W_r^{\oplus (d+1)});\F_p)\cap H^{N} (\B (\Z/p)^n;\F_p)\nonumber \\
		& \not\subseteq   & \ind_{(\Z/p)^n} ((\Delta_{2r-1,r})^{*(d+1)};\F_p) \\
		&\subseteq & H^{\geq N+d+1} (\B (\Z/p)^n;\F_p).\nonumber
\end{eqnarray}
If a $(\Z/p)^n$-equivariant map $(R_{(C_1,\ldots,C_{d+1})})^{*r}_{\Delta(2)}\rightarrow S(W_r^{\oplus (d+1)})$ exists, then the monotonicity property of the Fadell--Husseini index yields the inclusion
\[
\ind_{(\Z/p)^n} ((R_{(C_1,\ldots,C_{d+1})})^{*r}_{\Delta(2)};\F_p) \supseteq \ind_{(\Z/p)^n} (S(W_r^{\oplus (d+1)});\F_p),
\]
which does not hold, as we just proved. 
Thus the $(\Z/p)^n$-equivariant map in question does not exist.

\smallskip
Now observe the difference of dimensions in \eqref{eq : difference in dimensions} and compare the dimension of the element $b$ and the dimension of the group cohomology where the index of the join $(\Delta_{2r-1,r})^{*(d+1)}$ lives.
\emph{We could have proved more.}
Indeed, using the index computation \eqref{eq : index chess} we have that
\begin{eqnarray*}
0\neq b &\in & \ind_{(\Z/p)^n} (S(W_r^{\oplus (d+1)});\F_p)\cap H^{ N} (\B (\Z/p)^n;\F_p)\\
&\not\subseteq &\ind_{(\Z/p)^n} ((\Delta_{2r-1,r})^{* k};\F_p)\\
&\subseteq & H^{\geq k r } (\B (\Z/p)^n;\F_p)
\end{eqnarray*}
as long as $k r\geq N+1$.
We have just concluded that, if $k r\geq N+1$, then there is no $(\Z/p)^n$-equivariant map 
\[
(\Delta_{2r-1,r})^{*k}\cong  (R_{(C_1,\ldots,C_{k})})^{*r}_{\Delta(2)}\rightarrow S(W_r^{\oplus (d+1)}).
\]
Thus with Corollary \ref{cor : CSTM for colored Tverberg} we have proved the following ``colored Tverberg theorem of type B''  \cite[Thm.\,4]{Vrecica1994}.

\begin{theorem}[The Colored Tverberg theorem of type B of Vre\'cica and \v{Z}ivaljevi\'c]
\label{th : colored Tverberg of ZV-B}
	Let $d\geq 1$ and $k\geq 1$ be integers, $N=(d+1)(r-1)$, and let $r\geq 2$ be a prime power.
	For every continuous map $f : \Delta\rightarrow \R^{d}$, and every coloring $(C_1,\ldots,C_{k})$ of the vertex set $\mathcal{C}$ of the simplex $\Delta$ by $k$ colors, with each color of size at least $2r-1$ and $k r\geq N+1$, there exist $r$ pairwise disjoint rainbow faces $\sigma_1, \dots, \sigma_r$ of $\Delta$  whose $f$-images overlap, that is
\begin{equation*}
	f(\sigma_1)\cap\cdots\cap f(\sigma_r)\neq\emptyset.
\end{equation*}
\end{theorem}

As we have just seen, the proof of the colored Tverberg theorem of \v{Z}ivaljevi\'c and Vre\'cica is in fact also a proof of a type B colored Tverberg theorem.
Is it possible that this proof hides a way to prove, for example the B\'ar\'any--Larman conjecture?
For this we would need to prove that for some or all $r$ and $|C_1|=\cdots =|C_{d+1}|=r$ there is no $\Sym_r$-equivariant map 
\begin{equation}
	\label{eq : eqivariant map for BL}
	\Delta_{r,r}^{*(d+1)}\cong (R_{(C_1,\ldots,C_{d+1})})^{*r}_{\Delta(2)} \rightarrow S(W_r^{\oplus (d+1)}).
\end{equation} 
The connectivity of the chessboard $\Delta_{r,r}$ is only $\big(\lfloor\tfrac{2r+1}3\rfloor-2\big)$ and therefore the scheme of the proof of Theorem \ref{th : colored Tverberg of ZV} cannot be used. 
Even worse, the complete approach fails, as the following theorem of Blagojevi\'c, Matschke and Ziegler \cite[Prop.\,4.1]{Blagojevic2009} shows that an $\Sym_r$-equivariant map \eqref{eq : eqivariant map for BL} does exist.

\begin{theorem}
\label{th : exitence of equivariant map for BL}
	Let $r\geq 2$ and $d\geq 1$ be integers. 
	There exists an $\Sym_r$-equivariant map 
	\[
	\Delta_{r,r}^{*(d+1)}\rightarrow S(W_r^{\oplus (d+1)}).
	\]
\end{theorem}

\begin{proof}
For this we use equivariant obstruction theory, as presented by Tammo tom Dieck \cite[Sec.\,II.3]{tomDieck:TransformationGroups}.

\noindent
Let $N:=(d+1)(r-1)$, $M:=r(d+1)-1$, and let $(C_1,\ldots,C_{d+1})$ be a coloring of the vertex set of the simplex $\Delta_M$ by $d+1$ colors of the same size $r$, that is $|C_1|=\cdots =|C_{d+1}|=r$.
As we know the deleted join $(R_{(C_1,\ldots,C_{d+1})})^{*r}_{\Delta(2)}$ of the rainbow complex is isomorphic to the join of chessboards $\Delta_{r,r}^{*(d+1)}$.
The action of the symmetric group $\Sym_r$ on the complex $\Delta_{r,r}^{*(d+1)}$ is not free.
The subcomplex of $\Delta_{r,r}^{*(d+1)}$ whose points have non-trivial stabilizers with respect to the action of $\Sym_r$ can be described as follows:
\begin{eqnarray*}
	(\Delta_{r,r}^{*(d+1)})^{>1} & = & ((R_{(C_1,\ldots,C_{d+1})})^{*r}_{\Delta(2)})^{>1}\\
& = &	\{ \lambda_1x_1+\cdots+\lambda_rx_r \in (R_{(C_1,\ldots,C_{d+1})})^{*r}_{\Delta(2)} : \lambda_i=\lambda_j=0\text{ for some }i\neq j\}.
\end{eqnarray*}
Here for a $G$-space (CW complex) $X$ we we use notation $X^{>1}$ for the subspace (subcomplex) of all points (cells) with non-trivial stabilizer, meaning that $X{\setminus}X^{>1}$ is a free $G$-space.

\smallskip
Let $f :  \Delta_M\rightarrow \R^d$ be any continuous map.
As explained in Section \ref{subsec : Equivariant maps induced by f} the map $f$ induces the join map given by
\[
J_f :  (\Delta_M)^{* r}_{\Delta(2)}\rightarrow (\R^{d+1})^{\oplus r},
\qquad
\lambda_1 x_1+\cdots +\lambda_r x_r\longmapsto (\lambda_1, \lambda_1 f(x_1))\oplus\cdots\oplus(\lambda_r, \lambda_rf(x_r)).
\]
Since the rainbow complex $(R_{(C_1,\ldots,C_{d+1})})^{*r}_{\Delta(2)}$ is an $\Sym_r$-invariant subcomplex of $(\Delta)^{* r}_{\Delta(2)}$, the restriction
\[
J_f':=J_f|_{(R_{(C_1,\ldots,C_{d+1})})^{*r}_{\Delta(2)}}
 :  (R_{(C_1,\ldots,C_{d+1})})^{*r}_{\Delta(2)}\rightarrow (\R^{d+1})^{\oplus r}
\]
is also an $\Sym_r$-equivariant map.
Moreover  $\im (J_f'|_{((R_{(C_1,\ldots,C_{d+1})})^{*r}_{\Delta(2)})^{>1}})\cap D_J=\emptyset$ where, as before, $D_J=\{(z_1,\ldots,z_r)\in (\R^{d+1})^{\oplus r} : z_1=\cdots=z_r\}$.
Thus the map $J_f'$ induces an $\Sym_r$-equivariant map 
\[
(\Delta_{r,r}^{*(d+1)})^{>1} = ((R_{(C_1,\ldots,C_{d+1})})^{*r}_{\Delta(2)})^{>1}\rightarrow (\R^{d+1})^{\oplus r}{\setminus}D_J.
\]
Composing this map with the $\Sym_r$-equivariant retraction $R_j :  (\R^{d+1})^{\oplus r}{\setminus}D_J \rightarrow S(D_J^{\perp})\cong S(W_r^{\oplus (d+1)})$ introduced in \eqref{eq : eq-maps-02}, we get a continuous $\Sym_r$-equivariant map
\begin{equation}
	\label{eq : map for extension - 01}
	(\Delta_{r,r}^{*(d+1)})^{>1} = ((R_{(C_1,\ldots,C_{d+1})})^{*r}_{\Delta(2)})^{>1}\rightarrow  S(W_r^{\oplus (d+1)}).
	\end{equation}

The $(r-1)$-dimensional chessboard complex $\Delta_{r,r}$ equivariantly retracts to a subcomplex of dimension $r-2$.
Indeed, for each facet of $\Delta_{r,r}$ there is an elementary collapse obtained by deleting all of its subfacets (faces of dimension $r-2$) that contain the vertex in the $r$-th column.
Performing these collapses to all facets of $\Delta_{r,r}$, we get that $\Delta_{r,r}$ collapses $\Sym_r$-equivariantly to an $(r-2)$-dimensional subcomplex of $\Delta_{r,r}$.
Consequently, the join $(\Delta_{r,r})^{*(d+1)}$ equivariantly retracts to a subcomplex $K$ of dimension $(d+1)(r-1)-1$.
Thus in order to prove the existence of an $\Sym_r$-equivariant map $\Delta_{r,r}^{*(d+1)}\rightarrow S(W_r^{\oplus (d+1)})$ it suffices to construct an $\Sym_r$-equivariant map $K \rightarrow S(W_r^{\oplus (d+1)})$.
Since
\begin{compactitem}
	\item $\dim K=\dim S(W_r^{\oplus (d+1)})=N-1$,
	\item $S(W_r^{\oplus (d+1)})$ is $(N-1)$-\dictionary{simple} and $(N-2)$-connected, 
\end{compactitem}
and the groups where the obstructions would live are zero, the equivariant obstruction theory yields the existence of an $\Sym_r$-equivariant map $K \rightarrow S(W_r^{\oplus (d+1)})$, provided that an $\Sym_r$-equivariant map $K^{>1} \rightarrow S(W_r^{\oplus (d+1)})$ exists.
The subcomplex of all points with non-trivial stabilizer $K^{>1}=K\cap (\Delta_{r,r}^{*(d+1)})^{>1}$ is a subcomplex of $(\Delta_{r,r}^{*(d+1)})^{>1}$ and therefore the map \eqref{eq : map for extension - 01} restricted to $K^{>1}$ completes the argument.	
\end{proof}

After this theorem an urgent question emerges: \emph{How are we going to handle the B\'ar\'any--Larman conjecture?}
An answer to this question will bring us to our last section and the optimal colored Tverberg theorem.

\subsection{The weak colored Tverberg theorem}
\label{subsec : Weak colored Tverberg theorem}

How many colored Tverberg theorems can we get directly from the topological Tverberg theorem without major topological machinery? 
Here is an answer given by  \cite[Thm.\,5.3]{Blagojevic2014}.

\begin{theorem}[The weak colored Tverberg theorem] 
\label{th : colored_tverberg}
	Let $d\ge1$ be an integer, let $r$ be a prime power, $N=(2d+2)(r-1)$, and let  $f :  \Delta_N \rightarrow \R^d$ be a continuous map.
  If the vertices of the simplex $\Delta_N$ are colored by $d+1$ colors, where each color class has cardinality at most $2r-1$,
  then there are $r$ pairwise disjoint rainbow faces $\sigma_1, \dots, \sigma_r$ of $\Delta_N$  whose $f$-images overlap, that is
\begin{equation*}
	f(\sigma_1)\cap\cdots\cap f(\sigma_r)\neq\emptyset.
\end{equation*}	
\end{theorem}
\begin{proof}
Let $\mathcal{C}$ be the set of vertices of the simplex $\Delta_N$ and let $(C_1,\ldots,C_{d+1})$ be a coloring of $\mathcal{C}$ where $|C_i|\leq 2r-1$ for all $1\leq i\leq d+1$.
To each color class $C_i$ we associate the subcomplex $\Sigma_i$ of $\Delta_N$ defined by
\[
\Sigma_i:=\{\sigma\in\Delta_N : |\sigma\cap C_i|\leq 1\}.
\]
Observe that the intersection $\Sigma_1\cap\cdots\cap\Sigma_{d+1}$ is the subcomplex of all rainbow faces of $\Delta_N$ with respect to the given coloring.
Next consider the continuous map $g :  \Delta_N\rightarrow\R^{2d+1}$ defined by 
\[
g(x)=(f(x),\dist(x,\Sigma_1),\dist(x,\Sigma_2),\ldots,\dist(x,\Sigma_{d+1})).
\]
Since $N=(2d+2)(r-1)=((2d+1)+1)(r-1)$ and $r$ is a prime power, we can apply the topological Tverberg theorem to $g$.
Consequently there are $r$ pairwise disjoint faces $\sigma_1,\ldots,\sigma_r$ with points $x_1\in\relint\sigma_1,\ldots,x_r\in\relint\sigma_r$ such that $g(x_1) = \cdots = g(x_r)$, that is,
\[
f(x_1) 	 = \cdots =  f(x_r),\
\dist(x_1,\Sigma_1) 	 = \cdots =  \dist(x_r,\Sigma_1),\
	 \cdots, \
\dist(x_1,\Sigma_{d+1}) 	 = \cdots =  \dist(x_r,\Sigma_{d+1}).	
\]
Now observe that for every subcomplex $\Sigma_i$ one of the faces $\sigma_1,\ldots,\sigma_r$ is contained in it.
Indeed, if this would not hold then we would have $|\sigma_1\cap C_i|\geq 2,\ldots,|\sigma_r\cap C_i|\geq 2$, and consequently we would obtain the following contradiction:
\[
2r-1\geq |C_i|\geq |\sigma_1\cap C_i|+\cdots+|\sigma_r\cap C_i|\geq 2r.
\]
Hence the distances, which were previously know to be equal, have to vanish,
\[
\dist(x_1,\Sigma_1) 	 = \cdots =  \dist(x_r,\Sigma_1)=0,\
	 \cdots, \
\dist(x_1,\Sigma_{d+1}) 	 = \cdots =  \dist(x_r,\Sigma_{d+1})=0,
\]
implying that $x_i\in\Sigma_1\cap\cdots\cap\Sigma_{d+1}$ for every $1\leq i\leq r$.
Since $\Sigma_1,\ldots,\Sigma_{d+1}$ are subcomplexes of $\Delta_N$ and $x_1\in\relint\sigma_1,\ldots,x_r\in\relint\sigma_r$ it follows that the faces $\sigma_1, \ldots , \sigma_r$ belong to the subcomplex $\Sigma_1\cap\cdots\cap\Sigma_{d+1}$,
that is, $\sigma_1,\ldots,\sigma_r$ are rainbow faces.
\end{proof}

\noindent
A special case of the weak colored Tverberg theorem we just proved, namely $|C_1|=\cdots=|C_{d+1}|=2r-1$, yields $t(d,r)\leq tt(d,r)\leq 2r-1$ for $r$ a prime power. 
This is the colored Tverberg theorem of \v{Z}ivaljevi\'c and Vre\'cica presented in Theorem \ref{th : colored Tverberg of ZV}.

\medskip
Along the lines of the previous theorem we can prove the following colored Van Kampen--Flores theorem, where the number of color classes is at most $d+1$.
\begin{theorem}[The colored Van Kampen--Flores theorem]
\label{th : colored_tverberg_b}
	Let $d\ge1$ be an integer, let $r$ be a prime power, let $k \ge \lceil d\,\frac{r-1}{r} \rceil + 1$ be an integer, and $N=(d+k+1)(r-1)$.
	Let  $f :  \Delta_N \rightarrow \R^d$ be a continuous map.
  If the vertices of the simplex $\Delta_N$ are colored by $k$ colors, where each color class has cardinality at most $2r-1$,
  then there are $r$ pairwise disjoint rainbow faces $\sigma_1, \dots, \sigma_r$ of $\Delta_N$  whose $f$-images overlap,  
\begin{equation*}
	f(\sigma_1)\cap\cdots\cap f(\sigma_r)\neq\emptyset.
\end{equation*}	
 \end{theorem}

\begin{proof} 
Let $\mathcal{C}$ be the set of vertices of the simplex $\Delta_N$ and let $(C_1,\ldots,C_{k})$ be a coloring where $|C_i|\leq 2r-1$ for all $1\leq i\leq k$.
Such a coloring exists because $k(2r-1)\geq (d+k+1)(r-1)$ is equivalent to our assumption $k \ge \lceil d\,\frac{r-1}{r} \rceil + 1$. 
To each color class $C_i$ we associate the subcomplex $\Sigma_i$ of $\Delta_N$ defined as before by
\[
\Sigma_i:=\{\sigma\in\Delta_N : |\sigma\cap C_i|\leq 1\}.
\]
The subcomplex $\Sigma_1\cap\cdots\cap\Sigma_{k}$ is a subcomplex of all rainbow faces of $\Delta_N$ with respect to the given coloring.
Consider the continuous map $g :  \Delta_N\rightarrow\R^{d+k}$ defined by 
\[
g(x)=(f(x),\dist(x,\Sigma_1),\dist(x,\Sigma_2),\ldots,\dist(x,\Sigma_{k})).
\]
Since $N=(d+k+1)(r-1)$ and $r$ is a prime power the topological Tverberg theorem can be applied to the map $g$.
Therefore, there are $r$ pairwise disjoint faces $\sigma_1,\ldots,\sigma_r$ with points $x_1\in\relint\sigma_1,\ldots,x_r\in\relint\sigma_r$ such that $g(x_1) = \cdots = g(x_r)$, that is,
\[
f(x_1) 	 = \cdots =  f(x_r),\
\dist(x_1,\Sigma_1) 	 = \cdots =  \dist(x_r,\Sigma_1),\
	 \cdots, \
\dist(x_1,\Sigma_{k}) 	 = \cdots =  \dist(x_r,\Sigma_{k}).	
\]
Now observe that every subcomplex $\Sigma_i$ contains one of the faces $\sigma_1,\ldots,\sigma_r$.
Indeed, if this would not hold then $|\sigma_1\cap C_i|\geq 2,\ldots,|\sigma_r\cap C_i|\geq 2$, and we would get the contradiction
\[
2r-1\geq |C_i|\geq |\sigma_1\cap C_i|+\cdots+|\sigma_r\cap C_i|\geq 2r.
\]
Consequently the distances, which were previously known to be equal, have to vanish
\[
\dist(x_1,\Sigma_1) 	 = \cdots =  \dist(x_r,\Sigma_1)=0,\
	 \cdots, \
\dist(x_1,\Sigma_{k}) 	 = \cdots =  \dist(x_r,\Sigma_{k})=0,
\]
implying that $x_i\in\Sigma_1\cap\cdots\cap\Sigma_{k}$ for every $1\leq i\leq r$.
Since $\Sigma_1,\ldots,\Sigma_{k}$ are subcomplexes and $x_1\in\relint\sigma_1,\ldots,x_r\in\relint\sigma_r$, it follows that
\[
\sigma_1\in\Sigma_1\cap\cdots\cap\Sigma_{k},\ \ldots \ , \sigma_r\in\Sigma_1\cap\cdots\cap\Sigma_{k}, 
\] 
that is, $\sigma_1,\ldots,\sigma_r$ are rainbow faces.
\end{proof}

\noindent
The ``colored Tverberg theorem of type B'' of Vre\'cica and \v{Z}ivaljevi\'c, Theorem \ref{th : colored Tverberg of ZV-B}, is 
a particular case of this theorem, when the color classes have the same size.

\subsection{Tverberg points with equal barycentric coordinates}
\label{sec : Tverberg points with equal barycentric coordinate}

The last corollary of the Topological Tverberg theorem that we present here is the topological version  \cite[Thm.\,8.1]{Blagojevic2014} of a recent result by Sober\'on \cite[Thm.\,1.1]{Soberon2013} \cite[Thm.\,1]{Soberon2015}.

Let $N\geq 1$ be an integer, let $\mathcal{C}$ be the set of vertices of the simplex $\Delta_N$, and let 
$(C_1,\ldots,C_{\ell})$ be a coloring of $\mathcal{C}$.
Every point $x$ in the rainbow subcomplex $R_{(C_1,\ldots,C_{\ell})}$ has a unique presentation in barycentric coordinates as $x = \sum_{i=1}^{\ell} \lambda_i^{x}v_i^{x}$ where
$0 \le \lambda_i^x \le 1$ and $v_i^{x}\in C_i$ for all $0\le i\le \ell-1$.
Two points $x = \sum_{i=1}^{\ell} \lambda_i^{x}v_i^{x}$ and $y = \sum_{i=1}^{\ell} \lambda_i^{y}v_i^{y}$ in the rainbow subcomplex $R_{(C_1,\ldots,C_{\ell})}$ have \emph{equal barycentric coordinates} if $ \lambda_i^{x}= \lambda_i^{y}$ for all $1\leq i\leq\ell$.

\begin{theorem}
\label{theorem:equal_coeff}
Let $d \ge 1$ be an integer, let $r$ be a prime power, $N = r((r-1)d+1)-1=(r-1)(rd+1)$, and let $f :  \Delta_N \rightarrow \R^d$ be a continuous map. 
If the vertices of the simplex $\Delta_N$ are colored by $(r-1)d+1$ colors where each colored  class is of size $r$, 
then there are points $x_1, \dots, x_r$ with equal barycentric coordinates that belong to $r$ pairwise disjoint rainbow faces $\sigma_1, \dots, \sigma_r$ of~$\Delta_N$ whose $f$-images coincide, that is 
	\[
	 f(x_1) = \dots = f(x_r).
	\]
\end{theorem}
\begin{proof}
Let $\ell=(r-1)d+1$, and let $(C_1,\ldots,C_{\ell})$ be a coloring of the vertex set $\mathcal{C}=\{v_0,\ldots,v_N\}$ of the simplex $\Delta_N$.
Each point $x$ of the simplex $\Delta_N$ can be uniquely presented in the barycentric coordinates as $x = \sum_{j=0}^N \lambda_j^xv_j$.
For every color class $C_i$, $1\leq i\leq \ell$, we define the function $h_i : \Delta_N\rightarrow\R$ by 
$h_i\big(\sum_{j=0}^N \lambda_j^xv_j\big)= \sum_{v_j \in C_i} \lambda_j^x$.
All functions $h_j$ are affine functions and $\sum_{i=1}^{\ell}h_i(x)=\sum_{j=0}^N\lambda_j^x=1$ for every $x\in\Delta_N$.

\smallskip
Now consider the function $g : \Delta_N\rightarrow\R^{rd}$ given by $g(x)= (f(x),h_1(x),\ldots, h_{\ell-1}(x))$.
Since $N = (r-1)(rd+1)$, the topological Tverberg theorem applied to the function $g$ implies that there exist $r$ pairwise disjoint faces $\sigma_1,\ldots,\sigma_r$ of $\Delta_N$ and $r$ points $x_1\in\relint\sigma_1 , \ldots, x_r \in \relint\sigma_r$ such that  $f(x_1) = \dots = f(x_r)$ and $h_i(x_1) = \dots = h_i(x_r)$ for $1\leq i\leq\ell-1$.
In addition, the equality $\sum_{i=1}^{\ell}h_i(x)=1$ implies that also $h_{\ell}(x_1) = \cdots =h_{\ell}(x_r)$.

\smallskip
Assume now that $|\sigma_j\cap C_i|\geq 1$ for some $1\leq j\leq r$ and some $1\leq i\leq \ell$.
Then $h_i(x_j)>0$ since $x_j\in\relint\sigma_j$.
Consequently, $h_i(x_1) = \dots = h_i(x_r)>0$ implying that $|\sigma_j\cap C_i|\geq 1$ for all  $1\leq j\leq r$. 
Since $|C_i|=r$ and $\sigma_1,\ldots,\sigma_r$ are pairwise disjoint it follows that each $\sigma_j$ has precisely one vertex in the color class $C_i$.
Thus, repeating the argument for each color class we conclude that all faces $\sigma_1,\ldots,\sigma_r$ are rainbow faces.
The immediate consequence of this fact is that $h_i(x_j)$, $1\leq i\leq\ell$, are the barycentric coordinates of the point $x_i$ and so all the points $x_1,\ldots,x_r$ have equal barycentric coordinates. 
\end{proof}

\section{Counterexamples to the topological Tverberg conjecture}
\label{sec : Counterexamples to the topological Tverberg conjecture}

Now we are going to get to a very recent piece of the topological Tverberg puzzle:
We show how counterexamples to the topological Tverberg conjecture for any number of parts that his not a prime power were derived by Frick \cite{Frick2015} \cite{Blagojevic2015} from the remarkable works of \"Ozaydin \cite{Oezaydin1987} and of Mabillard and Wagner \cite{Mabillard2014} \cite{Mabillard2015}, via a lemma of Gromov \cite[p.\,445]{Gromov2010} that is an instance of the constraint method of Blagojevi\'c, Frick, Ziegler \cite[Lemmas\,4.1(iii) and~4.2]{Blagojevic2014}.

\subsection{Existence of equivariant maps if \emph{r} is not a prime power}
\label{subsec : Existence of equivariant maps for $r$ not a prime power}
First we present the second main result of \"Ozaydin's landmark manuscript \cite[Thm.\,4.2]{Oezaydin1987}.

\begin{theorem}
	\label{th : OZaydin}
	Let $d\geq 1$ and $r\geq 6$ be integers, and let $N=(d+1)(r-1)$.
	If $r$ is not a prime power, then there exists an $\Sym_r$-equivariant map
	\begin{equation}
		\label{eq : existence of product map}
		(\Delta_N)^{\times r}_{\Delta(2)}\rightarrow S(W_r^{\oplus d}).
	\end{equation}
\end{theorem}
\begin{proof}
In order to prove the existence of a continuous $\Sym_r$-equivariant map $(\Delta_N)^{\times r}_{\Delta(2)}\rightarrow S(W_r^{\oplus d})$ we again use the equivariant obstruction theory.
Since
\begin{compactitem}
\item
$(\Delta_N)^{\times r}_{\Delta(2)}$ is an $(N-r+1)$-dimensional, $(N-r)$-connected free $\Sym_r$-CW complex, and
\item 
$S(W_r^{\oplus d})$ is a path-connected $(N-r-1)$-connected, $(N-r)$-simple $\Sym_r$-space,
\end{compactitem}
we have that an $\Sym_r$-equivariant map $(\Delta_N)^{\times r}_{\Delta(2)}\rightarrow S(W_r^{\oplus d})$ exists if and only if the \dictionary{primary obstruction} 
\[
[\oo_{\Sym_r}^{N-r+1}(\mathrm{pt})]\in \mathcal{H}_{\Sym_r}^{N-r+1}((\Delta_N)^{\times r}_{\Delta(2)},\pi_{N-r}S(W_r^{\oplus d}))
\]
vanishes.
The obstruction element $[\oo_{\Sym_r}^{N-r+1}(\mathrm{pt})]=[\oo_{\Sym_r}^{N-r+1}(f)]$ does not depend on the particular $\Sym_r$-equivariant map $f :  \sk_{N-r}\big((\Delta_N)^{\times r}_{\Delta(2)}\big)\rightarrow  S(W_r^{\oplus d})$ used to define the obstruction cocycle $\oo_{\Sym_r}^{N-r+1}(f)$.
Thus, in order to prove the existence of an $\Sym_r$-equivariant map \eqref{eq : existence of product map} it suffices to prove that the obstruction element $[\oo_{\Sym_r}^{N-r+1}(f)]$ vanishes for some particular choice of $f$.

\smallskip
Let $p$ be a prime such that $p\,|\,|\Sym_r|=r!$, and let $\Sym_r^{(p)}$ denotes a $p$-Sylow subgroup of $\Sym_r$. 
Since $r$ is not a prime power each $p$-Sylow subgroup of $\Sym_r$ does not act transitively on the set $[r]$, and hence the fixed point set $S(W_r^{\oplus d})^{\Sym_r^{(p)}}\neq\emptyset$ is non-empty.
Thus there exists a (constant) $\Sym_r^{(p)}$-equivariant map $(\Delta_N)^{\times r}_{\Delta(2)}\rightarrow S(W_r^{\oplus d})$, or equivalently the primary obstruction element with respect to $\Sym_r^{(p)}$ vanishes, that is, $[\oo_{\Sym_r^{(p)}}^{N-r+1}(\mathrm{pt})]=[\oo_{\Sym_r^{(p)}}^{N-r+1}(f)]=0$.
Here, the $\Sym_r$-equivariant map $f$ is considered only as an $\Sym_r^{(p)}$-equivariant map.

\smallskip
The \dictionary{restriction} homomorphism 
\[
\res :  
\mathcal{H}_{\Sym_r}^{N-r+1}((\Delta_N)^{\times r}_{\Delta(2)},\pi_{N-r}S(W_r^{\oplus d}))
\rightarrow
\mathcal{H}_{\Sym_r^{(p)}}^{N-r+1}((\Delta_N)^{\times r}_{\Delta(2)},\pi_{N-r}S(W_r^{\oplus d})),
\]
is defined on the cochain level in \cite[Lem.\,5.4]{Blagojevic2012}.
According to the definition of the obstruction cochain (already on the cochain level) the restriction homomorphism sends the obstruction cochain $\oo_{\Sym_r}^{N-r+1}(f)$ to the obstruction cochain $\oo_{\Sym_r^{(p)}}^{N-r+1}(f)$.
Consequently the same hold for obstruction elements
\[
\res([\oo_{\Sym_r}^{N-r+1}(f)])=[\oo_{\Sym_r^{(p)}}^{N-r+1}(f)].
\]
Now, composing the restriction homomorphism with the \dictionary{transfer} homomorphism
\[
\trf : 
\mathcal{H}_{\Sym_r^{(p)}}^{N-r+1}((\Delta_N)^{\times r}_{\Delta(2)},\pi_{N-r}S(W_r^{\oplus d}))
\rightarrow
\mathcal{H}_{\Sym_r}^{N-r+1}((\Delta_N)^{\times r}_{\Delta(2)},\pi_{N-r}S(W_r^{\oplus d})),
\]
also defined on the cochain level in \cite[Lem.\,5.4]{Blagojevic2012}, we get
\[
[\Sym_r : \Sym_r^{(p)}]\cdot [\oo_{\Sym_r}^{N-r+1}(f)]=
\trf\circ\res([\oo_{\Sym_r}^{N-r+1}(f)])=
\trf([\oo_{\Sym_r^{(p)}}^{N-r+1}(f)])=
\trf(0)=0.
\]

\smallskip
Finally, since $[\Sym_r : \Sym_r^{(p)}]\cdot [\oo_{\Sym_r}^{N-r+1}(f)]=0$ for every prime $p$ that divides the order of the group $\Sym_r$, it follows that the obstruction element $[\oo_{\Sym_r}^{N-r+1}(f)]$ must vanish, and the existence of an $\Sym_r$-equivariant map \eqref{eq : existence of product map} is established.
\end{proof}

\begin{corollary}
	\label{cor : after Ozaydin}
	Let $d\geq 1$ be an integer, let $r\geq 6$ be an integer that is not a prime power and let $N=(d+1)(r-1)$.
	For any free $\Sym_r$-CW complex $X$ of dimension at most $N-r+1$ there exists an $\Sym_r$-equivariant map
	\[
		X\rightarrow S(W_r^{\oplus d}).
	\]
\end{corollary}
\begin{proof}
	The free $\Sym_r$-CW complex $X$ has dimension at most $N-r+1$, and the deleted product $(\Delta_N)^{\times r}_{\Delta(2)}$ is $(N-r)$-connected, therefore there are no obstructions for the existence of an  $\Sym_r$-equivariant map $h :  X\rightarrow (\Delta_N)^{\times r}_{\Delta(2)}$.
	Next, let $f :  (\Delta_N)^{\times r}_{\Delta(2)}\rightarrow S(W_r^{\oplus d})$ be an $\Sym_r$-equivariant map whose existence was guaranteed by Theorem \ref{th : OZaydin}.
	The composition $f\circ h :  X\rightarrow S(W_r^{\oplus d})$ yields the required $\Sym_r$-equivariant map.	
\end{proof}

With this theorem \"Ozaydin \emph{only} proved that the ``deleted product approach'' towards solving the topological Tverberg conjecture fails in the case 
that $r$ is not a prime power. 
\emph{What about the ``deleted join approach''?}
This question was discussed in \cite[Sec.\,3.4]{Blagojevic2011-1}.

\begin{theorem}
	Let $d\geq 1$ and $r$ be integers, and let $N=(d+1)(r-1)$.
	If $r$  is not a prime power, then there exists an $\Sym_r$-equivariant map
	\begin{equation}
		\label{eq : existence of join map}
		(\Delta_N)^{* r}_{\Delta(2)}\rightarrow S(W_r^{\oplus (d+1)}).
	\end{equation}
\end{theorem}
\begin{proof}
Since $r\geq 6$ is not a prime power Theorem \ref{th : OZaydin} implies the existence of an $\Sym_r$-equivariant map
\[
		f :  (\Delta_N)^{\times r}_{\Delta(2)}\rightarrow S(W_r^{\oplus d}).
\]
Now an $\Sym_r$-equivariant map
\[
g :  (\Delta_N)^{* r}_{\Delta(2)}\rightarrow S(W_r^{\oplus (d+1)})\cong S(W_r\oplus W_r^{\oplus d})
\]
can be defined by
\[
g(\lambda_1x_1+\cdots+\lambda_rx_r)=\frac1{\nu} \big((\lambda_1-\tfrac1r,\ldots,\lambda_r-\tfrac1r)\oplus
\prod_{i=1}^r\lambda_i\cdot f(x_1,\ldots,x_r)\big),
\]
where $\nu:= \| \big((\lambda_1-\tfrac1r,\ldots,\lambda_r-\tfrac1r)\oplus
\prod_{i=1}^r\lambda_i\cdot f(x_1,\ldots,x_r)\big)\|$.
The function $g$ is well defined, continuous and $\Sym_r$-equivariant. 
Thus an $\Sym_r$-equivariant map \eqref{eq : existence of join map} exists.
\end{proof}

Now we see that not only the ``deleted product approach'' fails if $r$ is not a prime power, but the ``deleted join approach'' fails as well.
\emph{Is this an indication that the topological Tverberg theorem fails if the number of parts is not a prime power?}

\subsection{The topological Tverberg conjecture does not hold if \emph{r} is not a prime power}
\label{subsec : Generalized van Kampen--Flores theorem does not hold for $r$ not a prime power}

It is time to show that the topological Tverberg conjecture fails in the case that $r$ is not a prime power.
This will be done following the presentation given in \cite{Blagojevic2015}.

Based on the work of Mabillard and Wagner \cite{Mabillard2014} \cite{Mabillard2015} we will prove that the generalized Van Kampen--Flores theorem for any $r$ that is not a prime power fails, as demonstrated by Frick \cite{Frick2015} \cite{Blagojevic2015}.
Since, by the constraint method, the generalized Van Kampen--Flores theorem for fixed number of overlaps $r$ is a consequence of the topological Tverberg theorem for the same number of overlaps $r$, failure of the generalized Van Kampen--Flores theorem implies the failure of the topological Tverberg theorem.

\begin{theorem}[The generalized Van Kampen--Flores theorem fails when $r$ is not a prime power]
	\label{thm:counter_van_kampen_flores}
	Let $k \ge 3$ be an integer, and let $r\geq 6$ be an integer that is not a prime power.
	For any integer $N>0$ there exists a continuous map $f :  \Delta_N \rightarrow \R^{rk}$ such that for any $r$ pairwise disjoint faces $\sigma_1, \dots, \sigma_r$ from the $((r-1)k)$-skeleton $\sk_{(r-1)k}(\Delta_N)$ of the simplex $\Delta_N$ the corresponding $f$-images do not overlap, 
		\[
		f(\sigma_1) \cap \dots \cap f(\sigma_r) = \emptyset.
		\]
\end{theorem} 
\begin{proof}
	The deleted product $(\sk_{(r-1)k}(\Delta_N))^{\times r}_{\Delta(2)}$ is a free $\Sym_r$-space of dimension at most $d:=(r-1)rk$.
	Since $r$ is not a power of a prime, according to Corollary \ref{cor : after Ozaydin}, there exists an $\Sym_r$-equivariant map 
\begin{equation}
	\label{eq : map h}
	h :  (\sk_{(r-1)k}(\Delta_N))^{\times r}_{\Delta(2)}\rightarrow S(W_r^{\oplus d}).
\end{equation}
Now we use the following result of Mabillard and Wagner \cite[Thm.\,3]{Mabillard2014} \cite[Thm.\,7]{Mabillard2015}, for which an alternative proof is given in \cite{MabillardWagner-III}.
Skopenkov \cite{Skopenkov-Survey} gives a user's guide.
\begin{quote}
	{\small
	\textbf{Theorem.} 
	\emph{Let $r\ge 2$ and $k \ge 3$ be integers, and let $K$ be an $((r-1)k)$-dimensional simplicial complex. 
	Then the following statements are equivalent:
	\begin{compactenum}[\rm(i)]
		\item There exists a continuous $\Sym_r$-equivariant map $K^{\times r}_{\Delta(2)} \rightarrow S(W_r^{\oplus rk})$.
		\item There exists a continuous map $f  :  K \rightarrow \R^{rk}$ such that for any $r$ pairwise disjoint faces $\sigma_1, \dots, \sigma_r$ of~$K$ we have that
		$f(\sigma_1) \cap \dots \cap f(\sigma_r) = \emptyset$.
	\end{compactenum}}
	}
\end{quote}
If we apply this result to the $\Sym_r$-equivariant map $h$ in \eqref{eq : map h} we get a continuous map $f  :  \sk_{(r-1)k}(\Delta_N) \rightarrow \R^{rk}$ with the property that for any collection of $r$ pairwise disjoint faces $\sigma_1, \dots, \sigma_r$ in $\sk_{(r-1)k}(\Delta_N)$ the corresponding $f$-images do not overlap,
	\[
		f(\sigma_1) \cap \dots \cap f(\sigma_r) = \emptyset.\vspace{-7pt}
	\]
\end{proof}

Thus we have proved that in the case when $r$ is not a prime power the generalized Van Kampen--Flores theorem fails.
As we have pointed out this means that the corresponding topological Tverberg theorem also fails \cite[Thm.\,4.3]{Blagojevic2015}.

\begin{theorem}[The topological Tverberg theorem fails for any $r$ that is not a prime power]
	\label{th : conterexample}
	Let $k \ge 3$ and $r\geq 6$ be integers, and let $N =(r-1)(rk+2)$. 
	If $r$ is not a prime power, then there exists a continuous map $g  :  \Delta_N \rightarrow \R^{rk+1}$ such that for any $r$ pairwise disjoint faces $\sigma_1, \dots, \sigma_r$ of $\Delta_N$ the corresponding $g$ images do not overlap,  
	\[
		g(\sigma_1) \cap \dots \cap g(\sigma_r) = \emptyset.
	\]		
\end{theorem}

\begin{proof}
	Since $r$ is not a power of a prime, Theorem~\ref{thm:counter_van_kampen_flores} yields a continuous map $f :  \Delta_N \rightarrow \R^{rk}$ such that for any $r$ pairwise disjoint faces $\sigma_1, \dots, \sigma_r$ in $\sk_{(r-1)k}\Delta_N$
	\[
		f(\sigma_1) \cap \dots \cap f(\sigma_r) = \emptyset.
	\]
	Motivated by the proof of Theorem \ref{thm : genrelized_van_Kampen_Flores} we consider the function $g :  \Delta_N \rightarrow \R^{rk+1}$ defined by
	\[
	g(x)= (f(x),\dist(x,\sk_{(r-1)k}(\Delta_N))). 
	\]
	We prove that the map $g$ fails the topological Tverberg conjecture.
	
	\smallskip
	Assume, to the contrary, that there are $r$ pairwise disjoint faces $\sigma_1, \dots, \sigma_r$ in $\Delta_N$ and $r$ points 
	\[
	x_1 \in \relint\sigma_1, \ \ldots, \ x_r \in \relint\sigma_r
	\]
	such that $g(x_1) = \cdots = g(x_r)$. 
	Consequently,
	\[
		\dist(x_1,\sk_{(r-1)k}(\Delta_N))=\cdots=\dist(x_r,\sk_{(r-1)k}(\Delta_N)).
	\]
	Next, at least one of the faces $\sigma_1, \dots, \sigma_r$ is in $\sk_{(r-1)k}(\Delta_N)$. 
	Indeed, if all the faces $\sigma_i$ would have dimension at least $(r-1)k+1$, then we would get the following contradiction:
	\[
	N+1 = (r-1)(rk+2)+1= |\Delta_N|\geq |\sigma_1|+\cdots+|\sigma_r|\geq ((r-1)rk+2)=(r-1)(rk+2)+2>N+1.
	\]
	Since one of the faces $\sigma_1, \dots, \sigma_r$ is in $\sk_{(r-1)k}(\Delta_N)$, all the distances vanish, meaning that 
	\[
	\dist(x_1,\sk_{(r-1)k}(\Delta_N))=\cdots=\dist(x_r,\sk_{(r-1)k}(\Delta_N))=0.
	\]
	Therefore, all the faces $\sigma_1, \dots, \sigma_r$ belong to $\sk_{(r-1)k}(\Delta_N)$ contradicting the choice of the map $f$.
	Thus the map $g$ is a counterexample to the topological Tverberg theorem.
\end{proof}

\begin{remark}
The smallest counterexample to the topological Tverberg theorem that can be obtained from Theorem \ref{th : conterexample} is a continuous map $\Delta_{100} \rightarrow \R^{19}$ with the property that no six pairwise disjoint faces in $\Delta_{100}$ have $f$-images that overlap.
	Recently, using additional ideas, Sergey Avvakumov, Isaac Mabillard, Arkadiy Skopenkov and Uli Wagner \cite{MabillardWagner-III} have improved this to get counterexamples $\Delta_{65} \rightarrow \R^{12}$.
\end{remark}

\section{The B\'ar\'any--Larman conjecture and the Optimal colored Tverberg theorem}
\label{sec : The optimal colored Tverberg theorem}

Let us briefly recall the original colored Tverberg problem posed by B\'ar\'any and Larman in their 1992 paper \cite{Barany1992}, see Section \ref{subsec : Colored Tverberg problem BL}.
\begin{problem}[B\'ar\'any--Larman colored Tverberg problem]
Let $d\geq 1$ and $r\geq 2$ be integers.
Determine the smallest number $n=n(d,r)$ such that for every affine (continuous) map $f:\Delta _{n-1}\rightarrow \R^{d}$, and every coloring $(C_1,\ldots,C_{d+1})$ of the vertex set $\mathcal{C}$ of the simplex $\Delta_{n-1}$ by $d+1$ colors with each color of size at least $r$, there exist $r$ pairwise disjoint rainbow faces $\sigma_1, \dots, \sigma_r$ of $\Delta_{n-1}$  whose $f$-images overlap, that is
\begin{equation*}
	f(\sigma_1)\cap\cdots\cap f(\sigma_r)\neq\emptyset.
\end{equation*}	
\end{problem}
\noindent 
A trivial lower bound for the function $n(d,r)$ is $(d+1)r$ and it is natural to conjecture the following.

\begin{conjecture}[B\'{a}r\'{a}ny--Larman Conjecture]
Let $r\ge2$ and $d\ge1$ be integers. 
Then $n(d,r)=(d+1)r$.	
\end{conjecture}

In Section \ref{subsec : Colored Tverberg problem ZV} we tried to use the approach of \v{Z}ivaljevi\'c and Vre\'cica to solve the B\'{a}r\'{a}ny--Larman conjecture and we failed dramatically. 
We hoped to prove that an $\Sym_r$-equivariant map 
	\[
	\Delta_{r,r}^{*(d+1)}\rightarrow S(W_r^{\oplus (d+1)})
	\]
does not exist, but Theorem \ref{th : exitence of equivariant map for BL} gave us exactly the opposite, the existence of this map.
\emph{What can we do now?}
We change the question, and prove the non-existence of an $\Sym_{r+1}$-equivariant map
\[
(R_{(C_1,\ldots,C_{d+2})})^{*r+1}_{\Delta(2)}\cong \Delta_{r,r+1}^{*(d+1)}*[r+1]\rightarrow S(W_{r+1}^{\oplus (d+1)})
\]
instead; here $|C_1|=\cdots=|C_{d+1}|=r$, and $|C_{d+2}|=1$.
\emph{Still, why should we be interested in such a result?}
 
\begin{theorem}
\label{th : non exitence --> BL}
	Let $r\ge2$ and $d\ge1$ be integers. 
	If there is no $\Sym_{r+1}$-equivariant map 
	\[
	\Delta_{r,r+1}^{*(d+1)}*[r+1]\rightarrow S(W_{r+1}^{\oplus (d+1)}),
	\]
	then $n(d,r)=(d+1)r$ and $tt(d,r)=r$.
\end{theorem}

\begin{proof}
Let $(C_1,\ldots,C_{d+1})$ be a coloring of the vertices of the simplex $\Delta$ with $|C_1|=\cdots=|C_{d+1}|=r$, and let $f :  \Delta\rightarrow\R^d$ be a continuous map.
Construct a simplex  $\Delta'$ as a pyramid over $\Delta$, and let $C_{d+2}$ be the additional color class containing only the apex of the pyramid. 
Thus, $(C_1,\ldots,C_{d+2})$ is a coloring of the vertices of the simplex $\Delta'$.

\smallskip
Let us assume that an $\Sym_{r+1}$-equivariant map $\Delta_{r,r+1}^{*(d+1)}*[r+1]\rightarrow S(W_{r+1}^{\oplus (d+1)})$ does not exist.
The non-existence of this map in combination with Corollary \ref{cor : CSTM for colored Tverberg} implies that there exist $r+1$ pairwise disjoint rainbow faces $\sigma_1, \dots, \sigma_{r+1}$ of the simplex $\Delta'$  whose $f$-images overlap,  $f(\sigma_1)\cap\cdots\cap f(\sigma_{r+1})\neq\emptyset$.

\smallskip
Without loss of generality we can assume that $\sigma_{r+1}\cap C_{d+2}\neq\emptyset$. 
Then the faces  $\sigma_1, \dots, \sigma_{r}$ are rainbow faces of the simplex $\Delta$ with respect to the coloring $(C_1,\ldots,C_{d+1})$ and
\[
f(\sigma_1)\cap\cdots\cap f(\sigma_{r})\neq\emptyset.
\]
Hence, $tt(d,r)=r$ and consequently $n(d,r)=(d+1)r$.
\end{proof}

\noindent
The theorem we just proved tells us that in order to make an advance on the B\'ar\'any--Larman conjecture we should try to prove the non-existence of a continuous $\Sym_{r+1}$-equivariant map 
\begin{equation}
\label{eq : 123 map}
\Delta_{r,r+1}^{*(d+1)}*[r+1]\rightarrow S(W_{r+1}^{\oplus (d+1)})	
\end{equation}
at least for some values of $r$.

\smallskip
Now in order to prove the non-existence of a continuous $\Sym_{r+1}$-equivariant map \eqref{eq : 123 map}, for $r+1=:p$ an odd prime, we will compute the Fadell--Husseini index of the join $\Delta_{r,r+1}^{*(d+1)}*[r+1]=\Delta_{p-1,p}^{*(d+1)}*[p]$ with respect to the cyclic group and compare the result with the index of the sphere $S(W_{r+1}^{\oplus (d+1)})$.

\subsection{The Fadell--Husseini index of chessboards}

Let $p:=r+1$ be an odd prime. 
In this section we compute the Fadell--Husseini index of chessboards
\[
\ind_{\Z/p}(\Delta_{k,p};\F_p)\subseteq H^*(\B(\Z/p);\F_p)=H^*(\Z/p;\F_p),\qquad\qquad k\geq1,
\]
and their joins with respect to the cyclic subgroup $\Z/p$ of the symmetric group $\Sym_p$.
Recall that 
\[
H^*(\B(\Z/p);\F_p)=H^*(\Z/p;\F_p)= \F_p[t]\otimes\Lambda[e],
\]
where $\deg t=2$, $\deg e=1$, and $\Lambda[\,\cdot\,]$ denotes the exterior algebra.
First, we collect some simple facts about the Fadell--Husseini index of chessboards.

\begin{lemma}
	\label{lem : FH index of chessboards}
	Let $k\geq 1$ be an integer, and let $p$ be an odd prime. 
	Then
	\begin{compactenum}[\rm (i)]
	\item $\ind_{\Z/p}\Delta_{1,p}=H^{\geq 1}(\B(\Z/p);\F_p)$,
	\item $\ind_{\Z/p}\Delta_{2p-1,p}=H^{\geq p}(\B(\Z/p);\F_p)$,
	\item $\ind_{\Z/p}\Delta_{1,p}\subseteq \ind_{\Z/p}\Delta_{2,p}\subseteq \cdots \subseteq\ind_{\Z/p}\Delta_{2p-1,p}= \ind_{\Z/p}\Delta_{2p,p}=\cdots=H^{\geq p}(\B(\Z/p);\F_p)$.
	\end{compactenum}
\end{lemma}

\begin{proof}
For the statement (i) observe that $\Delta_{1,p}=[p]$ and therefore $\E(\Z/p)\times_{\Z/p} \Delta_{1,p} \cong \E(\Z/p)$.
Since $\E(\Z/p)$ is a contractible space,
	\[
	\ind_{\Z/p}\Delta_{1,p}=\ker \big(H^*(\B(\Z/p);\F_p)\rightarrow H^*(\E(\Z/p);\F_p)\big)=H^{\geq 1}(\B(\Z/p);\F_p).
	\]

\smallskip	
In order to prove (ii) recall that $\Delta_{2p-1,p}$ is a $(p-1)$-dimensional $(p-2)$-connected free $\Z/p$ simplicial complex, see Theorem \ref{th : conn. of chessboard}. 	
The Serre spectral sequence associated to the Borel construction fiber bundle $\Delta_{2p-1,p}\rightarrow \E(\Z/p)\times_{\Z/p} \Delta_{2p-1,p}\rightarrow\B(\Z/p)$ has the $E_2$-term given by
\begin{multline*}
	E^{i,j}_2   =  H^i(\B(\Z/p); \mathcal{H}^j(\Delta_{2p-1,p};\F_p))= H^i(\Z/p; H^j(\Delta_{2p-1,p};\F_p))\\=
	\begin{cases}
		H^i(\Z/p;\F_p), & \text{for }j=0,\\
		H^i(\Z/p; H^{p-1}(\Delta_{2p-1,p};\F_p)), & \text{for }j=p-1,\\
		0,  &\text{otherwise}.
	\end{cases}
\end{multline*}
Thus, $E^{i,0}_{\infty}\cong E^{i,0}_2 \cong H^i(\Z/p;\F_p)$ for $0\leq i\leq p-1$.
Consequently $\ind_{\Z/p}\Delta_{2p-1,p}\subseteq H^{\geq p}(\B(\Z/p);\F_p)$.
Since $\Delta_{2p-1,p}$ is a free $\Z/p$ simplicial complex, we get $\E(\Z/p)\times_{\Z/p} \Delta_{2p-1,p}\simeq  \Delta_{2p-1,p}/\Z/p$, implying that $H^i(\E(\Z/p)\times_{\Z/p}\Delta_{2p-1,p};\F_p)=0$ for $i\geq p$.
Since the spectral sequence $E^{*,*}_*$ converges to the cohomology of the Borel construction  $H^*(\E(\Z/p)\times_{\Z/p}\Delta_{2p-1,p};\F_p)$, we have that
$E^{i,j}_{\infty}=0$ for $i+j\geq p$.
In particular, $E^{i,0}_{\infty}=0$ for $i\geq p$, implying that
\[
 \ind_{\Z/p}\Delta_{2p-1,p}= H^{\geq p}(\B(\Z/p);\F_p).
\]

\smallskip	
For (iii) observe that there is a sequence of $\Z/p$-equivariant inclusions
\[
\Delta_{1,p}\xhookrightarrow{\quad} \Delta_{2,p} \xhookrightarrow{\quad} \cdots \xhookrightarrow{\quad} \Delta_{k,p}  \xhookrightarrow{\quad}  \Delta_{k+1,p}   \xhookrightarrow{\quad}\cdots
\]
given by the inclusions of the corresponding vertex sets
\[
[1]\times [p]\xhookrightarrow{\quad} [2]\times [p] \xhookrightarrow{\quad} \cdots \xhookrightarrow{\quad} [k]\times [p]  \xhookrightarrow{\quad}  [k+1]\times [p]   \xhookrightarrow{\quad}\cdots .
\]
Consequently the monotonicity property of the Fadell--Husseini index, combined with the fact that for $k\geq 2p-1$ all chessboards $\Delta_{k,p}$ are $(p-1)$-dimensional free $\Z/p$ simplicial complexes, implies that
\[
\ind_{\Z/p}\Delta_{1,p}\subseteq \ind_{\Z/p}\Delta_{2,p}\subseteq \cdots \subseteq\ind_{\Z/p}\Delta_{2p-1,p}= \ind_{\Z/p}\Delta_{2p,p}=\cdots=H^{\geq p}(\B(\Z/p);\F_p).
\]
\end{proof}
 
In the next step we compute the index of the chessboard $\Delta_{p-1,p}$.
For that we need to establish the following fact.

\begin{lemma}
\label{lem : map f}
	Let $p$ be an odd prime. 
	There exists a $\Z/p$-equivariant map $f : \Delta_{p-1,p}\rightarrow S(W_p)$ such that the induced map in cohomology $f^* :  H^{p-2}(S(W_p);\F_p)\rightarrow H^{p-2}(\Delta_{p-1,p};\F_p)$ is an isomorphism.
\end{lemma}
\begin{proof}
	Let $\ee_1,\ldots,\ee_p$ be a standard basis of $\R^p$, let $\ee:=\tfrac1p(\ee_1+\cdots+\ee_p)$, and let $v_i:=\ee_i-\ee$ for $1\leq i\leq p$.
	Denote now by $\Delta_{p-1}\subseteq W_p$ the simplex $\conv\{v_1,\ldots,v_p\}$, which is invariant with respect to the action of the cyclic group $\Z/p$.
	Moreover, its boundary $\partial\Delta_{p-1}$ is equivariantly homeomorphic to the representation sphere $S(W_p)$.
	
	\smallskip
	We define a continuous map $f : \Delta_{p-1,p}\rightarrow \partial\Delta_{p-1}\cong S(W_p)$ to be the $\Z/p$-equivariant simplicial map given on the vertex set of $\Delta_{p-1,p}$ by $(i,j) \longmapsto v_j$, where $(i,j)\in [p-1]\times [p]$.
	It remains to be verified that $f^* :  H^{p-2}(S(W_p);\F_p)\rightarrow H^{p-2}(\Delta_{p-1,p};\F_p)$ is an isomorphism.
	
	\smallskip
	Since $p\geq 3$, the chessboard complex $\Delta_{p-1,p}$ is a connected, orientable pseudomanifold of dimension $p-2$, for this see \cite[p.\,145]{Jonsson2008}.
	Thus $H_{p-2}(\Delta_{p-1,p};\Z)=\Z$ and an orientation class is given by the chain
	\[
	z_{p-1,p}\ =\ \sum_{\pi\in\Sym_{p}}(\sgn \pi) \langle(1,\pi(1)),\dots,(p-1,\pi(p-1))\rangle.
	\]
	Then on the chain level we have that
	\begin{eqnarray*}
		f_{\#}(z_{p-1,p}) & = & \sum_{\pi\in\Sym_{p}}(\sgn \pi) \langle v_{\pi(1)},\dots, v_{\pi(p-1)}\rangle = \sum_{\pi\in\Sym_{p}}(\sgn \pi) \langle v_{\pi(1)},\dots, v_{\pi(p-1)}, \widehat{v_{\pi(p)}} \rangle\\
		 & = &\sum_{k=1}^{p} \sum_{\pi\in\Sym_{p} : \pi(p)=k} (-1)^{p+k}(\sgn \pi)^2\langle v_1,\ldots,\widehat{v_k},\ldots,v_p\rangle \\
		& = & \sum_{k=1}^{p} (-1)^{k-1} \sum_{\pi\in\Sym_{p} : \pi(p)=k} \langle v_1,\ldots,\widehat{v_k},\ldots,v_p\rangle 
		 =  \sum_{k=1}^{p} (-1)^{k-1} (p-1)! \langle v_1,\ldots,\widehat{v_k},\ldots,v_p\rangle\\
		& = &  (p-1)! \sum_{k=1}^{p} (-1)^{k-1}\langle v_1,\ldots,\widehat{v_k},\ldots,v_p\rangle.
	\end{eqnarray*}
	For this calculation keep in mind that $p$ is an odd prime.
 	The chain $\sum_{k=1}^{p} (-1)^{k-1}\langle v_1,\ldots,\widehat{v_k},\ldots,v_p\rangle$ is a generator of the top homology of the sphere $\partial\Delta_{p-1}\cong S(W_p)$.
 	Therefore, the induced map in homology 
 	\[
 	f_{*}  :  H_{p-2}(\Delta_{p-1,p};\Z) \rightarrow H_{p-2} (S(W_p);\Z)
 	\]
 	is just a multiplication by $(p-1)!\equiv -1\, (\mathrm{mod}\ p)$.
 	Using the naturality of the universal coefficient isomorphism \cite[Cor.\,7.5]{Bredon2010} we have that the induced map in homology with $\F_p$ field coefficients
 		\[
 		f_{*}  :  H_{p-2}(\Delta_{p-1,p};\F_p) \rightarrow H_{p-2} (S(W_p);\F_p)
 		\]
 	is again multiplication by $(p-1)!$.
 	Since $(p-1)!$ and $p$ are relatively prime the multiplication by $(p-1)!$ is an isomorphism.
 	Now using yet another universal coefficient isomorphism \cite[Cor.\,7.2]{Bredon2010} for the coefficients in a field we get that the induced map in cohomology with $\F_p$ coefficients
 		\[
 		f^{*}  :   H^{p-2} (S(W_p);\F_p) \rightarrow H^{p-2}(\Delta_{p-1,p};\F_p) 
 		\]
 	is an isomorphism.	
 	\end{proof}

Now we have all ingredients needed to compute the index of the chessboard $\Delta_{p-1,p}$.

\begin{theorem}
	\label{th : index of chessboard p-1}
	$\ind_{\Z/p}\Delta_{p-1,p}=\ind_{\Z/p}S(W_p)=H^{\geq p-1}(\B(\Z/p);\F_p)$.
\end{theorem}
\begin{proof}
Let us denote by $\lambda$ the Borel construction fiber bundle
\[
\lambda\quad : \quad
\Delta_{p-1,p}\rightarrow \E (\Z/p)\times_{\Z/p} \Delta_{p-1,p}\rightarrow \B (\Z/p),
\]
and by $\rho$ the Borel construction fiber bundle
\[
\rho\quad : \quad
S(W_p)\rightarrow \E (\Z/p)^n\times_{\Z/p} S(W_p) \rightarrow \B (\Z/p).
\]
The $\Z/p$-equivariant map $f : \Delta_{p-1,p}\rightarrow S(W_p)$ constructed in Lemma \ref{lem : map f} induces a morphism of the Borel construction fiber bundles $\lambda$ and $\rho$:
\[
\xymatrix{
\E(\Z/p)\times_{\Z/p} \Delta_{p-1,p}\ar[rr]^-{\id\times_{\Z/p} f}\ar[d] &  &\E(\Z/p)\times_{\Z/p} S(W_p)\ar[d]\\
\B(\Z/p) \ar[rr]^-{=}           &  &\B(\Z/p).
}
\]
This morphism induces a morphism of the corresponding Serre spectral sequences
\[
E^{i,j}_s (\lambda):=E^{i,j}_s (\E (\Z/p)\times_{\Z/p} \Delta_{p-1,p})\overset{f^{i,j}_s}{\longleftarrow}  E^{i,j}_s(\E(\Z/p)\times_{\Z/p} S(W_p))=:E^{i,j}_s (\rho)
\]
with the property that on the zero row of the second term the induced map 
\[
E^{i,0}_2 (\lambda)=E^{i,0}_2 (\E (\Z/p)\times_{\Z/p} \Delta_{p-1,p})\overset{f^{i,0}_2}{\longleftarrow}  E^{i,0}_2(\E(\Z/p)\times_{\Z/p} S(W_p))=E^{i,0}_2 (\rho)
\]
is the identity.
Here we use simplified notation $f^{i,j}_s:=E^{i,j}_s(\id\times_{\Z/p} f)$.
In the $E_2$-term, since the homomorphism $f^{*}  :   H^{p-2} (S(W_p);\F_p) \rightarrow H^{p-2}(\Delta_{p-1,p};\F_p)$ induces an isomorphism on the $(p-2)$-cohomology, and $\Z/p$ acts trivially on both cohomologies $H^{p-2} (S(W_p);\F_p) \cong H^{p-2}(\Delta_{p-1,p};\F_p)\cong \F_p$, the morphism of spectral sequences 
\begin{equation}
	\label{iso of p-2 row}
	f^{i,p-2}_2  :   E^{i,p-2}_2(\rho) \rightarrow E^{i,p-2}_2 (\lambda)
\end{equation}
is an isomorphism. 

\smallskip
The $E_2$-term of the Serre spectral sequence associated to the fiber bundle $\rho$ is given by
\[
E^{i,j}_2(\rho)   =  H^i(\B(\Z/p); \mathcal{H}^j(S(W_p) ;\F_p)) = H^i(\Z/p; H^j(S(W_p) ;\F_p))\cong  H^i(\Z/p;\F_p)\otimes_{\F_p} H^j(S(W_p) ;\F_p),
\]
because $\Z/p$ acts trivially on the cohomology $H^*(S(W_p) ;\F_p)$.
Thus the only possible non-trivial differential is 
\[
\partial_{p-1} :  E^{i,p-2}_2(\rho)\cong E^{i,p-2}_{p-1}(\rho)\rightarrow E^{i+p-1,0}_2(\rho)\cong E^{i+p-1,0}_{p-1}(\rho).
\]
Let $\ell\in H^{p-2}(S(W_p) ;\F_p)$ denote a generator.
Then the $(p-2)$-row of the $E_2$-term, as an $H^*(\Z/p;\F_p)$-module, is generated by $1\otimes_{\F_p}\ell\in E^{0,p-2}_2(\rho)$.
Since the differentials are $H^*(\Z/p;\F_p)$-module maps it follows that the differential $\partial_{p-1}$ is completely determined by its image $\partial_{p-1}(1\otimes_{\F_p}\ell)\in E^{p-1,0}_{p-1}(\rho)\cong E^{p-1,0}_2(\rho)$.
In order to find the image of the differential notice that $\Z/p$ acts freely on the sphere $S(W_p)$ and consequently $\E(\Z/p)\times_{\Z/p} S(W_p)\simeq S(W_p)/\Z/p$.
Since the spectral sequence $E^{i,j}_s$ converges to the cohomology $H^*(\E(\Z/p)\times_{\Z/p} S(W_p);\F_p)$ we have that $E^{i,j}_{\infty} (\rho)\cong E^{i,j}_p(\rho) = 0$ for $i+j\geq p-1$.
Thus, 
\[
\partial_{p-1}(1\otimes_{\F_p}\ell)=\omega\cdot t^{(p-1)/2} \neq 0
\]
for some $\omega\in\F_p{\setminus}\{0\}$.
Moreover,
\[
\ind_{\Z/p}S(W_p)=\big\langle t^{(p-1)/2} \big\rangle =H^{\geq p-1}(\B(\Z/p);\F_p).
\]

\smallskip
The $E_2$-term of the Serre spectral sequence associated to the fiber bundle $\lambda$ is given by
\[
	E^{i,j}_2(\lambda)   =  H^i(\B(\Z/p); \mathcal{H}^j(\Delta_{p-1,p};\F_p))=H^i(\Z/p; H^j(\Delta_{p-1,p};\F_p)).
\]
In particular, $E^{i,0}_2(\lambda)\cong H^i(\Z/p;\F_p)$ and $E^{i,p-2}_2(\lambda)\cong H^i(\Z/p;\F_p)$, because $\Z/p$ acts trivially on the cohomology $H^{p-2}(\Delta_{p-1,p};\F_p)\cong \F_p$.
Let $z:=f^{0,p-2}_2(1\otimes_{\F_p} \ell)$.
As we have seen in \eqref{iso of p-2 row} the map $f^{0,p-2}_2$ is an isomorphism. 
Thus $z$ is a generator of $E^{0,p-2}_2(\lambda) \cong\F_p$, and moreover $z$ is a generator of the $(p-2)$-row of the $E_2$-term as an $H^*(\Z/p;\F_p)$-module.
As in the case of the spectral sequence $E^{i,j}_s(\rho)$ the fact that $\Z/p$-acts freely on the chessboard $\Delta_{p-1,p}$ implies that $E^{i,j}_{\infty}(\lambda) \cong E^{i,j}_p(\lambda) = 0$ for $i+j\geq p-1$.

\smallskip
Since $f^{i,j}_s$ is a morphism of spectral sequences it has to commute with the differentials.
In particular, for $2\leq s\leq p-2$ we have
\[
\partial_s (z)=\partial_s (f^{0,p-2}_s (1\otimes_{\F_p} \ell))=f^{s,p-s-1}_s (\partial_s (1\otimes_{\F_p} \ell)) = 0.
\]
Now the fact that $z$ is a generator of the $(p-2)$-row of the $E_2$-term as an $H^*(\Z/p;\F_p)$-module yields 
\[
E^{i,p-2}_{p-1}(\lambda)\cong E^{i,p-2}_2(\lambda)\cong H^i(\Z/p;\F_p).
\]
If in addition $\partial_{p-1} (z)=0$, then for every $i\geq 0$ 
\[
E^{i,p-2}_{p}(\lambda)\cong E^{i,p-2}_{p-1}(\lambda)\cong E^{i,p-2}_2(\lambda)\cong H^i(\Z/p;\F_p)\neq 0,
\]
which contradicts the fact that $E^{i,j}_{\infty}(\lambda) \cong E^{i,j}_p(\lambda) = 0$ for $i+j\geq p-1$.
In summary we have that
\[
\partial_{p-1} (z)=\partial_{p-1} (f^{0,p-2}_s (1\otimes_{\F_p} \ell))=f^{p-1,0}_{p-1} (\partial_{p-1} (1\otimes_{\F_p} \ell))=f^{p-1,o}_{p-1}(\omega\cdot t^{(p-1)/2})=\omega\cdot f^{p-1,0}_{p-1}(t^{(p-1)/2})\neq 0.
\]
Moreover, we have that 
\[
\partial_{p-1}  :  E^{i,p-2}_{p-1}(\lambda)\rightarrow E^{i+p-1,0}_{p-1}(\lambda) 
\]
must be an isomorphism for every $i\geq 0$.
Hence, for $i\geq 0$ we have that 
\begin{equation}
	\label{eq : 11}
	E^{i+p-1,0}_{p-1}(\lambda)\cong E^{i+p-1,0}_{2}(\lambda)\cong H^{i+p-1}(\Z/p;\F_p)\cong \F_p.
\end{equation}
Since, $f^{p-1,0}_{p-1}(t^{(p-1)/2})\neq 0$ and $f^{p-1,0}_{2}$ is the identity map we conclude that  $f^{p-1,0}_{p-1}(t^{(p-1)/2})=t^{(p-1)/2}$ and consequently, 
\[
\ind_{\Z/p}\Delta_{p-1,p}\subseteq \big\langle t^{(p-1)/2} \big\rangle =H^{\geq p-1}(\B(\Z/p);\F_p).
\]

Finally we claim that no non-zero differential can arrive to the $0$-row on $E_s$-term for $2\leq s\leq p-2$, implying that
\[
\ind_{\Z/p}\Delta_{p-1,p}=\big\langle t^{(p-1)/2} \big\rangle =H^{\geq p-1}(\B(\Z/p);\F_p),
\]
and concluding the proof of the theorem.
Indeed, if this is not true, then there exists a minimal $s$ such that $2\leq s\leq p-2$ and $0\neq \partial_s(y)= t^ae^b\in E^{i,0}_s(\lambda)$ for some $y$ and $0\leq i\leq p-2$.
Since differentials are $H^*(\Z/p;\F_p)$-module maps we have that $\partial_s(t^c\cdot y)= t^c\cdot\partial_s( y)= t^{a+c} e^b\in E^{i+2c,0}_s(\lambda)$ for every $c\geq 0$.
Consequently, $E^{i+2c,0}_{s+1}(\lambda)=0$ for every $c\geq 0$ contradicting the existence of the isomorphisms \eqref{eq : 11}.
Thus, no non-zero differential can arrive to the $0$-row before the $E_{p-1}$-term.
\end{proof}

\noindent
The proof of the previous theorem, combined together with the fact that a join of pseudomanifolds is a pseudomanifold, yields the following corollary \cite[Cor.\,2.6]{Blagojevic2011-2}.

\begin{corollary}
		\label{cor : index of join of chessboard p-1}
		Let $m\geq1$ be an integer.
		Then
		\begin{compactenum}[\rm (i)]
		\item $\ind_{\Z/p}\Delta_{p-1,p}^{* m}=\ind_{\Z/p}S(W_p^{\oplus m})=H^{\geq m(p-1)}(\B(\Z/p);\F_p)$,
		\item $\ind_{\Z/p}(\Delta_{p-1,p}^{* m}*[p])=\ind_{\Z/p}(S(W_p^{\oplus m})*[p])=H^{\geq m(p-1)+1}(\B(\Z/p);\F_p)$,
		\item $\ind_{\Z/p}(\Delta_{p-1,p}^{* m}*\Delta_{2p-1,p})=\ind_{\Z/p}(S(W_p^{\oplus m})*[p]^{*p-1})=H^{\geq m(p-1)+p}(\B(\Z/p);\F_p)$.
		\end{compactenum}
\end{corollary}

In the next step we compute  the index of the chessboard $\Delta_{k,p}$ for $1\leq k\leq p-2$.

\begin{theorem}
	\label{th : index of chessboard k < p-2}
	$\ind_{\Z/p}\Delta_{k,p}=H^{\geq k}(\B(\Z/p);\F_p)$, for $1\leq k\leq p-1$. 
\end{theorem}
\begin{proof}
	Let $1\leq k\leq p-2$ be an integer.
	The chessboard $\Delta_{k,p}$ is a $(k-1)$-dimensional free $\Z/p$ simplicial complex.
	Thus $\E(\Z/p)\times_{\Z/p}\Delta_{k,p}\simeq \Delta_{k,p}/\Z/p$ and consequently $H^i(\E(\Z/p)\times_{\Z/p}\Delta_{k,p};\F_p)=0$ for all $i\geq k$.
	Therefore, $\ind_{\Z/p}\Delta_{k,p}\supseteq H^{\geq k}(\B(\Z/p);\F_p)$.
	For $k=1$ the theorem follows from Lemma~\ref{lem : FH index of chessboards}(i).
	Furthermore, for $k=p-1$ the statement is the content of Theorem \ref{th : index of chessboard p-1}.
	
	\smallskip
	Now let us assume that $2\leq k\leq p-3$ is even.
	Then $p-1-k$ is also even.
	Now consider the $\Z/p$-equivariant inclusion map
	\[
	\Delta_{p-1,p}\rightarrow \Delta_{k,p} * \Delta_{p-1-k,p}.
	\]
	From the monotonicity and join properties of the Fadell--Husseini index we have that
	\[
	\ind_{\Z/p}\Delta_{k,p}\cdot \ind_{\Z/p}\Delta_{p-1-k,p}\subseteq
	\ind_{\Z/p}(\Delta_{k,p} * \Delta_{p-1-k,p})\subseteq
	\ind_{\Z/p}\Delta_{p-1,p} .
	\]
	Since $p-1-k$ is even and, as we have seen, 
	\[
	\ind_{\Z/p}\Delta_{p-1-k,p}\supseteq H^{\geq p-1-k}(\B(\Z/p);\F_p)=\langle t^{(p-1-k)/2}\rangle
	\]
	we have that $t^{(p-1-k)/2}\in \ind_{\Z/p}\Delta_{p-1-k,p}$.
	On the other hand, assume that there is an element $u\in \ind_{\Z/p}\Delta_{k,p}$ such that $\deg(u)\leq k-1$.
	Then we have reached a contradiction
	\[
	0\neq u\cdot t^{(p-1-k)/2}\in \ind_{\Z/p}\Delta_{k,p}\cdot \ind_{\Z/p}\Delta_{p-1-k,p}\subseteq
	\ind_{\Z/p}\Delta_{p-1,p}=H^{\geq p-1}(\B(\Z/p);\F_p),
	\] 
	because $\deg(u\cdot t^{(p-1-k)/2})=\deg(u)+\deg (t^{(p-1-k)/2})= \deg(u)+ p-1-k \leq p-2$.
	Thus we have proved that for even $k$
	\[
	 \ind_{\Z/p}\Delta_{k,p}=H^{\geq k}(\B(\Z/p);\F_p).
	\]
	
	\smallskip
	Next let us assume that $3\leq k\leq p-2$ is odd. 
	As we observed at the start of the proof
	\[
	\ind_{\Z/p}\Delta_{k,p}\supseteq H^{\geq k}(\B(\Z/p);\F_p)=\langle t^{(k-1)/2} e, t^{(k+1)/2} \rangle.
	\]
	The $\Z/p$-equivariant inclusion map $\Delta_{k-1,p}\subseteq\Delta_{k,p}$ together with the computation of the index for even integers implies that
	\[
	\langle t^{(k-1)/2} \rangle = H^{\geq k-1}(\B(\Z/p);\F_p)=\ind_{\Z/p}\Delta_{k-1,p}\supseteq \ind_{\Z/p}\Delta_{k,p}\supseteq H^{\geq k}(\B(\Z/p);\F_p).
	\]
	In order to conclude the proof of the theorem it remains to prove that $t^{(k-1)/2}\notin \ind_{\Z/p}\Delta_{k,p}$.
	This would yield the equality
	\[
	 \ind_{\Z/p}\Delta_{k,p}=H^{\geq k}(\B(\Z/p);\F_p)
	\]
	for all odd $k$.
	Indeed, assume the opposite, that is, $t^{(k-1)/2}\in \ind_{\Z/p}\Delta_{k,p}$.
	The $\Z/p$-equivariant inclusion $\Delta_{k+1,p}\subseteq \Delta_{1,p} * \Delta_{k,p}$ combined with the monotonicity and join properties of the Fadell--Husseini index imply that
	\[
	\ind_{\Z/p} \Delta_{1,p}\cdot \ind_{\Z/p} \Delta_{k,p}\subseteq \ind_{\Z/p}(\Delta_{1,p} * \Delta_{k,p})\subseteq \ind_{\Z/p}\Delta_{k+1,p}.
	\]
	Since $e\in\ind_{\Z/p} \Delta_{1,p}$,  and we have assumed that $t^{(k-1)/2}\in \ind_{\Z/p}\Delta_{k,p}$, the previous relation implies that
	\[
	t^{(k-1)/2}e\in \ind_{\Z/p}\Delta_{k+1,p}= H^{\geq k+1}(\B(\Z/p);\F_p)=\langle t^{(k+1)/2} \rangle,
	\]
	a contradiction. 
	Hence $t^{(k-1)/2}\notin \ind_{\Z/p}\Delta_{k,p}$, and the proof of the theorem is complete.
\end{proof}

Let us review the results on the Fadell--Husseini index of chessboards we have obtained so far:
{\footnotesize  
\[
\xymatrix@1{
\ind_{\Z/p}\Delta_{1,p}\ar@{<->}[d]_{=} & \ind_{\Z/p}\Delta_{2,p}\ar[l]_-{\supseteq}\ar@{<->}[d]_{=} & \ldots\ar[l]_-{\supseteq}& \ind_{\Z/p}\Delta_{p-1,p}\ar[l]_-{\supseteq}\ar@{<->}[d]_{=}& \ldots\ar[l]_-{\supseteq} & \ind_{\Z/p}\Delta_{2p-1,p}\ar@{<->}[d]_{=}\ar[l]_-{\supseteq}&\ar[l]_-{=}\\
H^{\geq 1}(\B(\Z/p);\F_p)& H^{\geq 2}(\B(\Z/p);\F_p)\ar[l]_-{\supseteq} & \ldots\ar[l]_-{\supseteq}&H^{\geq p-1}(\B(\Z/p);\F_p)\ar[l]_-{\supseteq} & \ldots\ar[l]_-{\supseteq}&H^{\geq p}(\B(\Z/p);\F_p)\ar[l]_-{\supseteq}&\ar[l]_-{=}
}
\]}
The remaining question indicated by this diagram is: \emph{For which chessboard $\Delta_{k,p}$ with $p-1\leq k\leq 2p-1$ does the first jump in the index $H^{\geq p-1}(\B(\Z/p);\F_p)$ to $H^{\geq p}(\B(\Z/p);\F_p)$ happen?}

\begin{theorem}
	\label{th : index of chessboard k > p-1}
	$\ind_{\Z/p}\Delta_{k,p}=H^{\geq p-1}(\B(\Z/p);\F_p)$, for $p-1\leq k\leq 2p-2$.
\end{theorem}
\begin{proof}
	It suffices to show that $\ind_{\Z/p}\Delta_{2p-2,p}=H^{\geq p-1}(\B(\Z/p);\F_p)$.
	For this we are going to prove that $t^{(p-1)/2}\in \ind_{\Z/p}\Delta_{2p-2,p}$.
	
	\smallskip
	Consider the following composition of maps
	\begin{multline*}
		\Delta_{2p-2,p} \rightarrow \Delta_{p-1,p} *_{\Delta(2)} \Delta_{p-1,p} \overset{f * f} {\rightarrow} \partial\Delta_{p-1}*_{\Delta(2)}\partial\Delta_{p-1}\rightarrow\\
		\{\lambda x+(1-\lambda)y \in  S^{p-2} * S^{p-2} : \lambda\neq \tfrac12 \ \text{ or } \ x\neq y\}\rightarrow S^{p-2}\cong S(W_p),	
	\end{multline*}
	where the first map is an inclusion, the second map is the $2$-fold join of the
	map $f : \Delta_{p-1,p}\rightarrow \partial\Delta_{p-1}$, $\Delta_{p-1}\cong S(W_p)$ introduced in Lemma \ref{lem : map f}, the third map is again an inclusion, while the last map is a deformation retraction.
	All the maps in this composition are $\Z/p$-equivariant.
 	The monotonicity property of the Fadell--Husseini index implies that
 	\[
 	\ind_{\Z/p}\Delta_{2p-2,p}\supseteq \ind_{\Z/p}S(W_p) =\langle t^{(p-1)/2} \rangle = H^{\geq p-1}(\B(\Z/p);\F_p),
 	\]
 	according to \eqref{eq : index of the sphere S(W_r)}.
 	Thus $t^{(p-1)/2}\in \ind_{\Z/p}\Delta_{2p-2,p}$, and we have concluded the proof of the theorem.
\end{proof}

\noindent
Now we have the answer to our question. 
The jump happens in the last possible moment, that is for the index of $\Delta_{2p-1,p}$.
The proof of this is due to Carsten Schultz.

\medskip
We conclude the section with a very useful corollary \cite[Cor.\,2.6]{Blagojevic2011-2}, which also hides a proof for the upcoming optimal colored Tverberg theorem \ref{th : optimal colored Tverberg}. 
\begin{corollary}
	\label{cor : index of join}
	Let $1\leq k_1,\ldots,k_n\leq p-1$.
	Then
	\[
	\ind_{\Z/p}(\Delta_{k_1,p}*\cdots *\Delta_{k_n,p})=H^{\geq k_1+\cdots +k_n}(\B(\Z/p);\F_p).
	\]
\end{corollary}
\begin{proof}
	Let $K:=\Delta_{k_1,p}*\cdots *\Delta_{k_n,p}$, $K':=\Delta_{p-1-k_1,p}*\cdots *\Delta_{p-1-k_n,p}$, and $L:=\Delta_{p-1,p}^{*n}$.
	Then there is a $\Z/p$-equivariant inclusion $L\rightarrow K*K'$.
	Again the monotonicity and join properties of the Fadell--Husseini index imply that
	\[
	\ind_{\Z/p}L\supseteq \ind_{\Z/p}(K*K')\supseteq \ind_{\Z/p}K\cdot \ind_{\Z/p}K'.
	\]
	Furthermore $\dim L=\dim K+\dim K'+1$.
	The complexes $K$ and $K'$ are free $\Z/p$-spaces and therefore, as previously observed, it follows that
	\[
	\ind_{\Z/p}K\supseteq H^{\geq \dim K+1}(\B(\Z/p);\F_p)
	\qquad\text{and}\qquad
	\ind_{\Z/p}K'\supseteq H^{\geq \dim K'+1}(\B(\Z/p);\F_p).
	\]
	Since, by Corollary \ref{cor : index of join of chessboard p-1}, $\ind_{\Z/p}L=H^{\geq \dim L+1}(\B(\Z/p);\F_p)$ and $\dim L+1$ is an even integer, the relation between the indexes 
	\[
	\ind_{\Z/p}L\supseteq  \ind_{\Z/p}K\cdot \ind_{\Z/p}K'
	\]
	implies that
	\[
	\ind_{\Z/p}K= H^{\geq \dim K+1}(\B(\Z/p);\F_p),
	\]
	as claimed. 
	We have also proved that $\ind_{\Z/p}K'= H^{\geq \dim K'+1}(\B(\Z/p);\F_p)$.
\end{proof}

\subsection{The B\'ar\'any--Larman conjecture and the optimal colored Tverberg Theorem}
Finally we will, motivated by Theorem \ref{th : non exitence --> BL}, utilize the computation of the Fadell--Husseini index for the  chessboards to prove the following result \cite[Prop.\,4.2]{Blagojevic2009}.

\begin{theorem}
	\label{th : last 01}
	Let $d\ge1$ be an integer, and let $p$ be an odd prime.
	There is no $\Sym_p$-equivariant map 
\[
	\Delta_{p-1,p}^{*(d+1)}*[p]\rightarrow S(W_{p}^{\oplus (d+1)}).
\]
\end{theorem}
\begin{proof}
	It suffices to prove that there is no $\Z/p$-equivariant map $\Delta_{p-1,p}^{*(d+1)}*[p]\rightarrow S(W_{p}^{\oplus (d+1)})$, where $\Z/p$ is a subgroup of the symmetric group $\Sym_p$ generated by the cycle $(12\ldots p)$.
	The proof uses the monotonicity property of the Fadell--Husseini index.
	
	\smallskip
	According to \eqref{eq : index of the sphere S(W_r)} and the join property for the spheres, the index of the sphere $S(W_{p}^{\oplus (d+1)})$ is 
	\[
 	\ind_{\Z/p}S(W_p^{\oplus (d+1)}) = \langle t^{(d+1)(p-1)/2} \rangle = H^{\geq (d+1)(p-1)}(\B(\Z/p);\F_p).
 	\]
 	Using Corollary \ref{cor : index of join of chessboard p-1} we get that
 	\[
 	\ind_{\Z/p} (\Delta_{p-1,p}^{*(d+1)}*[p])=H^{\geq (d+1)(p-1)+1}(\E(\Z/p);\F_p),
 	\]
	and consequently $t^{(d+1)(p-1)/2}\notin \ind_{\Z/p} (\Delta_{p-1,p}^{*(d+1)}*[p])$.
	Thus,
	\[
	\ind_{\Z/p}S(W_p^{\oplus (d+1)})\not\subseteq   \ind_{\Z/p} (\Delta_{p-1,p}^{*(d+1)}*[p]),
	\]
	implying that a $\Z/p$-equivariant map $\Delta_{p-1,p}^{*(d+1)}*[p]\rightarrow S(W_{p}^{\oplus (d+1)})$ cannot exist. 
\end{proof}

A direct corollary of Theorems \ref{th : non exitence --> BL} and \ref{th : last 01} is that the B\'ar\'any--Larman conjecture holds for all integers $r$ such that $r+1$ is a prime \cite[Cor.\,2.3]{Blagojevic2009}.

\begin{corollary}[The B\'ar\'any--Larman conjecture for primes$-1$]
	\label{cor : BL for prime +1}
	Let $r\ge2$ and $d\ge1$ be integers such that $r+1=:p$ is a prime.
	Then $n(d,r)=(d+1)r$ and $tt(d,r)=r$.
\end{corollary}

Using the pigeonhole principle and the index computation for the chessboards we can in addition prove that in the case when $p$ is an odd prime the B\'ar\'any--Larman function $n(d,p)$ is finite.

\begin{theorem}
	Let $p$ be an odd prime.
	Then $n(d,p)\leq (d+1)(2p-2)+1$.
\end{theorem}
\begin{proof}
Let $n=	(d+1)(2p-2)+1$, and let $(C_1,\ldots,C_{d+1})$ be a coloring of the vertex set of the simplex $\Delta _{n-1}$ by $d+1$ colors with each color class of size at least $p$.
Then by the pigeonhole principle at least one of the colors, let say $C_{d+1}$, has to be of the size at least $2p-1$.
According to Corollary \ref{cor : CSTM for colored Tverberg}: If we can prove that there is no $\Sym_p$- or  $\Z/p$-equivariant map 
\[
\Delta_{|C_0|,p}*\cdots *\Delta_{|C_{d+1}|,p}\cong (R_{(C_1,\ldots,C_{d+1})})^{*p}_{\Delta(2)}\rightarrow S(W_{p}^{\oplus (d+1)}),
\]
then for every continuous map $f :  \Delta _{n-1}\rightarrow\R^d$ there are $p$ pairwise disjoint rainbow faces $\sigma_1, \dots, \sigma_r$ of $\Delta_{n-1}$  whose $f$-images overlap, that is $f(\sigma_1)\cap\cdots\cap f(\sigma_p)\neq\emptyset$.
Thus we will now prove that there is no $\Z/r$-equivariant map $\Delta_{|C_0|,p}*\cdots *\Delta_{|C_{d+1}|,p}\rightarrow S(W_{p}^{\oplus (d+1)})$.

\smallskip
Again, using \eqref{eq : index of the sphere S(W_r)} and the join property for the spheres, we have that
	\[
 	\ind_{\Z/p}S(W_p^{\oplus (d+1)}) = \langle t^{(d+1)(p-1)/2} \rangle = H^{\geq (d+1)(p-1)}(\B\Z/p;\F_p).
 	\]
Since $|C_0|\geq p,\ldots,|C_d|\geq p$ and $|C_{d+1}|\geq 2p-1$, there is a $\Z/p$-equivariant inclusion  	
 \[
 \Delta_{p-1,p}*\cdots *\Delta_{p-1,p}*\Delta_{2p-1,p}\rightarrow\Delta_{|C_0|,p}*\cdots *\Delta_{|C_{d}|,p}*\Delta_{|C_{d+1}|,p}.
 \]
Thus the monotonicity property of the Fadell--Husseini index and Corollary \ref{cor : index of join of chessboard p-1}\,(iii) imply that
  \begin{eqnarray*}
   H^{\geq (d+1)(p-1)+1}(\B\Z/p;\F_p) &=& \ind_{\Z/p}(\Delta_{p-1,p}*\cdots *\Delta_{p-1,p}*\Delta_{2p-1,p})\\
  &\supseteq &\ind_{\Z/p}(\Delta_{|C_0|,p}*\cdots *\Delta_{|C_{d}|,p}*\Delta_{|C_{d+1}|,p})	.
  \end{eqnarray*}
Therefore, 
\[
\ind_{\Z/p}S(W_p^{\oplus (d+1)})\not\subseteq \ind_{\Z/p}(\Delta_{|C_0|,p}*\cdots *\Delta_{|C_{d}|,p}*\Delta_{|C_{d+1}|,p}),
\] 
and consequently there is no $\Z/p$-equivariant map $\Delta_{|C_0|,p}*\cdots *\Delta_{|C_{d+1}|,p}\rightarrow S(W_{p}^{\oplus (d+1)})$.
This concludes the proof of the theorem.
\end{proof}

\medskip
While focusing on the B\'ar\'any--Larman conjecture and the corresponding function $n(d,r)$, we almost overlooked that the index computations for the chessboards establish a considerable strengthening of the topological Tverberg theorem that is known as the optimal colored Tverberg theorem \cite[Thm.\,2.1]{Blagojevic2009}. 

\begin{theorem}[The optimal colored Tverberg theorem]
\label{th : optimal colored Tverberg}
Let $d\ge1$ be an integer, let $p$ be a prime, $N\geq (d+1)(p-1)$, and let  $f :  \Delta_N \rightarrow \R^d$ be a continuous map.
If the vertices of the simplex $\Delta_N$ are colored by $m$ colors, where each color class has cardinality at most $p-1$, then there are $p$ pairwise disjoint rainbow faces $\sigma_1, \dots, \sigma_p$ of $\Delta_N$  whose $f$-images overlap,
\begin{equation*}
	f(\sigma_1)\cap\cdots\cap f(\sigma_p)\neq\emptyset.
\end{equation*}		
\end{theorem}

\bigskip

\section*{Dictionary}

\begin{footnotesize}
\begin{multicols}{2}

\subsection*{Borel construction}\cite{Adem2004} \cite{tomDieck:TransformationGroups} \cite{Hsiang1975}
Let $G$ be a finite group and let $X$ be a (left) $G$-space.
The \emph{Borel construction} of $X$ is the space given by $\E G\times_G X:=(\E G\times X)/G$, where $\E G$ is a free, contractible right $G$-space and $G$ acts on the product by $g\cdot(e,x)=(e\cdot g^{-1},g\cdot x)$.
The projection $\E G\times X\rightarrow\E G$ induces the following fiber bundle
\[
X\rightarrow \E G\times_G X \rightarrow \B G.
\]
This fiber bundle is called the \emph{Borel construction fiber bundle}.
The Serre spectral sequence associated to the Borel construction fiber bundle has the $E_2$-term given by
\[
E^{r,s}_2=H^r(\B G;\mathcal{H}^s(X;R))\cong H^r(G;H^s(X;R)),
\]
where the coefficients are local and determined by the action of $\pi_1(\B G)\cong G$ on the cohomology of $X$. 
Moreover, each row of the spectral sequence has the structure of an $H^*(\B G;R)$-module, while all differentials are $H^*(\B G;R)$-module morphisms.

\noindent 
The Borel construction and the associated fibration are natural with respect to equivariant maps, that is, 
any $G$-equivariant map $f :  X\rightarrow Y$ between $G$-spaces induces the following morphism of fiber bundles
\[
\xymatrix{
\E G\times_G X\ar[r]^-{\id\times_G f}\ar[d] & \E G\times_G Y\ar[d]\\
\B G \ar[r]^-{=}           & \B G.
}
\]
This morphism of fiber bundle induces a morphism of associated Serre spectral sequences
\[
E^{r,s}_t(f) :  E^{r,s}_t(\E G\times_G Y) \rightarrow E^{r,s}_t (\E G\times_G X),
\] 
such that
\[
E^{r,0}_2(f) :  E^{r,0}_2(\E G\times_G Y) \rightarrow E^{r,0}_2 (\E G\times_G X)
\]
is the identity.

\subsection*{\emph{BG}}\cite{Adem2004} \cite{Hsiang1975}
For a finite group $G$ the {\em classifying space} is the quotient space $\B G=\E G/G$.
The projection $\E G\rightarrow\B G$ is the universal principal $G$-bundle, that is, the set of all homotopy classes of maps $[X,\B G]$ is in bijection with the set of all isomorphism classes of principal $G$ bundles over $X$.

\subsection*{Borsuk--Ulam theorem}\cite{Matousek2008}
Let $S^n$ and $S^m$ be free $\Z/2$-spaces.
Then a continuous $\Z/2$-equivariant map $S^m\rightarrow S^n$ exists if and only if $m\leq n$.
 
\subsection*{Cohomology of a group (algebraic definition)}\cite{Adem2004} \cite{Brown1994}
Let $G$ be a finite group, and let $M$ be a (left) $G$-module.
Consider a projective resolution $(P_n,d_n)_{n\geq 0}$ of the trivial (left) $G$ module $\Z$, that is, an exact sequence 
\[
\xymatrix@r@C=1.45pc{
\cdots\ar[r] & P_{n+1}\ar[r]^-{d_{n}} & P_{n}\ar[r]^-{d_{n-1}} & \cdots\ar[r]^-{d_0} & P_0\ar[r]^-{\pi} & \Z\ar[r] & 0,
}
\]
where each $P_n$ is a projective (left) $G$-module.
The \emph{group cohomology} of $G$ with coefficients in the module $M$ is the cohomology of the following cochain complex
\[
\xymatrix@r@C=1.45pc{
\cdots & \ar[l]_-{d_{n}^*} \hom_{G}(P_{n},M) & \ar[l]_-{d_{n-1}^*}\cdots& \ar[l]_-{d_{0}^*} \hom_{G}(P_{0},M) & \ar[l] 0.
}
\]

\subsection*{Cohomology of group (topological definition)}\cite{Adem2004} \cite{Brown1994}
Let $G$ be a finite group, and let $M$ be a (left) $G$-module.
The \emph{group cohomology} of $G$ with coefficients in the module $M$ is the cohomology of $\B G$ with local coefficients in the $\pi_1(\B G)\cong G$-module $\mathcal{M}$, that is 
\[
H^*(G;M):=H^*(\B G;\mathcal{M}).
\]

\subsection*{Connectedness}\cite{Bredon2010} \cite{Matousek2008}
Let $n\geq -1$ be an integer.
A topological space $X$ is \emph{$n$-connected} if any continuous map $f :  S^k\rightarrow X$, where $-1\leq k\leq n$, can be continuously extended to a continuous map $g :  B^{k+1}\rightarrow X$, that is $g|_{\partial B^{k+1}=S^k}=f$.
Here $B^{k+1}$ denotes a $(k+1)$-dimensional closed ball whose boundary is the sphere $S^k$.
A topological space $X$ is $(-1)$-connected if it is non-empty; it is $0$-connected if and only if it is path-connected.
If the space $X$ is $n$-connected and $Y$ is $m$-connected, then the join $X*Y$ is $(n+m+2)$-connected.

\noindent
If the space $X$ is $n$-connected, but not $(n+1)$-connected, we write $\conn (X)=n$.
Then
\[
\conn (X*Y)\geq \conn (X) + \conn (Y) +2.
\]

\subsection*{Chessboard complex}\cite{Jonsson2008} \cite{Matousek2008}
The $m\times n$ \emph{chessboard complex} $\Delta_{m,n}$ is the simplicial complex whose vertex set is $[m]\times [n]$, and  where the set of vertices $\{(i_0,j_0),\ldots, (i_k,j_k)\}$ spans a $k$-simplex if and only if $\prod_{0\leq a<b\leq k}(i_a-i_b)(j_a-j_b)\neq 0$.
For example, $\Delta_{2,3}\cong S^1$, $\Delta_{3,4}\cong S^1\times S^1$.
The chessboard complex $\Delta_{m,n}$ is an $(\Sym_m\times \Sym_n)$-space by 
\begin{multline*}
	(\pi_1,\pi_2)\cdot \{(i_0,j_0),\ldots, (i_k,j_k)\} = \\
	\{(\pi_1(i_0),\pi_2(j_0)),\ldots, (\pi_1(i_k),\pi_2(j_k))\},
\end{multline*}
where $(\pi_1,\pi_2)\in \Sym_m\times \Sym_n$, and $\{(i_0,j_0),\ldots, (i_k,j_k)\}$ is a simplex in $\Delta_{m,n}$.
The connectivity of the chessboard complex $\Delta_{m,n}$ is
\[
\conn(\Delta_{m,n})= \min\left\{ m, n, \left\lfloor \tfrac{m+n+1}{3} \right\rfloor \right\}-2.
\]
For $n\geq3$, the chessboard complex $\Delta_{n-1,n}$ is a connected, orientable pseudomanifold of dimension $n-2$.
Therefore, ${H}_{n-2}(\Delta_{n-1,n};\Z)=\Z$ and an orientation homology class is given by the chain
\[
z_{n-1,n}\ =\ \sum_{\pi\in\Sym_{n}}(\sgn \pi) \langle(1,\pi(1)),\dots,(n-1,\pi(n-1))\rangle.
\]
The symmetric group $\Sym_n\cong 1\times \Sym_n\subseteq\Sym_{n-1}\times\Sym_n$ acts on $\Delta_{n-1,n}$ by the restriction action.
Then $\pi\cdot z_{n,n-1}=(\sgn\pi)\, z_{n-1,n-1}$.

\subsection*{Deleted join}\cite{Matousek2008}
Let $K$ be a simplicial complex, let $n\geq 2$, $k\geq 2$ be integers, and let $[n]:=\{1,\ldots,n\}$. 
The \emph{$n$-fold $k$-wise deleted join} of the simplicial complex $K$ is the simplicial complex
\begin{multline*}
K^{*n}_{\Delta(k)}:=
\{
\lambda_1x_1+\cdots+\lambda_nx_n\in \sigma_1*\cdots*\sigma_n\subset K^{*n} :\\
(\forall I\subset [n])\, \card I\geq k \Rightarrow \bigcap_{i\in I} \sigma_i=\emptyset
\},
\end{multline*}
where $\sigma_1,\ldots,\sigma_n$ are faces of $K$, including the empty face.
The symmetric group $\Sym_n$ acts on $K^{* n}_{\Delta(k)}$ by
\begin{multline*}
\pi\cdot (\lambda_1x_1+\cdots+\lambda_nx_n)	:=\\
	\lambda_{\pi^{-1}(1)}x_{\pi^{-1}(1)}+\cdots+\lambda_{\pi^{-1}(n)}x_{\pi^{-1}(n)},
\end{multline*}
for $\pi\in\Sym_n$ and $\lambda_1x_1+\cdots+\lambda_nx_n\in K^{* n}_{\Delta(k)}$. 

\subsection*{Deleted product}\cite{Matousek2008}
Let $K$ be a simplicial complex, let $n\geq 2$, $k\geq 2$ be integers, and let $[n]:=\{1,\ldots,n\}$. 
The  \emph{$n$-fold $k$-wise deleted product} of the simplicial complex $K$ is the cell complex
\begin{multline*}
K^{\times n}_{\Delta(k)}:=
\{
(x_1,\ldots,x_n)\in \sigma_1\times\cdots\times \sigma_n\subset K^{\times n} :\\
(\forall I\subset [n])\, \card I\geq k \Rightarrow \bigcap_{i\in I} \sigma_i=\emptyset
\},	
\end{multline*}
where $\sigma_1,\ldots,\sigma_n$ are non-empty faces of $K$.
The symmetric group $\Sym_n$ acts on $K^{\times n}_{\Delta(k)}$ by
\[
\pi\cdot (x_1,\ldots,x_n):=	(x_{\pi^{-1}(1)},\ldots,x_{\pi^{-1}(n)}),
\]
for $\pi\in\Sym_n$ and $(x_1,\ldots,x_n)\in K^{\times n}_{\Delta(k)}$.

\subsection*{Dold's theorem}\cite{Matousek2008}
Let $G$ be a non-trivial finite group.
For an $n$-connected $G$-space $X$ and an at most $n$-dimensional free $G$-CW complex $Y$ there is no  continuous $G$-equivariant map $X\rightarrow Y$. 

\subsection*{\emph{EG}}\cite{Adem2004} \cite{tomDieck:TransformationGroups} \cite{Hsiang1975}
For a finite group $G$ any contractible free $G$-CW complex equipped with the right $G$ cellular action is a model for an $\E G$ space.
Milnor's model is given by $\E G=\colim_{n\in\N}G^{*n}$ where $G$ stands for a $0$-dimensional free $G$-simplicial complex whose vertices are indexed by the elements of the group $G$ and the action on $G$ is given by the right translation, and $G^{*n}$ is an $n$-fold join of the $0$-dimensional simplicial complex with induced diagonal (right) action.

\subsection*{Equivariant cohomology (via the Borel construction)}
\cite{Adem2004} \cite{tomDieck:TransformationGroups} \cite{Hsiang1975}
Let $G$ be a finite group and let $X$ be a (left) $G$-space.
The singular or \v{C}ech cohomology of the Borel construction $\E G\times_G X$ of the space $X$ is called the \emph{equivariant cohomology} of $X$ and is denoted by $H_G(X;R)$.
Here $R$ denotes a group, or a ring of coefficients.

\subsection*{Equivariant cohomology of a relative $G$-CW complex}\cite{tomDieck:TransformationGroups}
Let $G$ be a finite group, let $(X,A)$ be a relative $G$-CW complex with a free action on $X{\setminus}A$, and let $C_{*}(X,A;\Z)$ denote the integral cellular chain complex. 
The cellular free $G$-action on every skeleton of $X{\setminus}A$ induces a free $G$-action on the chain complex $C_{*}(X,A;\Z)$. 
Thus $C_{*}(X,A;\Z)$ is a chain complex of free $\Z G$-modules.

\noindent
For a $\Z G$-module $M$ consider
\begin{compactitem}[\ $\bullet$]
\item the $G$-equivariant chain complex
\[
\mathcal{C}_{*}^{G}(X,A;M)=C_{*}(X,A;\Z)\otimes_{\Z G}M,
\]
and define the \emph{equivariant homology} $\mathcal{H}_{*}^{G}(X,A;M)$ of $(X,A)$ with coefficients in $M$ to be the homology of the chain complex $\mathcal{C}_{*}^{G}(X,A;M)$;

\item the $G$-equivariant cochain complex
\[
\mathcal{C}^{*}_{G}(X,A;M)=\mathrm{Hom}_{\Z G}( C_{*}(X,A;\Z),M), 
\]
and define the \emph{equivariant cohomology} $\mathcal{H}^{*}_{G}(X,A;M)$ of $(X,A)$ with coefficients in $M$ to be the cohomology of the cochain complex $\mathcal{C}^{*}_{G}(X,A;M)$.
\end{compactitem}

\subsection*{Exact obstruction sequence}\cite{tomDieck:TransformationGroups}
Let $G$ be a finite group, let $n\geq 1$ be an integer and let $Y$ be a path-connected $n$-simple $G$-space.
For every relative $G$-CW complex $(X,A)$ with a free action of $G$ on the complement $X{\setminus}A$, there exists the obstruction exact sequence
\begin{multline*}
[\sk_{n+1} X,Y]_{G}\rightarrow 
\mathrm{im}\big([\sk_{n}X,Y]_{G}\rightarrow [\sk_{n-1}X,Y]_{G}\big) \\
\overset{[\oo_{G}^{n+1}]}{\longrightarrow }\mathcal{H}_{G}^{n+1}(X,A;\pi_{n}Y),	
\end{multline*}
The sequence is natural in $X$ and $Y$. 
This should be understood as follows:
\begin{compactitem}[\ $\bullet$]
\item 
A $G$-equivariant map $f : \sk_{n-1}X\rightarrow Y$ that can be equivariantly extended to the $n$-skeleton $f' :  \sk_{n}X\rightarrow Y$, that is $f'|_{\sk_{n-1}X}=f$, defines a unique element $[\oo_{G}^{n+1}(f)]$ living in $\mathcal{H}_{G}^{n+1}(X,A;\pi _{n}Y)$, called the \emph{obstruction element} associated to the map $f$;
\item 
The exactness of the sequence means that the obstruction element $[\oo_{G}^{n+1}(f)]$ is zero if and only if there is a $G$-equivariant map $f' :  \sk_{n}X\rightarrow Y$ whose restriction is in the $G$ homotopy class of the restriction of $f$, that is $f'|_{\sk_{n-1}X}\simeq_G f|_{\sk_{n-1}X}$, which extends to the $(n+1)$-skeleton $\sk_{n+1}X$.
\end{compactitem}
The obstruction element $[\mathfrak{o}_{G}^{n+1}(f)]$ associated with the homotopy class $[f]\in [\sk_{n}X,Y]_{G}$ can be introduced on the cochain level as well. 
Let $h :  (D^{n+1},S^{n})\rightarrow (\sk_{n+1}X,\sk_{n}X)$ be an attaching map and $e\in C_{n+1}(X,A;\Z)$ the corresponding generator.
The \emph{obstruction cochain} $\oo_{G}^{n+1}(f)\in \mathcal{C}_{G}^{n+1}(X,A;\pi _{n}Y)$ of the map $f$ is
defined on $e$   by
\[
\mathfrak{o}_{G}^{n+1}(h)(e)=[f\circ h ]\in \lbrack S^{n},Y].
\]
The cohomology class of the obstruction cocycle coincides with the obstruction element defined via the exact sequence.

\subsection*{Fadell--Husseini index}\cite{Fadell1988}
Let $G$ be a finite group and $R$ be a commutative ring with unit. 
For a $G$-space $X$ and a ring $R$, the \emph{Fadell--Husseini index} of $X$ is defined to be the
kernel ideal of the map in equivariant cohomology induced by the $G$-equivariant
map $p_X  :  X\rightarrow \mathrm{pt}$:
\[
\ind_{G}(X;R)=\ker\Big(H^*(\B G;R)\rightarrow H^*(\E G\times _{G}X;R)\Big).
\]
Some basic properties of the index are:
\begin{compactitem}[\ $\bullet$]

\item \textit{Monotonicity}: If $X\rightarrow Y$ is a $G$-equivariant map then
\[
\ind_{G}(X;R) \supseteq \ind_{G}(Y;R).
\]

\item \textit{Additivity}: If $(X_1\cup X_2,X_1,X_2)$ is an excisive triple of $G$-spaces, then
\[
\ind_{G}(X_1;R)\cdot\ind_{G}(X_2;R)\subseteq\ind_{G}(X_1\cup X_2;R).
\]
\item \emph{Join:} Let $X$ and $Y$ be $G$-spaces, then
\[
\ind_{G}(X;R)\cdot \ind_{G}(Y;R)\subseteq \ind_{G}(X*Y).
\]

\item \emph{Generalized Borsuk--Ulam--Bourgin--Yang theorem:} 
Let $f :  X\rightarrow Y$ be a $G$-equivariant map, and let $Z\subseteq Y$ be a closed $G$-invariant subspace.  
Then
\[
\ind_{G}(f^{-1}(Z);R)\cdot\ind_{G}(Y{\setminus}Z;R)\subseteq \ind_{G}(X;R).
\]

\item Let $U$ and $V$ be finite dimensional real $G$-representations.
If $H^{\ast }(S(U),R)$ and $H^{\ast }(S(V),R)$ are trivial $G$-modules,
$\ind_{G}(S(U);R)=\langle f\rangle$ and $\ind_{G}(S(V);R)=\langle g\rangle$, then
\[
\ind_{G}(S(U\oplus V);R)=\langle f\cdot g\rangle \subseteq H^{\ast }(\B G;R).
\]
\end{compactitem}

\subsection*{\emph{G}-action}
Let $G$ be a group and let $X$ be a non-empty set.
A \emph{(left) $G$-action} on $X$ is a function
$G\times X\rightarrow X$, $(g,x)\longmapsto g\cdot x$
with the property that:
\[
g\cdot (h\cdot x)=(gh)\cdot x\qquad\text{and}\qquad 1\cdot x=x,
\]
for every $g,h\in G$ and $x\in X$.
A set $X$  with a $G$-action is called a $G$\emph{-set}.
Let $G$ and $X$ in addition be topological spaces.
Then a $G$-action is \emph{continuous} if the function $G\times X\rightarrow X$ is continuous with respect to the product topology on $G\times X$.
A topological space equipped with a continuous $G$-action is called a $G$\emph{-space}.

\subsection*{\emph{G}-equivariant map}
Let $X$ and $Y$ be $G$-sets (spaces). 
A (continuous) map $f :  X\rightarrow Y$ is a {\em $G$-equivariant map} if $f(g\cdot x)=g\cdot f(x)$ for all $x\in X$ and all $g\in G$.

\subsection*{\emph{G}-CW complex}\cite{Bredon1967} \cite{tomDieck:TransformationGroups}
Let $G$ be a finite group.
A CW-complex $X$ is a \emph{$G$-CW complex} if the group $G$ acts on $X$ by cellular maps and for every $g\in G$ the subspace $\{x\in X : g\cdot x=x\}$ is a CW-subcomplex of $X$.

\noindent
Let $X$ be a $G$-CW complex, and let $A$ be a subcomplex of $X$ that is invariant with respect to the action of the group $G$ and consequently a  $G$-CW complex in its own right.
The pair of $G$-CW complex $(X,A)$ is a \emph{relative $G$-CW complex}.

\subsection*{Localization theorem}\cite{tomDieck:TransformationGroups} \cite{Hsiang1975} 
The following result is a consequence of the localization theorem for elementary abelian groups:
Let $p$ be a prime, $G=(\Z/p)^n$ for $n\geq 1$, and let $X$ be a finite $G$-CW complex.
The fixed points set $X^G$ of the space $X$ is non-empty if and only the map in cohomology $H^*(\B G;\F_p)\rightarrow H^*(\E G\times_{G} X;\F_p)$, induced by the projection $\E G\times_{G} X\rightarrow\B G$, is a monomorphism.

\subsection*{\emph{n}-simple}\cite{Bredon2010}
A topological space $X$ is \emph{$n$-simple} if the fundamental group $\pi_{1}(X,x_{0})$ acts
trivially on the $n$-th homotopy group $\pi_{n}(X,x_{0})$ for every $x_{0}\in X$. 

\subsection*{Nerve of a family of subsets}
Let $X$ be a set and let $\mathcal{X}:=\{X_i : i\in I\}$ be a family of subsets of $X$.
The \emph{nerve} of the family $\mathcal{X}$ is the simplical complex $N_{\mathcal{X}}$ with the vertex set $I$, and a finite subset $\sigma\subseteq I$ is a face of the complex if and only if $\bigcap \{X_i : i\in \sigma\}\neq \emptyset$.

\subsection*{Nerve theorem}\cite{Bjorner1995}
Let $K$ be a finite simplicial complex, or a regular CW-complex, and let $\mathcal{K}:=\{ K_i : i\in I\}$ be a cover of $K$ by a family of subcomplexes, that is $K=\bigcup\{ K_i : i\in I\}$.
\begin{compactenum}[\rm (1)]
\item If for every face $\sigma$ of the nerve $N_{\mathcal{K}}$ the intersection $\bigcap \{K_i : i\in \sigma\}$ is contractible, then $K$ and $N_{\mathcal{K}}$ are homotopy equivalent, that is $K\simeq 	N_{\mathcal{K}}$.
\item If for every face $\sigma$ of the nerve $N_{\mathcal{K}}$ the intersection $\bigcap \{K_i : i\in \sigma\}$ is $(k-|\sigma|+1)$-connected, then the complex $K$ is $k$-connected if and only if the nerve $N_{\mathcal{K}}$ is $k$-connected.
\end{compactenum}

\subsection*{Primary obstruction.}\cite{Bredon1967} \cite{tomDieck:TransformationGroups} 
Let $G$ be a finite group, let $n\geq 1$ be an integer and let $Y$ be an $(n-1)$-connected and $n$-simple $G$-space.
Furthermore, let $(X,A)$ be a relative $G$-CW complex with the free $G$ action on $X{\setminus}A$, and let $f :  A\rightarrow Y$ be a $G$-equivariant map.
Then
\begin{compactitem}[\ $\bullet$]
\item 
there exists a $G$-equivariant map $f' : \sk_{n}X\rightarrow Y$ extending $f$, that is $f'|_{A}=f$,

\item 
every two $G$-equivariant extensions $f',f'' : \sk_{n}X\rightarrow Y$ of $f$ are $G$-homotopic, relative to $A$, on $\sk_{n-1}X$, that is
\[
\mathrm{im}\big( [\sk_nX,Y]_{G}\rightarrow [\sk_{n-1}X,Y]_{G}\big)=\{\mathrm{pt}\},
\]

\item 
if $H :  A\times I\rightarrow Y$ is a $G$-equivariant homotopy between $G$-equivariant maps $f :  A\rightarrow Y$ and $f' :  A\rightarrow Y$, and if $h :  \sk_{n}X\rightarrow Y$ and $h' :  \sk_{n}X\rightarrow Y$ are $G$-equivariant extensions of $f$ and $f'$, then there exists a $G$-equivariant homotopy $K : \sk_{n-1}X\times I\rightarrow Y$  between $h|_{\sk_{n-1}X}$ and $h'|_{\sk_{n-1}X}$ that extends $H$.
\end{compactitem}
In the case when $\mathrm{im}\big( [\sk_{n}X,Y]_{G}\rightarrow [\sk_{n-1}X,Y]_{G}\big)=\{\mathrm{pt}\}$ the obstruction sequence becomes
\[
[\sk_{n+1}X,Y]_{G}\rightarrow \{ \mathrm{pt}\}\overset{[\oo_{G}^{n+1}]}{\rightarrow }\mathcal{H}_{G}^{n+1}(X,\pi _{n}Y).
\]
The obstruction element $[\mathfrak{o}_{G}^{n+1}(\mathrm{pt})]\in \mathcal{H}_{G}^{n+1}(X,\pi
_{n}Y)$ is called the \emph{primary obstruction} and does not depend on the choice of a $G$-equivariant map on the $n$-th skeleton of $X$.

\subsection*{Restriction and transfer}\cite{Blagojevic2012} \cite{Brown1994}
Let $G$ be a finite group and let $H \subseteq G$ be its subgroup. 
Consider a $\Z G$-chain complex $C_*=(C_n,c_n)$ and a $\Z G$-module $M$.
Denote by $\res$ the restriction from $G$ to $H$.
For every integer $n$ there exists a homomorphism
\[
 \res  :  H^n\bigl(\hom_{\Z G}(C_*,M)\bigr) \rightarrow H^n\bigl(\hom_{\Z H}(\res C_*,\res M)\bigr)
\]
that we call the \emph{restriction} from $G$ to $H$, and a homomorphism
\[
\trf  :  H^n\bigl(\hom_{\Z H}(\res C_*,\res M)\bigr) \rightarrow H^n\bigl(\hom_{\Z G}(C_*,M)\bigr)
\]
that is called the \emph{transfer} from $H$ to $G$, with the property
\[
\trf \circ \res = [G:H] \cdot \id.
\]
 
\end{multicols}
\end{footnotesize}

\begin{small}

\providecommand{\noopsort}[1]{}
\providecommand{\bysame}{\leavevmode\hbox to3em{\hrulefill}\thinspace}
\providecommand{\MR}{\relax\ifhmode\unskip\space\fi MR }
\providecommand{\MRhref}[2]{%
  \href{http://www.ams.org/mathscinet-getitem?mr=#1}{#2}
}
\providecommand{\href}[2]{#2}

\end{small}

\end{document}